\documentclass[11pt]{article}
\usepackage{amsmath,amsthm,amssymb,amsfonts}
\usepackage{enumerate,amscd,amsxtra,MnSymbol}
\usepackage{mathrsfs}
\usepackage{color}
\usepackage[cmtip,all]{xy}

 \normalsize

\newtheoremstyle{theoremstyle}
  {10pt}      
  {5pt}       
  {\itshape}  
  {}          
  {\bfseries} 
  {:}         
  {.5em}      
  {}          

\newtheoremstyle{examplestyle}
  {10pt}      
  {5pt}       
  {}          
  {}          
  {\bfseries} 
  {:}         
  {.5em}      
  {}          

\theoremstyle{theoremstyle}
\newtheorem{theorem}{Theorem}[section]
\newtheorem*{theorem*}{Theorem}
\newtheorem{lemma}[theorem]{Lemma}
\newtheorem{proposition}[theorem]{Proposition}
\newtheorem*{proposition*}{Proposition}
\newtheorem{corollary}[theorem]{Corollary}
\newtheorem*{corollary*}{Corollary}

\newtheorem{definition}[theorem]{Definition}
\newtheorem{definition*}{Definition}
\newtheorem{remark}[theorem]{Remark}
\newtheorem{remark*}{Remark}
\newtheorem{conjecture}[theorem]{Conjecture}
\newtheorem{conjecture*}{Conjecture}

\newcommand{\bA}{{\mathbb A}}

\newcommand{\bF}{{\mathbb F}}

\newcommand{\bQ}{{\mathbb Q}}

\newcommand{\bT}{{\mathbb T}}

\newcommand{\bL}{{\mathbb L}}
\newcommand{\caR}{{\mathcal R}}
\newcommand{\caC}{{\mathcal C}}
\newcommand{\caD}{{\mathcal D}}

\newcommand{\caA}{{\mathcal A}}
\newcommand{\caE}{{\mathcal E}}

\newcommand{\caM}{{\mathcal M}}
\newcommand{\caN}{{\mathcal N}}
\newcommand{\caS}{{\mathcal S}}
\newcommand{\caO}{{\mathcal O}}
\newcommand{\caH}{{\mathcal H}}

\newcommand{\caT}{{\mathcal T}}

\newcommand{\caL}{{\mathcal L}}
\newcommand{\caW}{{\mathcal W}}
\newcommand{\caF}{{\mathcal F}}

\newcommand{\caG}{{\mathcal G}}
\newcommand{\caX}{{\mathfrak X}}

\renewcommand{\P}{{\mathbb P}}
\newcommand{\R}{{\mathbb R}}
\newcommand{\M}{{\mathcal M}}
\newcommand{\D}{{\mathsf D}}
\renewcommand{\L}{{\mathbb L}}
\newcommand{\p}{{\mathfrak p}}
\newcommand{\F}{{\mathbb F}}

\newcommand{\naturals}{{\mathbb N}}
\newcommand{\integers}{{\mathbb Z}}

\newcommand{\Ab}{\mathsf{Ab}}
\newcommand{\op}{\mathrm{op}}

\newcommand{\poauf}{[\![}
\newcommand{\pozu}{] \! ]}

\newcommand{\Ho}{\mathsf{Ho}}

\newcommand{\Hom}{\mathrm{Hom}}

\newcommand{\s}{\mathsf{s}}

\newcommand{\MGL}{\mathsf{MGL}}

\newcommand{\SH}{\mathsf{SH}}

\newcommand{\MU}{\mathsf{MU}}
\newcommand{\MZ}{\mathsf{M}\mathbb{Z}}

\newcommand{\MZm}{\mathsf{M}\mathbb{Z}/m}
\newcommand{\MFl}{\mathsf{M}\mathbb{F}_l}
\newcommand{\MR}{\mathsf{M}R}
\newcommand{\MA}{\mathsf{M}A}
\newcommand{\PMZ}{\mathsf{PM}\mathbb{Z}}
\newcommand{\MQ}{\mathsf{M}\mathbb{Q}}

\newcommand{\pr}{\mathrm{pr}}
\newcommand{\Sm}{\mathrm{Sm}}
\newcommand{\Sch}{\mathrm{Sch}}
\newcommand{\Sp}{\mathrm{Sp}}
\newcommand{\Spec}{\mathrm{Spec}}

\newcommand{\unit}{\mathbf{1}}

\newcommand{\Q}{\mathbb{Q}}

\newcommand{\Th}{\mathrm{Th}}

\newcommand{\colim}{\mathrm{colim}}

\newcommand{\Mod}{{\mathrm{Mod}}}

\newcommand{\im}{{\mathrm{im}}}
\renewcommand{\ker}{{\mathrm{ker}}}

\newcommand{\Sh}{\mathrm{Sh}}

\newcommand{\hocolim}{\mathrm{hocolim}}
\newcommand{\holim}{\mathrm{holim}}

\newcommand{\Gr}{\mathrm{Gr}}

\renewcommand{\H}{\mathcal{H}}
\newcommand{\et}{\mathit{\acute{e}t}}
\newcommand{\Zar}{\mathit{Zar}}
\newcommand{\Nis}{\mathit{Nis}}
\newcommand{\Cpx}{\mathrm{Cpx}}
\newcommand{\GmU}{\mathbb{G}_{m,U}}
\newcommand{\GmS}{\mathbb{G}_{m,S}}
\newcommand{\GmX}{\mathbb{G}_{m,X}}
\newcommand{\Gmk}{\mathbb{G}_{m,k}}
\newcommand{\Gmkappa}{\mathbb{G}_{m,\kappa}}
\newcommand{\AX}{{\mathbb{A}_X^1}}
\newcommand{\Sym}{\mathrm{Sym}}
\newcommand{\dlog}{\mathrm{dlog}}
\newcommand{\id}{\mathrm{id}}
\renewcommand{\u}{\mathrm{u}}
\renewcommand{\i}{\tilde{i}}
\renewcommand{\j}{\tilde{j}}
\newcommand{\Cat}{\mathsf{Cat}}
\newcommand{\sCat}{\mathsf{sCat}}
\newcommand{\Set}{\mathsf{Set}}
\newcommand{\sSet}{\mathsf{sSet}}
\newcommand{\fl}{\mathrm{fl}}
\newcommand{\tr}{\mathit{tr}}
\newcommand{\equi}{\mathit{equi}}
\newcommand{\map}{\mathrm{map}}
\newcommand{\DM}{\mathrm{DM}}
\newcommand{\DMT}{\mathrm{DMT}}
\newcommand{\Sect}{\mathrm{Sect}}
\newcommand{\etcart}{\mathrm{\acute{e}t-cart}}
\newcommand{\Pic}{\mathrm{Pic}}
\newcommand{\Tor}{\mathrm{Tor}}

\newcommand{\HB}{{\mathsf{H}_B}}
\newcommand{\PHB}{{\mathsf{PH}_B}}
\newcommand{\HBkappa}{{\mathsf{H}_{B,\kappa}}}
\newcommand{\HBkappap}{{\mathsf{H}_{B,\kappa(\p)}}}

\newcommand{\HBo}{{\mathsf{H}_{B,1}}}

\title{\bf A commutative $\P^1$-spectrum representing motivic cohomology over
Dedekind domains}
\author{Markus Spitzweck}
\date{\today}

\begin{document}

\maketitle


\begin{abstract}
We construct a motivic Eilenberg-MacLane spectrum with a highly structured
multiplication over smooth schemes over Dedekind domains
which represents Le\-vine's motivic cohomology. The latter is defined via
Bloch's cycle complexes.
Our method is by gluing $p$-completed and rational parts along an
arithmetic square. Hereby the finite coefficient spectra are obtained
by truncated \'etale sheaves (relying on the now proven
Bloch-Kato conjecture)
and a variant of Geisser's version of syntomic cohomology, and the rational
spectra are
the ones which represent Beilinson motivic cohomology.

As an application the arithmetic motivic cohomology groups can be realized as
Ext-groups
in a triangulated category of Tate sheaves with integral coefficients.
These can be modelled as representations of derived fundamental groups.

Our spectrum is compatible with base change giving rise
to a formalism of six functors for triangulated categories of
motivic sheaves over general base schemes including the localization triangle.

Further applications include a generalization of the Hopkins-Morel isomorphism
and a structure result for the dual motivic Steenrod algebra in the case where
the coefficient characteristic is invertible on the base scheme.
\end{abstract}

\tableofcontents

\section{Introduction}

This paper furnishes the construction of a motivic Eilenberg-MacLane
spectrum in mixed characteristic. 
Our main purpose is to construct from that spectrum triangulated categories of motivic sheaves
with integral (and thus also arbitrary) coefficients over general
base schemes which satisfy properties which are reminiscent of properties
of triangulated categories of motives which have already been constructed.
In \cite{cisinski-deglise} a theory of Beilinson motives is developed
yielding a satisfying theory of motives with rational coeffcients over general
base schemes. Voevodsky constructed triangulated categories of motives over a
(perfect) field (\cite{voevodsky.triangulated}, \cite{mazza-voevodsky-weibel})
in which motivic cohomology of smooth schemes is represented. In the Cisinski-Deglise
category over a regular base the resulting motivic cohomology are Adams-graded
pieces of rationalized $K$-theory, which fits with the envisioned theory of Beilinson.
In \cite{roendigs-oestvaer} modules over the motivic Eilenberg-MacLane spectrum over a field
are considered and it is proved that those are equivalent to Voevodsky's triangulated
categories of motives in the characteristic $0$ case. This result has recently
been generalized to perfect fields \cite{hoyois-kelly-oestvaer} where one has to invert
the characteristic of the base field in the coefficients.
\'Etale motives are developed in \cite{ayoub.real} and \cite{cisinski-deglise.etale}.

We build upon these works and construct motivic categories using motivic stable homotopy theory.
More precisely we define objects with a (coherent) multiplication in the category
of $\P^1$-spectra over base schemes and consider as in \cite{roendigs-oestvaer} their module categories. The resulting
homotopy categories are defined to be the categories of motivic sheaves.

This family of commutative ring spectra is cartesian, i.e. for any map between base schemes $X \to Y$
the pullback of the ring spectrum over $Y$ compares via an isomorphism
(in the the homotopy category of ring spectra) to the ring spectrum over $X$.
This is equivalent to saying that all spectra pull back from $\Spec(\integers)$.

To ensure good behavior of our construction our spectra have to satisfy a list
of desired properties. Over fields the spectra coincide with the usual
motivic Eilenberg-MacLane spectra (this ensures that over fields usual motivic
cohomology is represented in our categories of motivic sheaves). Rationally
we recover the theory of Beilinson motives, because the rationalizations
of our spectra are isomorphic to the respective Beilinson spectra,
and there is a relationship to Levine's motivic cohomology
defined using Bloch's cycle complexes in mixed characteristic (\cite{levine.schemes}).

To ensure all of that we first construct a spectrum over any Dedekind domain $D$
of mixed characteristic satisfying the following properties:
It represents Bloch-Levine's motivic cohomology of smooth schemes over $D$
(Corollary \ref{nrerz6u6}), it pulls back to the usual
motivic Eilenberg-MacLane spectrum with respect to maps from spectra of fields
to the spectrum of $D$ (Theorem \ref{grerhrh3}) and it is
an $E_\infty$-ring spectrum. (We remark that such an $E_\infty$-structure can always
be strictified to a strict commutative monoid in symmetric $\P^1$-spectra
by results of \cite{hornbostel.pre}.) 

The latter property makes it possible to consider the category of highly structured modules
over pullbacks of the spectrum from the terminal scheme (the spectrum of the integers),
thus defining triangulated categories of motivic sheaves $\DM(X)$ over
general base schemes $X$ such that over smooth schemes over
Dedekind domains of mixed characteristic
the Ext-groups compute Bloch-Levine's motivic cohomology (Corollary \ref{nghrrz5}).
For general base schemes we define motivic cohomology to be represented by our spectrum,
i.e. $$H^i_{\mathrm{mot}}(X,\integers(n)):= \Hom_{\SH(X)}(\unit,\Sigma^{i,n} f^* \MZ_{\Spec(\integers)})
\cong \Hom_{\DM(X)}(\integers(0),\integers(n)[i]).$$
Here $f \colon X \to \Spec(\integers)$ is the structure morphism, $\MZ_{\Spec(\integers)}$
is our spectrum over the integers and $\unit$ is the sphere spectrum (the unit with respect to the
smash product) in the stable motivic homotopy category $\SH(X)$.
By the base change property these cohomology groups coincide with Voevodsky's motivic cohomology
if $X$ is smooth over a field. We note that the ring structure on our Eilenberg-MacLane spectrum
gives the (bigraded) motivic cohomology groups a (graded commutative) ring structure,
a property which was (to the knowledge of the author) missing for Levine's motivic
cohomology.

By the work of Ayoub \cite{ayoub.I} the base change property enables one to get a full six functor formalism for these
categories of motivic sheaves including the localization triangle (Theorem \ref{ht346z4}).

We remark that the spectrum we obtain gives rise to motivic complexes over any base scheme $X$.
More precisely one can extract objects $\integers(n)^X$ in the derived category
of Zariski sheaves on the category of smooth schemes over $X$ representing our motivic cohomology.
There are unital, associative and commutative multiplication maps
$\integers(n)^X \otimes^\bL \integers(m)^X \to  \integers(n+m)^X$ inducing the multiplication
on motivic cohomology. (These multiplications are in fact part of a graded $E_\infty$-structure
(which follows from the existence of the strong periodization, see section \ref{gheu5444}),
but we do not make this explicit since
we have no application for this enhanced structure.)
If $X$ is a smooth scheme over a Dedekind domain of mixed characteristic or over a field
we have isomorphisms $\integers(0)^X \cong \underline{\integers}$ and
$\integers(1)^X \cong \caO_{/X}^*[-1]$ (for the latter isomorphism
see Theorem \ref{jr5t4565}).

One can also extract motivic Eilenberg-MacLane spaces. If $X$ is as above
there are isomorphisms $K(\integers(1),2)_X \cong \P^\infty_X$ (Proposition \ref{ghr34twe}) and
$K(\integers/n(1),1)_X \cong W_{X,n}$ (Proposition \ref{vgrehrr})
in the motivic pointed homotopy category $\caH_\bullet(X)$ of $X$.
Here $W_{X,n}$ is the total space of the line bundle $\caO_{\P_X^\infty}(-n)$ on $\P_X^\infty$ with the zero
section removed (a motivic lens space).

Among our applications is a generalization of the Hopkins-Morel ismorphism (Theorem \ref{gh543tz}),
relying on the recent work of Hoyois (\cite{hoyois.hopkins-morel}, which in turn relies
on work of Hoyois-Kelly-{\O}stv{\ae}r
\cite{hoyois-kelly-oestvaer}). In certain cases it follows that the
Eilenberg-MacLane spectrum is cellular (Corollary \ref{h5r324we}).
We obtain a description of the dual motivic Steenrod algebra over base schemes
over which the coefficient characteristic is invertible (Theorem \ref{h545egf}).
We note that one can ask if the statement of this Theorem is valid over any base scheme
(thus asking for a description of the smash product of the mod-$p$ motivic Eilenberg-MacLane
spectrum with itself in characteristic $p$).

The outline of this paper is as follows. In section \ref{45zgf} we define motivic complexes
over small sites and describe their main properties, most notably the localization sequence
due to Levine (Theorem \ref{gbter4}) and the relation to \'etale sheaves (where the Bloch-Kato conjecture
enters) (Theorem \ref{edfgth76}).

In section \ref{ju54efg} an $E_\infty$-spectrum
$\MZ$ is constructed with the main property that it represents motivic cohomology with finite coefficients
(which follows from Corollary \ref{guu43ddf})
and is rationally isomorphic to the Beilinson spectrum. 

For the definition we use an arithmetic square, i.e. we first define $p$-completed spectra for
all prime numbers $p$ and glue their product along the rationalization of this product
to the Beilinson spectrum (Definition \ref{ght544ede}).

The spectra with finite $p$-power coefficients which define the $p$-completed parts
are constructed using truncated \'etale sheaves outside characteristic $p$ and logarithmic
de Rham-Witt sheaves at characteristic $p$.

Our spectrum is constructed in the world of complexes of sheaves of abelian groups
and spectrum objects therein. By transfer of structure this also defines
($E_\infty$- or commutative ring) spectra in the world of $\P^1$-spectra in motivic spaces.

In order to prove that $\MZ$ represents integrally Bloch-Levine's motivic cohomology we define
in section \ref{h5643dfgf} a second motivic spectrum $\caM$ which by definition represents
Bloch-Levine's integral
motivic cohomology (and which will finally be isomorphic to $\MZ$).
To do that we introduce a strictification process for Bloch-Levine's cycle complexes
to get a strict presheaf on smooth schemes over a Dedekind domain. Hereby we rely heavily on
a moving Lemma due to Levine (Theorem \ref{6gfrfd}). Using a localization sequence
for the pair $(\mathbb{A}^1,\mathbb{G}_m)$ we obtain bonding maps arranging the motivic complexes
into a $\mathbb{G}_m$-spectrum (see section \ref{gtre5rdd}).
This section also contains the construction of an \'etale cycle class map
(inspired by the construction in \cite{levine.schemes}) which is compatible
with certain localization sequences (Proposition \ref{gtrf5rdd}).

After treating motivic complexes over a field (section \ref{hjrerrt}) we give our comparison
statements in section \ref{gf556tf}. First we compute the exceptional
inverse image of $\caM$ with respect to the inclusion of a closed point into our
Dedekind scheme (Theorem \ref{h47zt5e}). Theorem \ref{jf356zgf} states that
the rationalization $\caM_\Q$ is just the Beilinson spectrum.
Our main comparison statement is Theorem \ref{grferzt} which asserts
a canonical isomorphism between $\caM$ and $\MZ$ as spectra.

Our motivic Eilenberg-MacLane spectrum is strongly periodizable in the sense of
\cite{spitzweck.per} (Theorem \ref{jrfh56zr}, Remark
\ref{ht63423}). This shows that geometric mixed Tate sheaves with integral coefficients over a number
ring or similar bases which satisfy a weak version of the Beilinson-Soul{\'e} vanishing conjecture
can be modelled as representations of an affine derived group scheme along the lines of
\cite{spitzweck.fund} (Corollary \ref{redhz5eed}).

Section \ref{gf4z4z4} discusses base change. Here the Bloch-Kato
filtration on $p$-adic vanishing cycles plays a key role to obtain the part of base change
where the characteristics of the base field and of the coefficients coincide.

We treat the motivic functor formalism in section \ref{g5r3tfx}. Section \ref{g54r3grq}
contains the applications to the Hopkins-Morel isomorphism and the dual motivic Steenrod algebra.

Two appendices discuss (semi) model structures on sheaf categories and algebra objects therein
and definitions and properties of pullbacks of algebraic cycles.

We finally remark that it should be possible to generalize our strictification process in section \ref{h5643dfgf} 
to define a homotopy coniveau tower over Dedekind domains as in \cite{levine.htp}. We will come back
to this question in future work.

\vskip.7cm

{\bf Acknowledgements:} I would like to thank Joseph Ayoub for giving ideas for the stictification procedure
used in section \ref{kdtjjz} and spotting an error in an earlier version of the text, Oliver Br\"aunling
for discussions about cohomological dimensions of fields of positive characteristic (which now is used in
Proposition \ref{ge5t4ded}) and Paul Arne {\O}stv{\ae}r for having the idea of introducing the arithmetic
square in motivic homotopy theory, which now enters in Definition \ref{ght544ede}.
Furthermore I would like to thank
Peter Arndt, Mikhail Bondarko, Denis-Charles Cisinski,
David Gepner, Christian H\"asemeyer, Marc Hoyois, Moritz Kerz, Marc Levine, Jacob Lurie, Niko Naumann, Thomas Nikolaus,
Oliver R\"ondigs and Manfred Stelzer for very helpful
discussions and suggestions on the subject.

\section{Preliminaries and Notation}

For a site $\caS$ and a category $\caC$ we denote by
$\Sh(\caS,\caC)$ the category of sheaves on $\caS$ with
values in $\caC$. If $R$ is a commutative ring we set
$\Sh(\caS,R):=\Sh(\caS,\Mod_R)$, where $\Mod_R$ denotes
the category of $R$-modules.

For a Noetherian separated base scheme $S$ of finite Krull
dimension we denote by $\Sch_S$ the category of separated schemes
of finite type over $S$ and by $\Sm_S$ the full subcategory
of $\Sch_S$ of smooth schemes over $S$.

For $t \in \{\Zar,\Nis,\et \}$ we denote by $\Sm_{S,t}$ the site
$\Sm_S$ equipped with the topology $t$.

For $S$ and $t$ as above
we denote by $S_t$ the site consisting of the full
subcategory of $\Sm_S$
of $\mathrm{\acute{e}}$tale
schemes over $S$ equipped with the topology $t$.

If $m$ is invertible on $S$ we write $\integers/m(r)^S$ for
the sheaf ${\mathbb \mu}_m^{\otimes r}$ on $S_\et$.
If it is clear from the context we also write $\integers/m(r)$.

We let $\epsilon \colon \Sm_{S,\et} \to \Sm_{S,\Zar}$ and
$\epsilon \colon S_\et \to S_\Zar$
be the canonical maps of sites.

If $X$ is a presheaf of sets on $\Sm_S$ we let
$R[X]_t$ be the sheaf of $R$-modules on $\Sm_{S,t}$ freely generated by $X$.
If $Y \hookrightarrow X$ is a monomorphism we let
$R[X,Y]_t:=R[X]_t/R[Y]_t$.

For sections \ref{45zgf} through \ref{gheu5444} of the paper we fix a Dedekind domain $D$
of mixed characteristic and set $S:= \Spec(D)$.
For a prime $p$ we let $S[\frac{1}{p}]:=\Spec(D[\frac{1}{p}])$
and $Z_p \subset S$ the closed complement of $S[\frac{1}{p}]$
with the reduced scheme structure. Then $Z_p$ is a finite
union of spectra of fields of characteristic $p$.

For $S'$ the spectrum of a Dedekind domain
we let $\Sm_{S'}'$ be the full subcategory of $\Sch_{S'}$ of
schemes $X$ over $S'$ such that each connected component of $X$
is either smooth over $S'$ or smooth over a closed point of $S'$.

For an $\F_p$-scheme $Y$ we let $W_n \Omega_Y^\bullet$ be the
De Rham-Witt complex of $Y$. It is a complex of sheaves
on $Y_\et$ with a multiplication.
These complexes assemble to a complex of sheaves
on the category of all $\F_p$-schemes.
There are canonical epimorphisms
$W_{n+1} \Omega_Y^\bullet \twoheadrightarrow W_n \Omega_Y^\bullet$
respecting the multiplication.

For $Y$ as above let $\dlog \colon \caO_Y^* \to W_n \Omega_Y^1$
be defined by $x \mapsto \frac{d \underline{x}}{\underline{x}}$,
where $\underline{x}=(x,0,0,\ldots)$ is the Teichm\"uller
representative of $x$.

The logarithmic De Rham-Witt sheaf $W_n \Omega_{Y,\log}^r$
is defined to be the subsheaf
of $W_n \Omega_Y^r$ generated $\mathrm{\acute{e}}$tale
locally by sections of the form $\dlog x_1 \ldots \dlog x_r$.
Also $W_n \Omega_{Y,\log}^0$ is the constant sheaf on $\integers/p^n$.

These sheaves assemble to a subcomplex $W_n \Omega_{Y,\log}^\bullet$
of $W_n \Omega_Y^\bullet$.

The $W_n \Omega_{Y,\log}^r$ assemble to a sheaf $\nu_n^r$ on the category
of all $\F_p$-schemes.
We set $\nu_n^r=0$ for $r<0$.
There are natural epimorphisms $\nu_{n+1}^r \twoheadrightarrow \nu_n^r$.

We will also denote restrictions of $\nu_n^r$ to certain sites,
e.g. to $Y_\Zar$ or $\Sm_{k,t}$, $k$ some field of characteristic $p$,
by $\nu_n^r$.

If $\caA$ is an abelian category we denote by $\D(\caA)$ its derived category.
We denote by $\D^{\mathbb{A}^1}(\Sh(\Sm_{S,t},R))$ the $\mathbb{A}^1$-localization
of $\D(\Sh(\Sm_{S,t},R))$.

We let $\SH(S)$ be the stable motivic homotopy category and $\caH_\bullet(S)$ the pointed
$\mathbb{A}^1$-homotopy category of $S$.

We sometimes use the notation $f_*,f^*$ for a (non-derived or derived) push forward or
pullback between sheaf categories corresponding to sites
induced by a scheme morphism $f$. The precise sites which are used
can always be read off from the source and target categories.

$E_\infty$-structures are understood with respect to
(the image of) the linear isometries operad.

\section{Motivic complexes I}
\label{45zgf}

Let $S'$ be the spectrum of a Dedekind domain.
For $X \in \Sm_{S'}'$ and $r \ge 0$
we denote by $\M^X(r) \in \D(\Sh(X_\Zar,\integers))$ Levine's cycle complex.
A representative is the complex with $z^r(\_,2r-i)$
in cohomological degree $i$, see \cite[\S 3]{geisser.dede},
\cite{levine.techniques}. For $r<0$ we set $\M^X(r)=0$.
When it is clear from the context
which $X$ is meant we also write $\M(r)$.
We also write $\M_\et^X(r)$ for $\epsilon^* \M^X(r)$
and $\M^X(r)/m$ for $M^X(r) \otimes^\L \integers/m$.

\begin{theorem}
\label{gbter4}
(Levine)
Let $i \colon Z \to X$ be a closed inclusion in $\Sm_{S'}'$
of codimension $c$
and $j \colon U \to X$ the complementary open inclusion.
Then there is an exact triangle in $\D(\Sh(X_\Zar,\integers))$
\begin{equation}
\label{4t56uzg}
\R i_* \M^Z(r-c)[-2c] \to \M^X(r) \to \R j_* \M^U(r) \to 
\R i_* \M^Z(r-c)[-2c+1].
\end{equation}
\end{theorem}

\begin{proof}
This is \cite[Theorem 1.7]{levine.techniques}.
\end{proof}

\begin{corollary}
Let $i \colon Z \to X$ be a closed inclusion in $\Sm_{S'}'$
of codimension $c$. Then there is a canonical ismorphism
$$\R i^! \M^X(r) \cong \M^Z(r-c)[-2c]$$
in $\D(\Sh(X_\Zar,\integers))$.
\end{corollary}

\begin{theorem}
\label{gtfn}
For $X \in \Sm_{S'}'$ we have $\H^k(\M^X(r))=0$ for $k>r$.
\end{theorem}
\begin{proof}
This is \cite[Corollary 4.4]{geisser.dede}.
\end{proof}

\begin{theorem}
\label{fbmtzzt}
Suppose $X \in \Sm_{S'}'$ is of characteristic $p$.
Then there is an isomorphism
\begin{equation}
\label{hgfd}
\M^X(r)/p^n \cong \nu_n^r[-r]
\end{equation}
in $\D(\Sh(X_\Zar,\integers/p^n))$.
\end{theorem}
\begin{proof}
If $X$ is smooth over a perfect field this is \cite[Theorem 8.3]{geisser-levine.p}.
The general case follows by a colimit argument (using \cite[I. (1.10.1)]{illusie.complexe}).
\end{proof}

\begin{corollary}
\label{n4t6d}
Let $p$ be a prime, $X \in \Sm_S$ and $\pi \colon X \to S$ the structure
morphism. Let $i \colon Z:=\pi^{-1}(Z_p) \to X$ be the closed
and $j \colon U:=\pi^{-1}(S[\frac{1}{p}]) \to X$ be the open inclusion.
Then $\H^k(\R j_* \M^U(r) \otimes^\L \integers/p^n)=0$
for $k >r$ and the natural map
\begin{equation}
\label{nbgfe}
\H^r(\R j_* (\M^U(r)/p^n))
\to i_*\nu_n^{r-1}
\end{equation}
induced by the
triangle (\ref{4t56uzg}) and the isomorphism (\ref{hgfd})
is an epimorphism.
\end{corollary}
\begin{proof}
This follows from Theorem \ref{gtfn}, the exactness of $i_*$
and the long exact sequence of cohomology sheaves
induced by the exact triangle (\ref{4t56uzg}).
\end{proof}

\begin{lemma}
Suppose $X \in \Sm_{S'}'$ is of characteristic $p$.
Then the diagram
$$\xymatrix{\M^X(r)/p^{n+1} \ar[r]^(.6)\cong \ar[d] &
\nu_{n+1}^r[-r] \ar[d] \\
\M^X(r)/p^n \ar[r]^(.6)\cong &
\nu_n^r[-r]}$$
in $\D(\Sh(X_\Zar,\integers/p^{n+1}))$ commutes.
\end{lemma}

Suppose $m$ is invertible on $X \in \Sm_{S'}'$.
Then there is a cycle class map $$\M^X(r)/m
\to \R \epsilon_* \integers/m(r).$$ For a definition see the proof
of Theorem \ref{gtt5zjh}.
The {\'e}tale sheafification of the cycle
class map is an isomorphism in $\D(\Sh(X_\et,\integers/m))$,
see \cite[Theorem 1.2. 4.]{geisser.dede}.

Let $f \colon Y \to X$ be a flat morphism of schemes for which the
motivic cycle complexes are defined. Then there is a flat
pullback $f^*\M^X(r) \to \M^Y(r)$.

\begin{lemma}
\label{gbhtfrr}
Let $f \colon Y \to X$ be a flat morphism of schemes for which the
motivic cycle complexes are defined. Suppose $m$ is invertible on $X$.
Then the diagram
$$\xymatrix{f^* \M^X(r)/m \ar[r] \ar[d] &
f^* \R \epsilon_* \integers/m(r) \ar[d] \\
\M^Y(r)/m \ar[r] &
\R \epsilon_* \integers/m(r)}$$
commutes.
\end{lemma}

\begin{proof}
This follows from the definition of the {\'e}tale cycle class map.
\end{proof}

\begin{lemma}
\label{nhgrrh}
Let $X \in \Sm_{S'}'$ and suppose $m$ is invertible on $X$.
Let $m' |m$.
Then the diagram
$$\xymatrix{\M^X(r)/m \ar[r] \ar[d] &
\R \epsilon_* \integers/m(r) \ar[d] \\
\M^X(r)/m' \ar[r] &
\R \epsilon_* \integers/m'(r)}$$
in $\D(\Sh(X_\Zar,\integers/m))$ commutes.
\end{lemma}

\begin{proof}
This follows from the definition of the {\'e}tale cycle class map.
\end{proof}

\begin{theorem}
\label{edfgth76}
Let $X \in \Sm_{S'}'$ and suppose $m$ is invertible on $X$.
Then there is an isomorphism
$$\M^X(r)/m \cong
\tau_{\le r}(\R \epsilon_* \integers/m(r))$$
in $\D(\Sh(X_\Zar,\integers/m))$ induced by the cycle
class map.
\end{theorem}
\begin{proof}
By \cite[Theorem 1.2. 2.]{geisser.dede}
(which we can apply since the Bloch-Kato conjecture is proven,
\cite{voevodsky}) we have
$$\M^X(r) \cong \tau_{\le(r+1)} \R \epsilon_* \M_\et^X(r).$$
By Theorem \ref{gtfn} it follows that
$\R^{r+1} \epsilon_* \M_\et^X(r)=0$. Thus
$$\M^X(r)/m \cong \tau_{\le r}
(\R \epsilon_* \M_\et^X(r)/m).$$
But by \cite[Theorem 1.2. 4.]{geisser.dede} we have
$$\M_\et^X(r)/m \cong \integers/m(r)$$
induced by the cycle class map (see the proof of
\cite[Theorem 1.2. 4.]{geisser.dede}).
This shows the claim.
\end{proof}

\begin{theorem}
\label{gfre5r5}
Let $i \colon Z \to X$ be a closed inclusion in $\Sm_{S'}'$ of codimension $c$
and suppose $m$ is invertible on $X$.
Then there is a canonical isomorphism
$$\R i^! \integers/m(r) \cong \integers/m(r-c)[-2c]$$
in $\D(\Sh(Z_\et,\integers/m))$.
\end{theorem}
\begin{proof}
This is contained in \cite{riou.gysin}.
\end{proof}

A consequence is the localization/Gysin exact triangle for
$\mathrm{\acute{e}}$tale cohomology.

\begin{corollary}
\label{grtehh}
Let $i \colon Z \to X$ be a closed inclusion in $\Sm_{S'}'$ of codimension $c$
and $j \colon U \to X$ the complementary open inclusion. Suppose $m$ is
invertible on $X$.
Then there is an exact triangle
$$i_* \integers/m(r-c)[-2c] \to \integers/m(r) \to \R j_* \integers/m(r)
\to i_* \integers/m(r-c)[-2c+1]$$
in $\D(\Sh(X_\et,\integers/m))$.
\end{corollary}

\begin{proof}
This follows from Theorem \ref{gfre5r5} and the corresponding
exact triangle involving $\R i^! \integers/m(r)$.
\end{proof}

\begin{theorem}
\label{gtt5zjh}
Let $i \colon Z \to X$ be a closed inclusion in $\Sm_{S'}'$ of codimension $c$ and suppose $m$
is invertible on $X$.
Then the diagram
$$\xymatrix{\R i^! \M^X(r)/m \ar[r]^(.44)\cong \ar[dd] &
\M^Z(r-c)/m[-2c] \ar[ddd] \\
& \\
\R i^! \R \epsilon_* \integers/m(r) \ar[d]^\cong & \\
\R \epsilon_* \R i^! \integers/m(r) \ar[r]^(.44)\cong &
\R \epsilon_* \integers/m(r-c)[-2c]}$$
commutes.
\end{theorem}
\begin{proof}
Let $U=X \setminus Z$. 
For $V \in \Sm_{S'}'$
we denote by $c^r(V,n)$ the set of cycles (closed integral subschemes) of $V \times \Delta^n$ which
intersect all $V \times Y$ with $Y$ a face of $\Delta^n$ properly.

Let $\mu_m^{\otimes r} \to \caG$ be an injectively fibrant
replacement in $\Cpx(\Sh(\Sm_{X,\et},\integers/m))$.

Let $V \in \Sm_X$. For $W$ a closed subset of $X$ such that each irreducible component has codimension greater
or equal to $r$ set $\caG^W(V):= \ker(\caG(V) \to \caG(V \setminus W))$.

As in \cite[12.3]{levine.schemes} there is a canonical isomorphism of $H^{2r}(\caG^W(V))$ with the free $\integers/m$-module
on the irreducible components of $W$ of codimension $r$ and the map $\tau_{\le 2r} \caG^W(V) \to H^{2r}(\caG^W(V))[-2r]$
is a quasi isomorphism.

For $V \in X_\et$ denote by
$\caG^r(V,n)$ the colimit of the $\caG^W(V \times \Delta^n)$ where $W$ runs through the finite unions of
elements of $c^r(V,n)$. The simplicial complex of $\integers/m$-modules $\tau_{\le 2r}\caG^r(V,\bullet)$ augments
to the simplicial abelian group $z^r(V,\bullet)/m[-2r]$. This augmentation is a levelwise quasi isomorphism.
We denote by $\caG^r(V)$ the total complex associated to the double complex which is the normalized complex
associated to $\tau_{\le 2r}\caG^r(V,\bullet)$. Thus we get a quasi isomorphism
$\caG^r(X) \to z^r(X)/m[-2r]$. Here for $V \in \Sm_{S'}'$ the complex $z^r(V)$ is defined to be the normalized
complex associated to the simplicial abelian group $z^r(V,\bullet)$.

On the other hand for $V \in X_\et$ we have a canonical map $\caG^r(V,n) \to \caG(V \times \Delta^n)$ compatible with
the simplicial structure. We denote by $\caG'(V)$ the total complex associated to the double complex
which is the normalized complex associated to $\caG(V \times \Delta^\bullet)$.
We have a canonical quasi isomorphism $\caG(V) \to \caG'(V)$ and a canonical map
$\caG^r(V) \to \caG'(V)$. The above groups and maps are functorial in $V \in X_\et$.

Thus we get a map $$z^r(\_)/m[-2r] \cong \caG^r \to \caG' \cong \caG$$
in $\D(\Sh(X_\et,\integers/m))$.
This is (the adjoint of) the cycle class map.

Denote by $\tilde{\caG}$, $\tilde{\caG}'$, $\tilde{\caG}^{r-c}$ the analogous objects defined for $Z$ instead for $X$,
so we have a diagram $$z^{r-c}(\_)/m[-2(r-c)] \overset{\sim}{\leftarrow} \tilde{\caG}^{r-c} \to \tilde{\caG}' \overset{\sim}{\leftarrow}
\tilde{\caG}$$ in $\Cpx(\Sh(Z_\et,\integers/m))$.

For $V \in \Sm_X$ set $\caG_Z(V):=\ker(\caG(V) \to \caG(V|_U))$. Thus $\caG_Z \in \Cpx(\Sh(\Sm_{X,\et},\integers/m))$
computes $i_*\R i^! \mu_m^{\otimes r}$.

There is an absolute purity isomorphism $\caG_Z \cong i_*\tilde{\caG}[-2c]$ in $\D(\Sh(\Sm_{X,\et},\integers/m))$.
Choose a representative $\varphi \colon \caG_Z \to i_*\tilde{\caG}[-2c]$ in $\Cpx(\Sh(\Sm_{X,\et},\integers/m))$ of this
isomorphism. This exists since $i_*\tilde{\caG}[-2c]$ is injectively fibrant.

For $V \in X_\et$ denote by $\caG_Z'(V)$ the total complex associated to the double complex which is the normalized complex associated
to $\caG_Z(V \times \Delta^\bullet)$. Moreover let $\caG_Z^r(V,n)$ be the colimit of the $\caG^W(V \times \Delta^n)$ where $W$
runs through the finite unions of elements of $c^{r-c}(V|_Z,n)$.
Denote by $\caG_Z^r(V)$ the total complex associated to the double complex which is the normalized complex
associated to $\tau_{\le 2r}\caG_Z^r(V,\bullet)$. Denote by $z_Z^r(V)$ the complex $z^{r-c}(V|_Z)$. 

Set $\caG_U(V):=\caG(V|_U)$, $\caG_U'(V):=\caG'(V|_U)$, $\caG_U^r(V):=\caG^r(V|_U)$ and $z_U^r(V):=z^r(V|_U)$.

We have the diagram

$$\xymatrix{i_*\tilde{\caG}[-2c] \ar[d]^\sim & \caG_Z \ar[l]_\sim \ar[r] \ar[d]^\sim & \caG \ar[r] \ar[d]^\sim & \caG_U \ar[d]^\sim \\
i_*\tilde{\caG}'[-2c] & \caG_Z' \ar[l]_\sim \ar[r] & \caG' \ar[r] & \caG_U' \\
i_*\tilde{\caG}^{r-c}[-2c] \ar[d]^\sim \ar[u] & \caG_Z^r \ar[u] \ar[l]_\sim \ar[r] \ar[d]^\sim & \caG^r \ar[u] \ar[r] \ar[d]^\sim & \caG_U^r \ar[d]^\sim \ar[u] \\
i_* z^{r-c}(\_)[-2r] & z_Z^r(\_)[-2r] \ar[l]_\cong \ar[r] & z^r(\_)[-2r] \ar[r] & z_U^r(\_)[-2r].
}$$

The upper three left most horizontal maps are induced by $\varphi$.
The lower left square commutes by the naturality of the purity maps in {\'e}tale cohomology. All other squares
commute by construction.
The last two arrows in each horizontal line compose to $0$ and constitute an exact triangle, thus
the second vertical line computes $i_*\R i^!$ of the third vertical line.
The claim follows.

\end{proof}

\begin{corollary}
\label{grer54s}
Let $i \colon Z \to X$ be a closed inclusion in $\Sm_{S'}'$ of codimension $c$
and $j \colon U \to X$ the complementary open inclusion.
Suppose $m$ is invertible on $X$.
Then the diagram
$$\xymatrix{i_* \M^Z(r-c)/m[-2c] \ar[r] \ar[dd] &
\M^X(r)/m \ar[ddd] \ar[r] &
\R j_* \M^U(r)/m \ar[r] \ar[dd] &
i_* \M^Z(r-c)/m[-2c+1] \ar[dd] \\
& & & \\
i_* \R \epsilon_* \integers/m(r-c)[-2c] \ar[d]^\cong & &
\R j_* \R \epsilon_* \integers/m(r) \ar[d]^\cong &
i_* \R \epsilon_* \integers/m(r-c)[-2c+1] \ar[d]^\cong \\
\R \epsilon_* i_* \integers/m(r-c)[-2c] \ar[r] &
\R \epsilon_* \integers/m(r) \ar[r] &
\R \epsilon_* \R j_* \integers/m(r) \ar[r] &
\R \epsilon_* i_* \integers/m(r-c)[-2c+1]}$$
commutes.
\end{corollary}

\begin{proof}
The diagram
$$\xymatrix{i_* \R i^! \M^X(r)/m \ar[r] \ar[dd] &
\M^X(r)/m \ar[dddd] \ar[r] &
\R j_* \M^U(r)/m \ar[ddd] \ar[r] &
i_* \R i^! \M^X(r)/m[1] \ar[dd] \\
& & & \\
i_* \R i^! \R \epsilon_* \integers/m \ar[d]^\cong & & & 
i_* \R i^! \R \epsilon_* \integers/m[1] \ar[d]^\cong \\
i_* \R \epsilon_* \R i^! \integers/m(r) \ar[d]^\cong & &
\R j_* \R \epsilon_* \integers/m(r) \ar[d]^\cong &
i_* \R \epsilon_* \R i^! \integers/m(r)[1] \ar[d]^\cong \\
\R \epsilon_* i_* \R i^! \integers/m(r) \ar[r] &
\R \epsilon_* \integers/m(r) \ar[r] &
\R \epsilon_* \R j_* \integers/m(r) \ar[r] &
\R \epsilon_* i_* \R i^! \integers/m(r)[1]}$$
commutes. Thus the claim follows from
Theorem \ref{gtt5zjh}.
\end{proof}

\begin{theorem}
\label{htehg}
Let $X \in \Sm_{S'}'$. Let $q \colon \AX \to X$ be the
projection. Then the canonical map
$$\M^X(r) \to \R q_* \M^\AX(r)$$\\
is an isomorphism in $\D(\Sh(X_\Zar,\integers))$.
\end{theorem}

\begin{proof}
This is \cite[Corollary 3.5]{geisser.dede}.
\end{proof}

\section{The construction}
\label{ju54efg}

\subsection{The $p$-parts}

\subsubsection{Finite coefficients}
\label{htr46u}

We fix a prime $p$ and set $U:=S[\frac{1}{p}]$, $Z:=Z_p$,
$i \colon Z \hookrightarrow S$ the closed and $j \colon
U \hookrightarrow S$ the open inclusion.

For a scheme $X$ for which the motivic complexes are defined
we set $\M_n^X(r):= \M^X(r)/p^n$.

For $n \ge 1$ and $r \in \integers$ let $L_n(r):= {\mathbb\mu}_{p^n}^{\otimes
r}$
viewed as sheaf of $\integers/p^n$-modules on $\Sm_{U,\et}$.

The pullback $j^{-1} \colon \Sm_S \to \Sm_U$, $X \mapsto X \times_S U$,
induces a push forward $$j_* \colon \Sh(\Sm_{U,\Zar},\integers/p^n) \to
\Sh(\Sm_{S,\Zar},\integers/p^n)$$ (we suppress the dependence on $n$
of the functor $j_*$). The same is true for $\mathrm{\acute{e}}$tale
sheaves.

Similarly, we have the pullback $i^{-1} \colon \Sm_S \to \Sm_Z$,
$X \mapsto X \times_S Z$, inducing also a push forward on
sheaf categories.

Let $QL_n(1) \to L_n(1)$ be a cofibrant replacement
in $\Cpx(\Sh(\Sm_{U,\et},\integers/p^n))$ (the latter category
is equipped with the local projective model structure, see Appendix \ref{ht34efw})
and let $QL_n(1) \to RQL_n(1)$ be a fibrant
replacement via a cofibration. Thus $\caT:=RQL_n(1)[1]$ is both fibrant
and cofibrant.

Recall the decomposition
\begin{equation}
\label{bfedhed}
\underline{\R \Hom}_{\D(\Sh(\Sm_{U,\et},\integers/p^n))}
(\GmU,L_n(1)[1])=L_n(1)[1] \oplus L_n(0).
\end{equation}

The first summand splits off because the projection
$\GmU \to U$ has the section $\{1\}$.

To define the isomorphism of the remaining summand with $L_n(0)$ we use
the Gysin sequence for the situation
$$\GmU \hookrightarrow \mathbb{A}_S^1 \hookleftarrow \{0\}.$$

Let $\iota \colon \integers/p^n[\GmU,\{1\}]_\et \to \caT$
be a map which classifies the canonical element $1 \in H_\et^1(\GmU,L_n(1))$
under the above decomposition (here the source of $\iota$ is the chain complex having the indicated
object in degree $0$).
Note that $\integers/p^n[\GmU,\{1\}]_\et$ is cofibrant.

\begin{remark}
The map $\integers/p^n[\GmU]_\et \to \caT$ induced by $\iota$ represents the map
induced by the last map of the exact triangle
$$L_n(1) \to \GmU \overset{p^n}{\to} \GmU \to L_n(1)[1]$$
in $\D(\Sh(\Sm_{U,\et},\integers))$.
This follows from the construction of the Gysin isomorphism.
\end{remark}

We get a map
$$\Sym(\iota) \colon \Sym(\integers/p^n[\GmU,\{1\}]_\et) \to \Sym(\caT)$$
of commutative monoids in symmetric sequences in 
$\Cpx(\Sh(\Sm_{U,\et},\integers/p^n))$, in other words
$\Sym(\caT)$ is a commutative monoid in the category of symmetric
$\integers/p^n[\GmU,\{1\}]_\et$-spectra $\Sp_{\integers/p^n[\GmU,\{1\}]_\et}^\Sigma$.
In particular it gives rise to an $E_\infty$-object in
$\Sp_{\integers/p^n[\GmU,\{1\}]_\et}^\Sigma$.

Let $Q \Sym(\caT) \to \Sym(\caT)$ be a cofibrant replacement via a trivial fibration and
$Q \Sym(\caT) \to R Q \Sym(\caT)$
a fibrant resolution of $Q\Sym(\caT)$ in $E_\infty(\Sp_{\integers/p^n[\GmU,\{1\}]_\et}^\Sigma)$
(here $E_\infty(\Sp_{\integers/p^n[\GmU,\{1\}]_\et}^\Sigma)$ is equipped with the transferred
semi model structure, see Appendix \ref{ht34efw}, in particular $R Q \Sym(\caT)$ is underlying
levelwise fibrant for the local projective model structure
and is therefore suitable to compute the derived push forward
along $\epsilon$).

\begin{lemma}
The map $Q \Sym(\caT) \to R Q \Sym(\caT)$ is a level equivalence,
i.e. $\Sym(\caT)$ is an $\Omega$-spectrum.
\end{lemma}

\begin{proof}
This follows from the fact that we have chosen the map $\iota$
in such a way that the derived adjoints of the structure maps
of $\Sym(\caT)$ give rise to the isomorphism
$\underline{\mathbb{R}\Hom}((\GmU,\{1\}),L_n(r)[r]) \simeq L_n(r-1)[r-1]$.
\end{proof}

Set $A:=\epsilon_*(R Q \Sym(\caT))$, so the spectrum $A$ is
$R Q \Sym(\caT)$ viewed as $E_\infty$-algebra in
$\integers/p^n[\GmU,\{1\}]_\Zar$-spectra
in $\Cpx(\Sh(\Sm_{U,\Zar},\integers/p^n))$.

We denote by $A_r$ the $r$-th level of $A$. 
Thus $A_r \simeq \R \epsilon_* L_n(r)[r]$.

Set $A_r':= \tau_{\le 0}(A_r)$, where $\tau_{\le 0}$ denotes
the good truncation at degree $0$, i.e. the complex $A_r'$
equals $A_r$ in (cohomological) degrees $<0$, consists
of the cycles in degree $0$ and is $0$ in positive degree.

Thus by Theorem \ref{edfgth76} there is for every $X \in \Sm_U$ an isomorphism
\begin{equation}
\label{4j6n}
A_r'|_{X_\Zar} \cong \M_n^X(r)[r]
\end{equation}
in  $\D(\Sh(X_\Zar,\integers/p^n))$,
where $A_r'|_{X_\Zar}$ denotes the restriction of $A_r'$ to
$X_\Zar$.

\begin{lemma}
\label{htrgrzj}
The complexes $A_r'$ assemble to a $\integers/p^n[\GmU,\{1\}]_\Zar$-spectrum
$A'$. This spectrum is equipped with an $E_\infty$-structure
together with a map of $E_\infty$-algebras $A' \to A$ which is
levelwise the canonical map $A_r' \to A_r$.
\end{lemma}

\begin{proof}
This follows from the fact that the truncation $\tau_{\le 0}$
is right adjoint to the symmetric monoidal inclusion of
(cohomologically) non-positively graded complexes into all complexes
and that $\integers/p^n[\GmU,\{1\}]_\Zar$ lies in this subcategory
of non-positively graded complexes.
\end{proof}

Let $Q A' \to A'$ be a cofibrant replacement via a trivial fibration
in $\integers/p^n[\GmU,\{1\}]_\Zar$-spectra in $\Cpx^{\le 0}(\Sh(\Sm_{U,\Zar},\integers/p^n))$
(so $Q A'$ is also cofibrant viewed as spectrum in unbounded complexes) and
$Q A' \to RQA'$ be a fibrant resolution (as $E_\infty$-algebras
in $\integers/p^n[\GmU,\{1\}]_\Zar$-spectra in $\Cpx(\Sh(\Sm_{U,\Zar},\integers/p^n))$).

\begin{proposition}
\label{bgerhh4}
The map $QA' \to RQA'$ is a level equivalence, i.e.
$A'$ and $QA'$ are $\Omega$-spectra.
\end{proposition}

\begin{proof}
Set $m:=p^n$.
Let $X \in \Sm_U$. Let $\i \colon \{0\} \to \AX$
be the closed, $\j \colon \GmX \to \AX$ the open
inclusion and $q \colon \AX \to X$ the projection.
By Corollary \ref{grtehh} we have an exact triangle
$$\i_* \integers/m(r-1)[-2] \to \integers/m(r)^\AX \to
\R \j_* \integers/m(r) \to \i_* \integers/m(r-1)[-1].$$
Note $$\R q_* \R \j_* \integers/m(r) \cong
\underline{\R \Hom}_{\D(\Sh(\Sm_{U,\et},\integers/m))}
(\GmU,L_n(r))|_{X_\et}$$
and that $\R q_*$ applied to the last map
in the sequence gives the projection to the second summand
in our decomposition (\ref{bfedhed}).
Thus by construction of the map $\iota$ this map
also gives the inverse of the adjoint of the structure map in $R \Sym(\caT)$.

By Theorem \ref{4t56uzg} there is an exact triangle
$$\i _* \M_n(r-1)[-2] \to \M_n^\AX(r) \to \R \j_* \M_n(r) \to
\i_* \M_n(r-1)[-1].$$

Hence by Theorem \ref{gtfn} the canonical map
$$\tau_{\le r}(\R \j_* \M_n(r)) \to \R \j_* \M_n(r)$$
is an ismorphism. Thus in view of Theorem \ref{edfgth76}
the same truncation property holds for
$\R \j_* \tau_{\le r}\R \epsilon_* \integers/m(r)$.
Thus the map
$$\R \j_* \tau_{\le r}\R \epsilon_* \integers/m(r)
\to \R \epsilon_* \i_* \integers/m(r-1)[-1]$$
factors through $\tau_{\le r}(\R \epsilon_* \i_* \integers/m(r-1)[-1])$.

Moreover the map
$$\R \j_* \tau_{\le r}\R \epsilon_* \integers/m(r)
\cong \tau_{\le r} \R \j_* \tau_{\le r}\R \epsilon_* \integers/m(r) \to
\tau_{\le r} \R \j_* \R \epsilon_* \integers/m(r)$$
is an isomorphism, thus we have a canonical map
$$\tau_{\le r} \R \epsilon_* \integers/m(r)^\AX \to
\R \j_* \tau_{\le r}\R \epsilon_* \integers/m(r).$$
Using Corollary \ref{grer54s} these maps fit into the commutative diagram
\begin{equation}
\label{edrghj}
\xymatrix{\M_n^\AX(r) \ar[r] \ar[d]^\cong &
\R \j_* \M_n(r) \ar[r] \ar[d]^\cong &
\i_* \M_n(r-1)[-1] \ar[d]^\cong \\
\tau_{\le r} \R \epsilon_* \integers/m(r)^\AX \ar[r] &
\R \j_* \tau_{\le r} \R \epsilon_* \integers/m(r) \ar[r] &
\tau_{\le r}(\i_* \R \epsilon_* \integers/m(r-1)[-1]),}
\end{equation}
where the top row is part of the triangle given by
Theorem \ref{gbter4}.
The composition
$$A_{r-1}'[-r]|_{X_\Zar} \to \underline{\R \Hom}(\integers/m[\GmU,\{1\}]_\Zar,
A_r'[-r])|_{X_\Zar}$$
$$\to \underline{\R \Hom}(\integers/m[\GmU]_\Zar,
\tau_{\le r} \R \epsilon_* L_n(r))|_{X_\Zar}$$
$$\cong \R q_* \R \j_* \tau_{\le r} \R \epsilon_* \integers/m(r) \to
\tau_{\le r}(\R \epsilon_* \integers/m(r-1)[-1]) \cong A_{r-1}'[-r]|_{X_\Zar}$$
is the identity.

By Theorem \ref{htehg} $\R q_* \M_n^\AX(r)$ identifies
with $\M_n^X(r)$, thus $\R q_*$ applied to the left
bottom arrow in (\ref{edrghj}) is an isomorphism
to the trivial summand and $\R q_*$ of the bottom
row splits. Thus also $\R q_*$ of the top
row splits. This shows that in fact
$$\R q_* \i_* \M_n(r-1)[-1] \cong \M_n^X(r-1)[-1]$$
is via the right vertical isomorphism and the
right lower map in the diagram
isomorphic to the non-trivial summand in
$\R q_* \R \j_* \tau_{\le r} \R \epsilon_* \integers/m(r)$.
Since this holds over every $X \in \Sm_U$ we are done.
\end{proof}

Thus $B:=j_*(RQA')$ is a $\integers/p^n[\GmS,\{1\}]_\Zar$-spectrum
and computes also levelwise the derived push forward of $A'$ along $j$.
(Note that to compute the levelwise push forward we also could have
used the levelwise model structure.)

By (\ref{4j6n}) for every $X \in \Sm_S$
we have 
\begin{equation}
\label{rthr}
B_r|_{X_\Zar} \cong
\R (j_X)_* (\M_n^{X_U}(r))[r]
\end{equation}
in $\D(\Sh(X_\Zar,\integers/p^n))$
(here $X_U=X \times_S U$ and $j_X$ denotes the inclusion
$X_U \hookrightarrow X$).

Thus by Corollary \ref{n4t6d} the map
$B_r':=\tau_{\le 0} B_r \to B_r$ is a quasi-isomorphism.

As in Lemma \ref{htrgrzj} the $B_r'$ assemble to an
$E_\infty$-algebra $B'$, and the natural map
$B' \to B$ is an equivalence.

By the following Lemma we could have used $j_*A'$ instead of
$B$ and $B'$.

\begin{lemma}
\label{jh44444}
The natural maps $j_*(QA') \to j_* A'$ and $j_*(QA') \to B$ are level equivalences.
\end{lemma}

\begin{proof}
Note first that each $A_r'$ is fibrant in
$\Cpx^{\le 0}(\Sh(\Sm_{U,\Zar},\integers/p^n))$, hence so are the $(QA')_r$, thus
the $j_*A'$ and $j_*(QA')$ are the derived push forwards to the homotopy category of
$\Cpx^{\le 0}(\Sh(\Sm_{S,\Zar},\integers/p^n))$.
But truncation commutes with derived push forward
(both are right adjoints), so the claim follows from the fact
that $B' \to B$ is an equivalence.
\end{proof}

\begin{corollary}
\label{h64de}
There is a natural isomorphism
$$\R^r j_* L_n(r) \cong \epsilon^*\H^0(B_r')=\H^0(B_r')_\et$$
in $\Sh(\Sm_{S,\et},\integers/p^n)$.
\end{corollary}

\begin{proof}
We have $j_*A'=\epsilon_* \tau_{\le 0} j_* (R Q \Sym(\caT))$,
thus $$\R^r j_* L_n(r) \cong \H^0(j_* ((R Q \Sym(\caT))_r)
=\H^0(\tau_{\le 0} j_* ((R Q \Sym(\caT))_r)$$
$$=\epsilon^* \H^0(\epsilon_* \tau_{\le 0} j_* ((R Q \Sym(\caT))_r)
=\epsilon^* \H^0(j_*A_r') = \epsilon^* \H^0(B_r').$$
At the end we used Lemma \ref{jh44444}.
\end{proof}

By (\ref{rthr}) and
Corollary \ref{n4t6d} we have for every $X \in \Sm_S$ a natural epimorphism
\begin{equation}
s_X \colon \H^0(B_r'|_{X_\Zar}) \twoheadrightarrow (i_X)_*\nu_n^{r-1},
\end{equation}
where $i_X$ is the inclusion $X \times_S Z \hookrightarrow X$.
 
\begin{proposition}
\label{htgree}
The maps $s_X$ assemble
to an epimorphism $$s \colon
\H^0(B_r') \twoheadrightarrow i_* \nu_n^{r-1}.$$
\end{proposition}

In order to prove this Proposition we describe the maps
$s_X$ in a way Geisser used to define his version of
syntomic cohomology in \cite[\S 1,6]{geisser.dede}.

Let $X \in \Sm_S$.
We first give a construction of a map
$$b_X \colon (i_X)^* (\R^r j_* L_n(r)|_{X_\et}) \to \nu_n^{r-1}$$
in $\Sh((X_Z)_\et,\integers/p^n)$.
Over a complete discrete valuation ring of mixed characteristic
such a map was constructed in \cite[\S (6.6)]{bloch-kato.p-adic},
see also \cite[\S 6]{geisser.dede}.

We fix a point $\p \in Z$ and let $\Lambda$ be the
completion of the discrete valuation ring $D_\p$. Set
$T:= \Spec(\Lambda)$.
Let $\eta$ be the generic point of $T$.
Let $X_T:= X \times_S T$, and let
$X_\p$ be the special fiber and $X_\eta$ the generic fiber
of $X_T$.

We let $j_{X_T} \colon X_\eta \to X_T$ and
$i_{X_T} \colon X_\p \to X_T$ be the canonical inclusions.

Then the map $$b_{X_T} \colon
M_{n,X_T}^r:=(i_{X_T})^* \R^r (j_{X_T})_*(\integers/p^n(r))
\to \nu_n^{r-1}$$
in \cite[\S (6.6)]{bloch-kato.p-adic} is defined as follows
(recall $\integers/p^n(r)={\mathbb \mu}_{p^n}^{\otimes r}$):

By \cite[Corollary (6.1.1)]{bloch-kato.p-adic} the sheaf
$M_{n,X_T}^r$ is ($\mathrm{\acute{e}}$tale) locally generated
by symbols $\{x_1,\ldots,x_r\}$,
$x_i \in (i_{X_T})^*(j_{X_T})_* \caO_{X_\eta}^*$
(for the definition of symbol see
\cite[\S (1.2)]{bloch-kato.p-adic}).

Then for any $f_1,\ldots,f_r \in (i_{X_T})^* \caO_{X_T}^*$
the map $b_{X_T}$ sends the symbol $\{f_1,\ldots,f_r\}$ to $0$
and the symbol $\{f_1,\ldots,f_{r-1},\pi\}$
($\pi$ a uniformizer of $\Lambda$) to
$\dlog \overline{f}_1 \ldots \dlog \overline{f}_{r-1}$,
where $\overline{f}_i$ is the reduction of $f_i$ to
$\caO_{X_\p}^*$.

By multilinearity and the fact that $\{x,-x\}=0$ for $x \in
(i_{X_T})^*(j_{X_T})_* \caO_{X_\eta}^*$ this characterizes $b_{X_T}$ uniquely.

The base change morphism for the square
$$\xymatrix{X_\eta \ar[r]^{f_{X_U}} \ar[d]^{j_{X_T}} & X_U \ar[d]^{j_X} \\
X_T \ar[r]^{f_X} & X}$$
applied to the sheaf $\integers/p^n(r)$ on $(X_U)_\et$ yields
$$(f_X)^* \R^r(j_X)_* \integers/p^n(r) \to \R^r (j_{X_T})_* \integers/p^n(r)$$
(note that $(f_{X_U})^* \integers/p^n(r) = \integers/p^n(r)$).
Applying $(i_{X_T})^*$ and noting that $(i_{X_T})^* (f_X)^*=(i_\p)^*$
where $i_\p$ is the
inclusion $i_\p \colon X_\p \to X$ we get a map
$$(i_\p)^* \R^r(j_X)_* \integers/p^n(r) \to
M_{n,X_T}^r.$$

Composing with $b_{X_T}$ gives a map
$$(i_\p)^* \R^r(j_X)_* \integers/p^n(r) \to \nu_n^{r-1}.$$

Taking the disjoint union over all points in $Z$ we finally
get the map $$b_X \colon (i_X)^* \R^r (j_X)_* \integers/p^n(r)
\to \nu_n^{r-1},$$
the adjoint of which is a map
$$b_X' \colon \R^r (j_X)_* \integers/p^n(r) \to (i_X)_*\nu_n^{r-1}.$$

Together with the isomorphism of Corollary
\ref{h64de} we get the composition
$$s_X' \colon \H^0(B_r')|_{X_\Zar} \to \epsilon_* \H^0(B_r)_\et |_{X_\Zar}
\cong \epsilon_*\R^r (j_X)_* \integers/p^n(r) \to (i_X)_*\nu_n^{r-1}$$
(by our convention $\nu_n^{r-1}$ also denotes the logarithmic De Rham-Witt
sheaf on $(X_Z)_\Zar$).

\begin{proposition}
\label{gfedr5t5}
With the notation as above we have $s_X=s_X'$.
\end{proposition}

\begin{proof}
We keep the local completed situation at a point $\p$
of $Z$ from above.

We have a natural map induced by
flat pullback $(f_{X_U})^* \M_n^{X_U}(r) \to \M_n^{X_\eta}(r)$,
whence we get a base change morphism
$$f_X^*\R^r (j_X)_* \M_n^{X_U}(r) \to \R^r (j_{X_T})_* \M_n^{X_\eta}(r).$$

We get a diagram
$$\xymatrix{f_X^*\R^r (j_X)_* \M_n^{X_U}(r) \ar[r] \ar[d] &
\R^r (j_{X_T})_* \M_n^{X_\eta}(r) \ar[r] \ar[d] &
\H^{r-1}((i_{X_T})_* \M_n^{X_\p}(r-1)) \ar[d]^\cong \\
f_X^* \epsilon_* \R^r(j_X)_* \integers/p^n(r) \ar[r] &
\epsilon_* \R^r (j_{X_T})_* \integers/p^n(r) \ar[r] &
(i_{X_T})_* \nu_n^{r-1}.}$$
The left and middle vertical maps are induced by
the isomorphism of Corollary \ref{h64de} and (\ref{rthr}).
The left lower horizontal map is induced by the
transformation $f_X^*\epsilon_* \to \epsilon_* f_X^*$.
The upper right horizontal arrow is part of the
localization sequence for the motivic complexes.
The lower right horizontal map is induced by
$b_{X_T}$.

The claim of the Proposition follows from the
commutativity of the outer square. Indeed, a map
from the left upper corner to the right lower corner is
adjoint to a map
$$(i_\p)^*\R^r(j_X)_* \M_n^{X_U}(r)=
(i_{X_T})^*f^* \R^r(j_X)_* \M_n^{X_U}(r)
\to \nu_n^{r-1}.$$
The assertion that the outside compositions are the same
implies that the adjoints of $s_X$ and $s_X'$ coincide
over the point $\p$. Since this is true for all points in $Z$
the claim follows.

The left square of the above square commutes by naturality
of the cycle class map, Lemma \ref{gbhtfrr}.

So we are left to prove the commutativity of the right hand
square.

Since the right lower corner is an $\mathrm{\acute{e}}$tale
sheaf we can also sheafify this square in the
$\mathrm{\acute{e}}$tale topology to test commutativity.

The resulting square is adjoint to a square
$$\xymatrix{(i_{X_T})^* \epsilon^*\R^r(j_{X_T})_*
\M_n^{X_\eta}(r) \ar[r] \ar[d]^\cong &
\epsilon^*\H^{r-1}(\M_n^{X_\p}(r-1)) \ar[d]^\cong \\
(i_{X_T})^* \R^r(j_{X_T})_* \integers/p^n(r) \ar[r] &
\nu_n^{r-1}}$$
(the left vertical map is an isomorphism by Corollary \ref{h64de}).
This commutativity would follow from the commutativity
of the right hand square in the first diagram
in the proof of \cite[Theorem 1.3]{geisser.dede}.
This commutativity is not explicitely stated in loc.~cit.,
but the proof in loc.~cit.~that $\kappa \circ \alpha \circ c$
is $0$ shows the commutativity of our diagram:

As in loc.~cit.~let $R$ be the strictly henselian local ring
of a point in the closed fiber $X_\p$ of $X_T$,
let $L$ be the field of quotients of $R$, $F$ the field of
quotients of $R/\pi$, $V=R_{(\pi)}$, $V^h$ the henselization of
$V$ and $L^h$ the quotient field of $V^h$.

We have to show the commutativity of
$$\xymatrix{H^r(R[\frac{1}{\pi}],\M_n(r)) \ar[r] \ar[d]^\cong &
H^{r-1}(R/\pi, \M_n(r-1)) \ar[d]^\cong \\
H_\et^r(R[\frac{1}{\pi}],\integers/p^n(r)) \ar[r] &
\nu_n^{r-1}(R/\pi).}$$

The map $\nu_n^{r-1}(R/\pi) \to \nu_n^{r-1}(F)$ is injective
(see the proof of \cite[Theorem 1.3]{geisser.dede}, where
it is attributed to \cite[Corollary 1.6]{gros-suwa}).

Thus by the naturality of the localization sequence
for motivic complexes and the fact that the $b_{X_T}$
are sheaf maps
it is enough to show commutativity of the square which
one gets from the last square by replacing $R[\frac{1}{\pi}]$
with $L$ and $R/\pi$ with $F$.
But this square factors as
$$\xymatrix{H^r(L,\M_n(r)) \ar[r] \ar[d]^\cong &
H^r(L^h,\M_n(r)) \ar[r] \ar[d]^\cong &
H^{r-1}(F, \M_n(r-1)) \ar[d]^\cong \\
H_\et^r(L,\integers/p^n(r)) \ar[r] &
H_\et^r(L^h,\integers/p^n(r)) \ar[r] &
\nu_n^{r-1}(F).}$$

The right upper horizontal map is induced from the localization
sequence of the motivic complexes for $V^h$, its generic and its
closed point.

The left hand square commutes by naturality of the cycle
class map, and the commutativity of the right hand square
is shown in the proof of \cite[Theorem 1.3]{geisser.dede}
in the paragraph before the last paragraph.
This finishes the proof.
\end{proof}

We next discuss functoriality of the construction of the
morphisms $s_X'$. So let $g \colon Y \to X$ be a morphism
in $\Sm_S$. We still keep the local completed situation from
above. We let $g_Z$, $g_T$, $g_\eta$ and $g_\p$ be the base changes
of $g$ (over $S$) to $Z$, $T$, $\eta$ and $\p$.

Consider the diagram

$$\xymatrix{Y_\eta \ar[r]^{g_\eta} \ar@{^(->}[d]^{j_{Y_T}} &
X_\eta \ar@{^(->}[d]^{j_{X_T}} \\
Y_T \ar[r]^{g_T} & X_T \\
Y_\p \ar[r]^{g_\p} \ar@{_(->}[u]_{i_{Y_T}} & X_\p \ar@{_(->}[u]_{i_{X_T}}.}$$

A base change morphism gives us
$$(g_T)^* \R^r(j_{X_T})_*(\integers/p^n(r)) \to
\R^r(j_{Y_T})_*(\integers/p^n(r)).$$
Applying $(i_{Y_T})^*$ and using $(i_{Y_T})^*(g_T)^*
\cong (g_\p)^*(i_{X_T})^*$ gives 
$$(g_\p)^* M_{n,X_T}^r \to M_{n,Y_T}^r.$$

\begin{lemma}
\label{hteedd}
The diagram
$$\xymatrix{(g_\p)^* M_{n,X_T}^r \ar[rr]^{(g_\p)^*(b_{X_T})} \ar[d] & &
(g_\p)^* \nu_n^{r-1} \ar[d] \\
M_{n,Y_T}^r \ar[rr]^{b_{Y_T}} & &
\nu_n^{r-1}}$$
commutes.
\end{lemma}

\begin{proof}
This follows from the definition of the morphisms
$b_{X_T}$ and $b_{Y_T}$ in terms of symbols
and the functoriality of the symbols.
\end{proof}

As above for $X$ let $f_Y$ be the map $Y_T \to Y$.

\begin{lemma}
\label{gtdsgh}
The diagram
$$\xymatrix{g_T^*f_X^* \R^r(j_X)_* \integers/p^n(r) \ar[r] \ar[d] &
g_T^* \R^r (j_{X_T})_* \integers/p^n(r) \ar[d] \\
f_Y^* \R^r(j_Y)_* \integers/p^n(r) \ar[r] &
\R^r (j_{Y_T})_* \integers/p^n(r),}$$
where all maps are induced by base change morphisms,
commutes.
\end{lemma}

\begin{proof}
This follows by the naturality of the base change morphisms.
\end{proof}

\begin{corollary}
\label{gfrr4r}
The diagram
$$\xymatrix{(g_Z)^*(i_X)^* \R^r (j_X)_* \integers/p^n(r)
\ar[rr]^(.633){(g_Z)^*(b_X)} \ar[d] & &
(g_Z)^* \nu_n^{r-1} \ar[d] \\
(i_Y)^* \R^r (j_Y)_* \integers/p^n(r) \ar[rr]^(.633){b_Y} & &
\nu_n^{r-1},}$$
where the left vertical map is induced by a base change morphism,
commutes.
\end{corollary}
\begin{proof}
This follows by combining Lemmas \ref{hteedd} and \ref{gtdsgh}.
\end{proof}

\begin{corollary}
\label{bngzr}
The diagram
$$\xymatrix{g^* \R^r(j_X)_* \integers/p^n(r)
\ar[rr]^{g^*(b_X')} \ar[dd] & &
g^* (i_X)_* \nu_n^{r-1} \ar[d]^\cong \\
 & & (i_Y)_* (g_Z)^* \nu_n^{r-1} \ar[d] \\
\R^r (j_Y)_* \integers/p^n(r) \ar[rr]^{b_Y'} & &
(i_Y)_* \nu_n^{r-1}}$$
commutes.
\end{corollary}

\begin{proof}
We check that the adjoints with respect to the pair $(i_Y)^*$, $(i_Y)_*$
of the two compositions
are the two compositions of Corollary \ref{gfrr4r}.
For the composition via the left lower corner this is
immediate. For the other composition one uses
a compatibility between adjoints and pullbacks.
\end{proof}

\begin{corollary}
\label{bfdrr}
The maps $s_X'$ assemble to a map of sheaves
$\H^0(B_r') \to i_* \nu_n^{r-1}$.
\end{corollary}

\begin{proof}
This follows directly from Corollary \ref{bngzr}.
\end{proof}

\begin{proof}[Proof of Proposition \ref{htgree}]
The assertion follows by combining Proposition \ref{gfedr5t5}
and Corollary \ref{bfdrr}.
\end{proof}

Let $C_r$ be the kernel of the composition
$$B_r' \twoheadrightarrow \H^0(B_r') \overset{s}{\to} i_* \nu_n^{r-1}.$$

Then by construction of the maps $s_X$ we have for any
$X \in \Sm_S$ an isomorphism
\begin{equation}
\label{ht46jk}
C_r|_{X_\Zar} \cong \M_n^X(r)[r]
\end{equation}
in $\D(\Sh(X_\Zar,\integers/p^n))$
since both objects appear as (shifted) homotopy fibers
of the map
$$\R (j_X)_* \M_n^{X_U}(r) \to \R (i_X)_* \M_n^{X_Z}(r-1)[-1].$$
This isomorphism is even uniquely determined since
there are no non-trivial maps $\M_n^X(r) \to (i_X)_* \nu_n^{r-1}[-r-1]$
in $\D(\Sh(X_\Zar,\integers/p^n))$.

\begin{lemma}
\label{hnzt4}
Let $R$ be a commutative ring, $T \in \Sh(\Sm_{S,\Zar},R)$ and
$E$ an $E_\infty$-algebra in symmetric $T$-spectra in
$\Cpx^{\le 0}(\Sh(\Sm_{S,\Zar},R))$. Let $E_r$ be the levels
of $E$. Let for any $r>0$ an epimorphism
$\H^0(E_r) \twoheadrightarrow e_r$ in $\Sh(\Sm_{S,\Zar},R)[\Sigma_r]$
be given.
Let $E_r'$ be the kernel of the induced map
$E_r \to e_r$ and set $E_0':=E_0$.
Suppose the canoncial map
$\varphi \colon T \to E_1$ (which is the composition
$T \cong R \otimes T
\overset{\u \otimes \id}{\longrightarrow}
E_0 \otimes T \to E_1$ ($\u$ abbreviates $\mathrm{unit}$))
factors through $E_1'$ and that
for any $r,r'\ge 0$ the composition
in $\Sh(\Sm_{S,\Zar},R)$ induced by the $E_\infty$-multiplication
on $E$
$$\H^0(E_r') \otimes \H^0(E_{r'}') \to
\H^0(E_r) \otimes \H^0(E_{r'}) \to \H^0(E_{r+r'}) \to e_{r+r'}$$
(the tensor products are over $R$)
is the zero map. Then there is an induced structure of
an $E_\infty$-algebra $E'$ in symmetric $T$-spectra on the collection
of the $E_r'$ together with a map of $E_\infty$-algebras
$E' \to E$ which is levelwise the canonical map $E_r' \to E_r$.
\end{lemma}

\begin{proof}
The condition implies that we have natural maps
$$\phi_{r,r'} \colon \H^0(E_r') \otimes \H^0(E_{r'}') \to
\H^0(E_{r+r'}').$$
Let $\caO$ be our $E_\infty$-operad in $\Cpx^{\le 0}(\Sh(\Sm_{S,\Zar},R))$.
Note that each $E_r'$ carries an action of $\Sigma_r$.
The structure maps of the $E_\infty$-algebra in $T$-spectra $E$ are maps
$$s \colon E_r \otimes T \to E_{r+1}$$
and
$$a \colon \caO(k) \otimes E_{r_1} \otimes \cdots \otimes E_{r_k} \to E_r,$$
$r =\sum_{i=1}^k r_i$.
These are subject to certain conditions.
We show that when restricting these maps to the $E_r'$ they
factor through $E_r'$ (for the appropriate $r$).
Then it is clear that these new structure maps
also satisfy the conditions required.

To show that the composition
$$\caO(k) \otimes E_{r_1}' \otimes \cdots \otimes E_{r_k}' \to
\caO(k) \otimes E_{r_1} \otimes \cdots \otimes E_{r_k} \to E_r$$
factors through $E_r'$ it is suffcient to show that the induced
map on $\H^0$ factors through $\H^0(E_r')$.
But since $\caO$ is $E_\infty$ the map on $\H^0$ is a map
$$\H^0(E_{r_1}') \otimes \cdots \otimes \H^0(E_{r_k}') \to \H^0(E_r)$$
and the conditions to be $E_\infty$ imply that this map
is an iteration of the maps $\phi_{r',r''}$. Thus we get the factorization.

To handle the case of the $T$-spectrum structure maps
it is again sufficient to show that the composition
$$\psi \colon \H^0(E_r') \otimes T \to \H^0(E_r) \otimes T \to \H^0(E_{r+1})$$
factors through $\H^0(E_{r+1}')$.
But the commutativity of the diagram
$$\xymatrix{& \caO(2) \otimes E_r \otimes T
\ar[d]^\cong \ar@/_1pc/[dddl]_{\id \otimes \varphi} & \\
& \caO(2) \otimes E_r \otimes R \otimes T
\ar[d]^{\id \otimes \u \otimes \id} \ar[r] &
\caO(1) \otimes E_r \otimes T \ar[dd]^{a \otimes \id} \\
& \caO(2) \otimes E_r \otimes E_0 \otimes T
\ar[dl]^{\id \otimes s} \ar[dr]^{a \otimes \id} & \\
\caO(2) \otimes E_r \otimes E_1 \ar[dr]^a & & E_r \otimes T \ar[dl]^s \\
& E_{r+1} &}$$
(the only horizontal arrow is a structure map of the operad
using $R \cong \caO(0)$)
implies that $\psi$ is the composition
$$\H^0(E_r') \otimes T \to \H^0(E_r') \otimes \H^0(E_1') \to
H^0(E_{r+1})$$
which factors through $\H^0(E_{r+1}')$ by assumption.
This finishes the proof.
\end{proof}

We want to apply Lemma \ref{hnzt4} with $T=\integers/p^n[\GmS,\{1\}]_\Zar$,
$E=B'$ and $e_r = i_* \nu_n^{r-1}$. Then we have $E_r'=C_r$.

\begin{lemma}
The $\Sigma_r$-action on $\H^0(B_r')$ is the sign representation.
\end{lemma}

\begin{proof}
This follows from the fact that there is a zig zag of quasi-isomorphisms
between $\caT^{\otimes r}$ and $(L_n(1)[1])^{\otimes r}$, and
on the latter the $\Sigma_r$-action is strictly the sign
representation.
\end{proof}

So if we equip $\nu_n^{r-1}$ with the sign representation of
$\Sigma_r$ the map $\H^0(B_r') \to i_* \nu_n^{r-1}$
is $\Sigma_r$-equivariant.

The exact sequence
$$0 \to L_n(1) \to \GmU \overset{p^n}{\to} \GmU \to 0$$
on $\Sm_{U,\et}$ induces a boundary homomorphism
$$\beta \colon j_* \GmU \to \R^1 j_* L_n(1)$$
of sheaves on $\Sm_{S,\et}$. We denote the precomposition of
$\beta$ with the canonical map $\GmS \to j_* \GmU$ by $\beta'$.

\begin{lemma}
\label{hbfgt}
The composition
$$\GmS \to \integers/p^n[\GmS,\{1\}]_\Zar \overset{\varphi}{\to}
B_1' \to \H^0(B_1') \to \H^0(B_1')_\et \cong \R^1 j_* L_n(1)$$
equals $\beta'$.
\end{lemma}

\begin{proof}
This follows from the defining property of the map $\iota$.
\end{proof}

\begin{corollary}
\label{bfrejj}
The composition
$$\GmS \to \integers/p^n[\GmS,\{1\}]_\Zar \overset{\varphi}{\to}
B_1' \to \H^0(B_1') \to i_* \nu_n^{r-1}$$
is the constant map to zero.
\end{corollary}

\begin{proof}
This follows from Lemma \ref{hbfgt}, the definition of the map $b_{X_T}$
and the definition of symbol:
The symbol $\{x\}$ for $x$ an invertible section over a smooth
scheme over $S$ is sent to $0$ via $b_{X_T}$.
\end{proof}

Thus the first condition of Lemma \ref{hnzt4}
about the factorization of the map $\varphi$
is satisfied.

For the second condition we get back to our local
completed situation.
Let $X \in \Sm_S$, $\p \in Z$ and let the notation be as above.
By \cite[\S 3, top of p. 277]{kurihara}
there is an exact sequence
\begin{equation}
\label{teggh}
0 \to U^0 M_n^r \to M_n^r \to \nu_n^{r-1} \to 0
\end{equation}
on $(X_\p)_\et$, where
$U^0 M_n^r$ is the subsheaf of $M_n^r$ generated
\'etale locally by symbols $\{x_1,\ldots, x_r\}$
with $x_i \in (i_{X_T})^* \caO_{X_T}^*$.
This follows from the exact sequence
$$0 \to U^1 M_n^r \to M_n^r \to \nu_n^r \oplus \nu_n^{r-1} \to 0$$
(\cite[Theorem (1.4)(i)]{bloch-kato.p-adic}),
where $U^1 M_n^r$ is generated \'etale locally
by symbols $\{x_1,\ldots,x_r\}$ with
$x_1-1 \in \pi \cdot (i_{X_T})^* \caO_{X_T}$. Indeed, given an element
in the kernel of $M_n^r \to \nu_n^{r-1}$ we can first change
it by symbols $\{x_1,\ldots,x_r\}$ with $x_i \in (i_{X_T})^* \caO_{X_T}^*$
to lie also in the kernel of the map $M_n^r \to \nu_n^r$,
and then it lies in $U^1 M_n^r$ which is also generated by symbols
(of the indicated type).

\begin{lemma}
\label{gtrekj}
Let $r,r' \ge 0$. The composition
$$\H^0(C_r) \otimes \H^0(C_{r'}) \to
\H^0(B_r') \otimes \H^0(B_{r'}') \to \H^0(B_{r+r'}'),$$
where the second map is induced by the $E_\infty$-structure on $B'$,
factors through $\H^0(C_{r+r'})$.
\end{lemma}

\begin{proof}
Let $y$ be a local section lying in the kernel of
$\H^0(B_r') \to \nu_n^{r-1}$, similarly for $y'$.
We may view $y$ and $y'$ as local sections of
$M_n^r$ and $M_n^{r'}$. They are mapped to $0$ by the maps
to $\nu_n^{r-1}$ and $\nu_n^{r'-1}$, thus by the exact sequence
(\ref{teggh}) the sections $y$ and $y'$ can be written
locally as linear combinations
of symbols of the form $\{x_1,\ldots x_r\}$ and
$\{x_1',\ldots,x_{r'}'\}$
with $x_i, x_i' \in (i_{X_T})^* \caO_{X_T}^*$.
But the product of such symbols is just the
concatenated symbol $\{x_1,\ldots,x_r,x_1',\ldots,x_{r'}'\}$
which thus also lies in the kernel of the map
$M_n^{r+r'} \to \nu_n^{r+r'-1}$. This is true over all
points $\p$ of $Z$, so we see that $y \otimes y'$ is sent
to $0$ in $i_* \nu_n^{r+r'-1}$.
\end{proof}

\begin{corollary}
The collection of the $C_r$ forms an $E_\infty$-algebra $C$
in $\integers/p^n[\GmS,\{1\}]_\Zar$-spectra which comes with
a map of $E_\infty$-algebras $C \to B'$ wich is levelwise
the canonical map $C_r \to B_r'$.
\end{corollary}

\begin{proof}
This follows with Corollary \ref{bfrejj} and Lemma
\ref{gtrekj} from Lemma \ref{hnzt4}.
\end{proof}

Thus with (\ref{ht46jk}) we have arranged the motivic complexes
$\M_n^X(r)[r]$, $r \ge 0$, into an $E_\infty$-algebra in
$\integers/p^n[\GmS,\{1\}]_\Zar$-spectra on $\Sm_{S,\Zar}$.

\begin{proposition}
The algebra $C$ is an $\Omega$-spectrum.
\end{proposition}

\begin{proof}
Set $m:=p^n$.
Let $X \in \Sm_S$. Let $\i \colon \{0\} \to \AX$
be the closed, $\j \colon \GmX \to \AX$ the open
inclusion and $q \colon \AX \to X$ the projection.

Since $\integers/m[\GmS]_\Zar \cong
\integers/m \oplus \integers/m[\GmS,\{1\}]_\Zar$
we have a decomposition
$$\underline{\R \Hom}(\GmS,C_r) \cong
C_r \oplus \caR.$$

By Theorem \ref{gbter4} we have an exact triangle
$$\i_* \M_n(r-1)[-2] \to \M_n^\AX(r) \to \R \j_* \M_n(r)
\to \i_* \M_n(r-1)[-1].$$
The composition
$$C_r[-r]|_{X_\Zar} \cong \R q_* \M_n^\AX(r) \to \R q_* \R \j_* \M_n(r)$$
$$\cong \underline{\R \Hom}(\GmS,C_r[-r])|_{X_\Zar} \cong
C_r[-r]|_{X_\Zar} \oplus \caR[-r]|_{X_\Zar} \to C_r[-r]|_{X_\Zar}$$
is the identity.
Thus when we apply $\R q_*$ to the above triangle
we obtain a split triangle.
Let $\phi \colon \M^X_n(r-1)[-1] \overset{\cong}{\to} \caR[-r]|_{X_\Zar}$
be the resulting isomorphism.

We are finished when we prove that the diagram
$$\xymatrix{C_{r-1}|_{X_\Zar} \ar[r] \ar[d]^\cong &
\underline{\R \Hom}(\integers/m[\GmS,\{1\}]_\Zar,C_r)|_{X_\Zar}
\ar[d]^\cong \\
\M_n^X(r-1)[r-1] \ar[r]^(.6)\phi_(.6)\cong &
\caR|_{X_\Zar},}$$
where the upper horizontal map is the derived adjoint of the
structure map of the spectrum $C$,
commutes.
To see this it is sufficient to show that the post composition
of the two compositions with the map
$\caR|_{X_\Zar} \to \R j_*\caR'|_{X_\Zar}$,
where $\caR'$ is defined to be the second summand in
the decomposition $\underline{\R \Hom}(\GmU,A_r') \cong
A_r' \oplus \caR'$, coincide, since there are no non-trivial maps
from $C_{r-1}|_{X_\Zar}$ to $(i_X)_* \nu_n^{r-2}[-1]$.

But we have a transformation of diagrams from the above
diagram to the diagram
\begin{equation}
\label{bfrgjk}
\xymatrix{B_{r-1}'|_{X_\Zar} \ar[r] \ar[d]^\cong &
\underline{\R \Hom}(\integers/m[\GmS,\{1\}]_\Zar,B_r')|_{X_\Zar}
\ar[d]^\cong \\
\R (j_X)_* \M_n^{X_U}(r-1)[r-1] \ar[r]^(.6)\cong &
\R j_* \caR'|_{X_\Zar}}
\end{equation}
which commutes by the arguments in the proof of
Proposition \ref{bgerhh4}.
So the two prolongued compositions in question
are the two compositions in diagram (\ref{bfrgjk})
precomposed with the map $C_{r-1}|_{X_\Zar} \to
B_{r-1}'|_{X_\Zar}$, thus they coincide.
This finishes the proof.
\end{proof}

\subsubsection{The $p$-completed parts}
\label{jfezztrzz}

In this section we want to arrange (variants of) the $C$
for varying $n$ into a compatible family, such that
we can then take the (homotopy) limit of this system.

To start with write $\integers/p^\bullet$ for the inverse system
comprised by the commutative rings $\integers/p^n$ with
the obvious transition maps and $\Mod_{\integers/p^\bullet}$
for the category of modules over this system, i.e. the category
whose objects are systems of abelian groups
$$\cdots \to M_n \to \cdots \to M_2 \to M_1$$
where each $M_n$ is annihilated by $p^n$.

For a site $\caS$ write $\Sh(\caS,\integers/p^\bullet)$ for
$\Sh(\caS,\Mod_{\integers/p^\bullet})$.

The system of the $L_n(r)$ comprises a natural object
$L_\bullet(r)$ of $\Sh(\Sm_{U,\et},\integers/p^\bullet)$.

Let $QL_\bullet(1) \to L_\bullet(1)$ be a cofibrant replacement
in $\Cpx(\Sh(\Sm_{U,\et},\integers/p^\bullet))$ (the latter category
is equipped with the inverse local projective model structure)
and let $QL_\bullet(1) \to RQL_\bullet(1)$ be a fibrant
replacement via a cofibration. Thus $\caT:=RQL_\bullet(1)[1]$ is both fibrant
and cofibrant.

We claim that the maps $\iota$ from section \ref{htr46u}
can be arranged to a map
$$\iota \colon \integers/p^\bullet[\GmU,\{1\}]_\et \to \caT.$$
Indeed, suppose we have already defined $\iota$ up to level
$n$ in such a way that on each level $k \le n$ the map
represents the canonical element $1 \in H^1_\et(\GmU,L_k(1))$.
We claim that we can extend the system of maps to level $n+1$:
Choose a representative
$$\iota' \colon \integers/p^{n+1}[\GmU,\{1\}]_\et \to \caT_{n+1}.$$
Then the composition with the fibration $\caT_{n+1} \to \caT_n$
is homotopic to the map in level $n$. This homotopy can be lifted
giving as second endpoint the required lift.

As in section \ref{htr46u} the map of symmetric
sequences $\Sym(\iota)$ gives rise to an $E_\infty$-algebra
$\Sym(\caT)$ in $\Omega$-$\integers/p^\bullet[\GmU,\{1\}]_\et$-spectra,
and we let $Q \Sym(\caT) \to \Sym(\caT)$ be a cofbrant replacement via a trivial fibration and
$Q \Sym(\caT) \to R Q \Sym(\caT)$ be
a fibrant resolution.

Set $A:=\epsilon_*(R Q \Sym(\caT))$ and
$A':= \tau_{\le 0}(A)$. As in section \ref{htr46u}
$A'$ is again an $E_\infty$-algebra. We set $B:=j_*A'$.
By Lemma \ref{jh44444} the algebra $B$ computes levelwise in the $n$-direction
the algebra which was denoted $B'$ in section \ref{htr46u}.

Thus we have for every $n$ and $r$ the epimorphism
$$s_{r,n} \colon \H^0(B_{r,n}) \to i_* \nu_n^{r-1}$$
of Proposition \ref{htgree}.

\begin{lemma}
We have a commutative diagram
$$\xymatrix{\H^0(B_{r,n+1}) \ar[r]^(.57){s_{r,n+1}} \ar[d] &
i_*\nu_{n+1}^{r-1} \ar[d] \\
\H^0(B_{r,n}) \ar[r]^(.57){s_{r,n}} & \nu_n^{r-1}.}$$
\end{lemma}

\begin{proof}
We only have to verify that a corresponding diagram
involving the maps $s_X'$ commutes.
This follows by the explicit definition of the maps $b_{X_T}$.
\end{proof}

We thus get an epimorphism $$B_r \to i_* \nu_\bullet^{r-1}.$$
We denote by $C_r$ the kernel of this epimorphism.

As in section \ref{htr46u} we can apply a variant of
Lemma \ref{hnzt4} (or the Lemma levelwise in the $n$-direction
and using functoriality) to see that the collection of the $C_r$
gives rise to an $E_\infty$-algebra $C$ together with a map
of $E_\infty$-algebras $C \to B$ which is levelwise
(for the $r$-direction) the canonical map $C_r \to B_r$.

Let $X \in \Sm_S$.
We want to see that the canonical isomorphisms (\ref{ht46jk})
$$C_{r,n}|_{X_\Zar} \cong \M_n^X(r)[r]$$
are compatible with the reductions $\integers/p^{n+1}
\to \integers/p^n$.

First by Lemma \ref{nhgrrh} the diagram
$$\xymatrix{\M_{n+1}^{X_U}(r)[r] \ar[r]^\cong \ar[d] &
A_{r,n+1}'|_{(X_U)_\Zar} \ar[d] \\
\M_n^{X_U}(r)[r] \ar[r]^\cong &
A_{r,n}'|_{(X_U)_\Zar}}$$
commutes.

This shows that if we compose the
two compositions in the square
\begin{equation}
\label{hbgtedr}
\xymatrix{C_{r,n+1}|_{X_\Zar} \ar[r]^\cong \ar[d] &
\M_{n+1}^X(r)[r] \ar[d] \\
C_{r,n}|_{X_\Zar} \ar[r]^\cong &
\M_n^X(r)[r]}
\end{equation}
with the map $\M_n^X(r)[r] \to \R (j_X)_*\M_n^{X_U}(r)[r]$
the resulting two maps coincide. But
$\Hom_{\D(\Sh(X_\Zar,\integers/p^{n+1}))}
(C_{r,n+1}|_{X_\Zar},\nu_n^{r-1}[-1])=0$, hence
(\ref{hbgtedr}) commutes.

Let $QC \to C$ be a cofibrant and
$QC \to C'$ be a fibrant replacement as $E_\infty$-algebras.
Then $$D_p :=\lim_n C_{\bullet,n}'$$ is an $E_\infty$-algebra
in $\integers_p[\GmS,\{1\}]_\Zar$-spectra in
$\Cpx(\Sh(\Sm_{S,\Zar},\integers_p))$.

\begin{corollary}
For $X \in \Sm_S$
there is an isomorphism $$D_{p,r}|_{X_\Zar} \cong (\M^X(r))^{\wedge p}[r]$$
in $\D(\Sh(X_\Zar,\integers_p))$,
where $(\M^X(r))^{\wedge p}$ is the $p$-completion of
$\M^X(r)$.
\end{corollary}
\label{guu43ddf}

\begin{proof}
This follows from the commutativity of (\ref{hbgtedr}),
since the $p$-completion of $\M^X(r)$ is the homotopy limit
over all the $\M_n^X(r)$.
\end{proof}

Next we will equip $D_p$ with an orientation.

Denote by $\caO^*_{/U}$ the sheaf (in any of the considered topologies) of abelian groups represented by $\GmU$ over $\Sm_U$, let
$\caO^*_{/S}$ be defined similarly. For $M$ a sheaf of abelian groups
or an object in a triangulated category we set $M/p^n:=M \otimes ^\bL \integers/p^n$.

Using the resolution of $\caO^*_{/U}$ by the sheaf of meromorphic functions and the sheaf of
codimension $1$ cycles one sees that $\R^i j_* \caO^*_{/U}=0$ for $i>0$.
Thus we have an exact triangle
$$\caO^*_{/S} \to \R j_* \caO^*_{/U} \to i_* \integers \to \caO^*_{/S}[1]$$
in the Zariski topology,
from which we derive an exact triangle
\begin{equation}
\label{gddez54}
\caO^*_{/S}/p^n \to \R j_* \caO^*_{/U}/p^n \to i_* \integers/p^n \to \caO^*_{/S}/p^n[1].
\end{equation}

We have a map of exact triangles
$$\xymatrix{\caO^*_{/U} \ar[r]^{p^n} \ar[d] & \caO^*_{/U} \ar[r] \ar[d] & \caO^*_{/U}/p^n \ar[r] \ar[d] & \caO^*_{/U}[1] \ar[d] \\ 
\R \epsilon_* \caO^*_{/U} \ar[r]^{p^n} & \R \epsilon_* \caO^*_{/U} \ar[r] & \R \epsilon_* L_n(1)[1] \ar[r] & \R \epsilon_* \caO^*_{/U}[1].}$$
The third vertical map factors uniquely through a map $\caO^*_{/U}/p^n \to \tau_{\le 0}(\R \epsilon_* L_n(1)[1])$. Since $\R^1 \epsilon_* \caO^*_{/U}=0$
we see by the long exact cohomology sheaf sequences associated to these triangles that this map is an isomorphism.
Note we have $A_{1,n}' \cong \tau_{\le 0}(\R \epsilon_* L_n(1)[1])$ in the derived category, and thus
$B_{1,n} \cong \R j_* \tau_{\le 0}(\R \epsilon_* L_n(1)[1]) \cong \R j_* \caO^*_{/U}/p^n$.

We note that the diagram
$$\xymatrix{\R j_* \caO^*_{/U}/p^n \ar[rr] \ar[d]^\cong & & i_* \integers/p^n \ar[d]^= \\
\R j_* \tau_{\le 0}(\R \epsilon_* L_n(1)[1]) \ar[r] & \caH^0(B_{1,n}) \ar[r]^{s_{r,n}} & i_* \integers/p^n}$$
commutes (this follows from the definition of the maps $s_{r,n}$, Proposition \ref{gfedr5t5} and the definition of the maps $s_X'$).
Thus together with the triangle (\ref{gddez54}) we derive an isomorphism $C_{1,n} \cong \caO^*_{/S}/p^n$
in $\D(\Sh(\Sm_{S,\Zar},\integers/p^n))$. This isomorphism is moreover unique since
there are no non-trivial homomorphisms from $\caO^*_{/S}/p^n$ to $i_* \integers/p^n[-1]$.

We see that there is an isomorphism $D_{p,1} \cong (\caO^*_{/S})^{\wedge p}$ in $\D(\Sh(\Sm_{S,\Zar},\integers_p))$. We denote any
such isomorphism which is compatible with the projections to $C_{1,n}$ and $\caO^*_{/S}/p^n$ by $\varphi$.

Since $D_p$ is an $\Omega$-spectrum which satisfies Nisnevich descent
and is $\mathbb{A}^1$-local the maps $\Sigma^{-2,-1} \Sigma^\infty_+ \P^\infty \to D_p$ in
$\SH(S)$ correspond to maps $$\integers[\P^\infty]_\Zar[-1] \to D_{p,1}$$ in $\D(\Sh(\Sm_{S,\Zar},\integers))$.
We let $o \colon \Sigma^{-2,-1} \Sigma^\infty_+ \P^\infty \to D_p$ correspond to
$\integers[\P^\infty]_\Zar \to \caO^*_{/S}[1] \to (\caO^*_{/S})^{\wedge p}[1] \overset{\varphi^{-1}[1]}{\longrightarrow} D_{p,1}[1]$,
where the first map classifies the tautological line bundle $\caO(-1)$ on $\P^\infty$.

The definition of the bonding maps in $D_p$ implies that the map
$\integers[\GmS,\{1\}]_\Zar \to D_{p,1}$ corresponding to the unit map
$\Sigma^{-1,-1} \Sigma^\infty (\GmS,\{1\}) \cong \unit \to D_p$ is the map
$$\integers[\GmS,\{1\}]_\Zar \to \caO^*_{/S} \to (\caO^*_{/S})^{\wedge p} \overset{\varphi^{-1}}{\longrightarrow} D_{p,1}.$$
Note that this composition is independent of the particular choice of $\varphi$ since
we have $$\Hom_{\D(\Sh(\Sm_{S,\Zar},\integers))}(\integers[\GmS,\{1\}]_\Zar,(\caO^*_{/S})^{\wedge p}) \cong \integers_p$$ 
$$\cong \lim_n \Hom_{\D(\Sh(\Sm_{S,\Zar},\integers))}(\integers[\GmS,\{1\}]_\Zar,\caO^*_{/S}/p^n).$$

Let $\psi \colon (\P^1,\{\infty\}) \to \GmS \wedge S^1$ be the canonical isomorphism in $\caH_\bullet(S)$ and let
$c \colon \caH_\bullet(S) \to \D^{\mathbb{A}^1}(\Sh(\Sm_{S,\Nis},\integers))$ be the canonical map.
Then the composition
$$\integers[\P^1,\{\infty\}]_\Zar \cong c((\P^1,\{\infty\})) \overset{c(\psi)}{\longrightarrow} c(\GmS \wedge S^1) \cong
\integers[\GmS,\{1\}]_\Zar[1] \to \caO^*_{/S}[1]$$
in $\D^{\mathbb{A}^1}(\Sh(\Sm_{S,\Nis},\integers))$
classifies the tautological line bundle on $\P^1$. We see from these considerations that $o$ is indeed an orientation.

\begin{proposition}
\label{vfjt43}
The spectrum in $\SH(S)$ associated with $D_p$ is orientable.
\end{proposition}

\subsection{The completed part}
\label{bfdtzt}

Set $D:=\prod_p D_p$, where the $D_p$ are the algebras from
the last section viewed as $E_\infty$-algebras in spectra
in $\Cpx(\Sh(\Sm_{S,\Zar},\hat{\integers}))$
and the product is taken over all primes.

Then for $X \in \Sm_S$ we have $$D_r|_{X_\Zar} \cong (\prod_p (\M^X(r))^{\wedge p})[r]$$
in $\D(\Sh(X_\Zar,\hat{\integers}))$.

\begin{corollary}
\label{fvfds5t}
The spectrum in $\SH(S)$ associated with $D$ is orientable.
\end{corollary}

\begin{proof}
This follows from Proposition \ref{vfjt43}.
\end{proof}

\subsection{The rational parts}

We denote by $D_\Q$ the rationalization of $D$ as an $E_\infty$-spectrum.

We denote by $\HB$ the Beilinson spectrum over $S$, see \cite[Definition 13.1.2]{cisinski-deglise}.
It has a natural $E_\infty$-structure (\cite[Corollary 13.2.6]{cisinski-deglise}) and is orientable
(\cite[13.1.5]{cisinski-deglise}).

\begin{theorem}
\label{vfdtjuzu}
$\HB$ is the initial $E_\infty$-spectrum among rational orientable $E_\infty$-spectra.
\end{theorem}

\begin{proof}
This is \cite[Corollary 13.2.15 (Rv)]{cisinski-deglise}.
\end{proof}

\begin{corollary}
There is a canonical map of $E_\infty$-spectra $\HB \to D_\Q$.
\end{corollary}

\begin{proof}
This follows from Corollary \ref{fvfds5t} and Theorem \ref{vfdtjuzu}.
\end{proof}

\subsection{The definition}

\begin{definition}
\label{ght544ede}
We denote by $\MZ$ the homotopy pullback in $E_\infty$-spectra of the diagram
$$\xymatrix{& D \ar[d] \\ \HB \ar[r] & D_\Q.}$$
\end{definition}
If we want to emphasize the dependence on $S$ we also write $\MZ_S$.

\section{Motivic Complexes II}
\label{h5643dfgf}

\subsection{A strictification}
\label{kdtjjz}

In this section we enlarge the motivic complexes from section \ref{45zgf}
to presheaves on all of $\Sm_S$. We need some preparations.

For each $n \in \naturals$ we define a category $\caE_n$ together with a functor
$\varphi_n \colon \caE_n \to [n]$, where $[n]$ is the category $0 \to 1 \to \cdots \to n$.
The objects of $\caE_n$ are triples $(A,B,i)$ where $i \in [n]$ and
$A \subset B \subset \{i,\ldots,n\}$ with $i \in A$. There is exactly one morphism from
$(A,B,i)$ to $(A',B',j)$ if $i \le j$, $B \cap \{j,\ldots,n\} \subset B'$ and
$A' \subset A$, otherwise there is no such morphism.
The functor $\varphi_n$ is determined by the fact that $(A,B,i)$ is mapped to $i$.
We declare a map $f$ in $\caE_n$ to be a weak equivalence if $\varphi_n(f)$ is
an identity.

A {\em category with weak equivalences} is a category $\caC$ together with a
subcategory
$\caW$ of $\caC$ such that every isomorphism in $\caC$ lies in $\caW$. A {\em
homotopical
category} is a category with weak equivalences satisfying the two out of six
property,
see \cite[8.2]{dwyer-hirschhorn-kan-smith}.

For a category $\caC$ with weak equivalences $\caW$ we denote
by $L^H_\caW \caC$ its hammock localization, see \cite{dwyer-kan.calculating}.
If it is clear which weak equivalences are meant we also write $L^H \caC$.

A morphism in $[n]$ is defined to be a weak equivalence if it is an identity.
So both $\caE_n$ and $[n]$ are homotopical categories. Since $[n]$ is the homotopy
category of $L^H([n])$ there is a natural simplicial functor $L^H([n]) \to [n]$ which
is an equivalence of simplicial categories.
Composing with the natural functor $L^H \caE_n \to L^H([n])$ gives us the simplicial
functor $L^H \caE_n \to [n]$.

\begin{proposition}
\label{ghte56}
The natural functor $L^H \caE_n \to [n]$ is an equivalence of simplicial categories.
\end{proposition}

Before giving the proof we need some preparations.

For us a {\em direct} category is a category with a chosen degree function,
see \cite[Definition 5.1.1]{hovey.model}.

\begin{lemma}
\label{k4fdb}
Let $I$ be a direct category and $J \subset I$ a full subcategory such that
no arrow in $I$ has a domain which is not in $J$ and a codomain which is in $J$.
Let $\caC$ be a model category and $D \colon I \to \caC$ a cofibrant diagram for the projective
model structure.
Then $D|_J$ is cofibrant in $\caC^J$.
\end{lemma}
\begin{proof}
The right adjoint $r$ to the restriction functor $\caC^I \to \caC^J$ is a right Quillen functor
since for $i \in I \setminus J$ we have $r(D)(i)=*$.
\end{proof}

\begin{lemma}
\label{ghjujd}
Let $I$ be a direct category and $J \subset I$ a full subcategory such that
no arrow in $I$ has a domain which is not in $J$ and a codomain which is in $J$.
Let $\caC$ be a model category and $D \colon I \to \caC$ a cofibrant diagram for the projective
model structure.
Then the canonical map $\colim (D|_J) \to \colim D$ is a cofibration.
\end{lemma}
\begin{proof}
The object $\colim D$ is obtained from $\colim (D|_J)$ by succesively gluing in the $D(i)$ for
$i \in I \setminus J$ for increasing degree of $i$. The domains of the attaching maps
are corresponding latching spaces.
\end{proof}

For $i \in [n]$ let $\caE_{n,i}:=\varphi^{-1}(i)$ and
$\caE_{n,\le i}$ be the full subcategory of $\caE_n$ of objects
$(A,B,j)$ with $j \le i$. 
It is easily seen that $\caE_{n,\le i}$ can be given the structure of a direct category.
For $j \le i \le n$ let $\caE_{n,[j,i]}:=\varphi_n^{-1}(\{j,\ldots,i\})$.

\begin{lemma}
\label{hgtr7i}
Let $\caC$ be a model category and $D \colon \caE_{n,\le i} \to \caC$ be a projectively
cofibrant diagram. Then for $k \le j \le i$ the restriction $D|_{\caE_{n,[k,j]}}$ is also cofibrant.
\end{lemma}

\begin{proof}
Let $F \colon \caE_{n,[k,j]} \to \caE_{n,\le i}$ be the inclusion. We claim that the right adjoint to
the restriction functor $\caC^{\caE_{n,\le i}} \to \caC^{\caE_{n,[k,j]}}$ is a right Quillen functor.
This follows from the fact that for $l<k$ and an object $(A,B,l) \in \caE_{n,\le i}$ with $A \cap [k,j] \neq \emptyset$
the category $(A,B,l)/F$ has the initial object $(A,B,l) \to (A \cap \{m,\ldots,n\},B \cap \{m,\ldots,n\},m\}$,
where $m=\min(A \cap [k,j])$.
\end{proof}

\begin{lemma}
\label{h4ehh}
Let $\caC$ be a model category and $D \colon \caE_{n,\le i} \to \caC$ be a projectively
cofibrant diagram such that for any weak equivalence $f$ in $\caE_{n,\le i}$ the map $D(f)$
is also a weak equivalence. Then for any $X \in \varphi_n^{-1}(i)$ the map
$D(X) \to \colim D$ is a weak equivalence.
\end{lemma}

\begin{proof}
We show by descending induction on $j$, starting with $j=i$, that for any
$X \in \varphi_n^{-1}(i)$ the map
$D(X) \to \colim D|_{\caE_{n,[j,i]}}$ is a weak equivalence.
For $j=i$ this follows from the fact $\caE_{n,i}$ has a final object.
Let the statement be true for $0<j+1 \le i$ and let us show it for $j$.
Let $J \subset \caE_{n,j}$ be the full subcategory on objects
$(A,B,j)$ such that $A \cap \{j+1,\ldots,i\} \neq \emptyset$.
Then we have a pushout diagram
$$\xymatrix{\colim D|_J \ar[r] \ar[d] & \colim D|_{\caE_{n,j}} \ar[d] \\
\colim D|_{\caE_{n,[j+1,i]}} \ar[r] & \colim D|_{\caE_{n,[j,i]}}}$$
First note that by Lemmas \ref{hgtr7i} and \ref{k4fdb} all objects in this
diagram are cofibrant. Furthermore the upper horizontal map is a cofibration
by Lemma \ref{ghjujd}. 
The full subcategory of $J$ consisting of objects $(A,B,j)$ with $B=\{j,\ldots,n\}$ is homotopy
right cofinal in $J$ and contractible (it has an initial object), thus $J$ is contractible.
Since the diagram $D|_J$ is weakly equivalent to a constant diagram it follows that $D(X) \to \colim D|_J$
is a weak equivalence for any $X \in J$, thus the upper horizontal map
in the above diagram is also a weak equivalence and the induction step follows.
\end{proof}

\begin{lemma}
\label{gre5g}
Let $\caC$ be a model category and $l$ the left adjoint to the pull back functor
$r \colon \caC^{[n]} \to \caC^{\caE_n}$. Let $D \colon \caE_n \to \caC$ be (projectively) 
cofibrant such that for any weak equivalence $f$ in $\caE_n$ the map $D(f)$
is a weak equivalence. Then $D \to r(l(D))$ is a weak equivalence.
\end{lemma}
\begin{proof}
We have $l(D)(i)=\colim D|_{\caE_{n,\le i}}$. Thus the claim follows from Lemma
\ref{h4ehh}.
\end{proof}

\begin{lemma}
\label{gfvre5}
Let $F \colon I \to J$ be an essentially surjective functor between small categories and $\caW \subset I$ a subcategory making $I$
into a category with weak equivalences. Suppose $F$ sends any map in $\caW$ to an isomorphism.
Then the natural map $L^H_\caW I \to J$ is a weak equivalence between simplicial categories if and only
if for any projectively cofibrant diagram $D  \colon I \to \sSet$ such that for any map $f$ in $\caW$
the map $D(f)$ is a weak equivalence the map $D \to r(l(D))$ is a weak equivalence where
$l$ is the left adjoint to the pull back functor $r \colon \sSet^J \to \sSet^I$.
\end{lemma}

\begin{proof}
Let $\caC$ be the left Bousfield localization of the model category $\sSet^I$ (equipped with
the projective model structure) along the maps
$\Hom(f,\_)$ where $f$ runs through the maps of $\caW$. Then $(L^H_\caW I)^\op$ is
weakly equivalent to the full simplicial subcategory of $\sSet^I$ consisting of cofibrant fibrant objects
which become isomorphic in $\Ho(\sSet^I)$ to objects in the image of the composed functor
$I^\op \to \Ho(\sSet^I) \to \Ho \caC \hookrightarrow \Ho(\sSet^I)$.
Similarly (but easier) $J^\op$ is weakly equivalent to a full simplicial subcategory
of $\sSet^J$.
The functor $(L^H_\caW I)^\op \to J^\op$
is described via these equivalences by the restriction of the push forward
$\sSet^I \to \sSet^J$ followed by a fibrant replacement functor. The claim follows.
\end{proof}

\begin{proof}[Proof of Proposition \ref{ghte56}]
The claim follows from Lemmas \ref{gre5g} and \ref{gfvre5}.
\end{proof}

Let $f \colon [n] \to [m]$ be a map in $\bigtriangleup$. We define a functor
$f_* \colon \caE_n \to \caE_m$ by setting $f_*((A,B,i))=(f(A),f(B),f(i))$. One checks
that this determines uniquely $f_*$. Thus we get a cosimplicial object
$\caE \colon [n] \mapsto \caE_n$ in the category of small categories with weak
equivalences. Applying the hammock localization
yields a cosimplicical simplicial category $L^H \caE \colon [n] \mapsto L^H \caE_n$
together with a map from $L^H \caE$ to the standard cosimplicial simplicial category
$[\bullet]$ which is levelwise a Dwyer-Kan equivalence.

Let $\caS$ be the category of triples $(A,n,(a_1,\ldots,a_n))$ where $A$ is a $D$-algebra such that $\Spec(A) \in \Sm_S$
and $a_1,\ldots,a_n \in A$ generate $A$ as a $D$-algebra. Morphisms are morphisms of $D$-algebras
with no compatibility of the generators required. Clearly the functor $\caS^\op \to \Sm_S$, $A \mapsto \Spec(A)$,
is an equivalence onto the full subcategory of $\Sm_S$ of affine schemes.

For $X \in\Sm_S$ and $F=\{f_1,\ldots f_n\}$ a set of closed immersions $f_i \colon Z_i \hookrightarrow X$ in $\Sm_S$
we denote by $z^r_F(X)$ the normalized chain complex associated to the simplicial abelian group
$[n] \mapsto z^r_F(X,n)$ which is the subsimplicial abelian group of $z^r(X,\bullet)$ of cycles in good position
with respect to the $Z_i$. We also write $z^r_F(A)$ for $z^r_F(\Spec(A))$.
We have the following moving Lemma due to Marc Levine.

\begin{theorem}
\label{6gfrfd}
If $X \in \Sm_S$ is affine then the inclusion of $z^r_F(X)$ into the normalized chain complex associated to
$z^r(X,\bullet)$ is a quasi isomorphism.
\end{theorem}
\begin{proof}
This is \cite[Theorem 4.9]{levine.schemes}.
\end{proof}

Let $$(A_0,k_0,(a_{0,1},\ldots,a_{0,k_0})) \to \cdots \to (A_n,k_n,(a_{n,1},\ldots,a_{n,k_n}))$$ be a chain of maps
in $\caS$, i.e. a $n$-simplex, which we denote by $K$, in the nerve of
$\caS$. Let $i \in [n]$ and $B \subset \{i,\ldots,n\}$ with $i \in B$. Set
$$C_{i,B}:= \bigotimes_{j \in B \setminus \{i\}} A_i[T_1 \ldots,T_{k_j}]
\cong \bigotimes_{j \in B \setminus \{i\}} A_i[T_{j,1} \ldots,T_{j,k_j}],$$
where the tensor products are over $A_i$. If $i \le j \le n$, $B' \subset \{j,\ldots,n\}$ with $j \in B'$ and $B \cap \{j,\ldots,n\} \subset B'$
we define a map $g_{i,B,j,B'} \colon C_{i,B} \to C_{j,B'}$ over the map $A_i \to A_j$ by sending a variable $T_{l,m}$ for $l>j$ to the respective variable
$T_{l,m}$ and to the image of the element $a_{l,m}$ in $A_j$ for $l \le j$. If furthermore $j \le k \le n$ 
and $B'' \subset \{k,\ldots,n\}$ with $k \in B''$ and $B' \cap \{k,\ldots,n\} \subset B''$ then we have
\begin{equation}
\label{jhe56d}
g_{j,B',k,B''} \circ g_{i,B,j,B'} = g_{i,B,k,B''}. 
\end{equation}

For $t=(A,B,i) \in\caE_n$ we let $F_t$ be the set of closed subschemes of $\Spec(C_{i,B})$ consisting of the
$\Spec(g_{i,B,j,B \cap \{j,\ldots,n\}})$ for $j \in A \setminus \{i\}$. For such $t$ set $C_t:=C_{i,B}$.
Clearly for $j \in A$ we have a pullback functor $z^r_{F_t}(C_t) \to z^r(C_{j,B \cap \{j,\ldots,n\}})$
induced by pullback of cycles (see Appendix \ref{ht635zh}) which for
$B' \subset \{j,\ldots,n\}$ with $B \cap \{j,\ldots,n\} \subset B'$ we can prolong via smooth pullback to
a map to $z^r(C_{j,B'})$.

\begin{lemma}
\label{jg45fj}
Let $t \to s$ be a map in $\caE_n$.
Then the above map $z^r_{F_t}(C_t) \to z^r(C_s)$ factors through $z^r_{F_s}(C_s)$.
\end{lemma}

\begin{proof}
Let $t=(A,B,i)$ and $s=(A',B',j)$. Set $s':=(A \cap \{j,\ldots,n\},B \cap \{j,\ldots,n\},j)$.
Without loss of generality we can assume $A'=A \cap \{j,\ldots,n\}$.
Clearly the map $$z^r_{F_t}(C_t) \to z^r(C_{s'})$$ factors through $z^r_{F_{s'}}(C_{s'})$.
If $A'=\{j\}$ we are done, otherwise fix $k \in A' \setminus \{j\}$. Set $B'':=(B \cap \{j,\ldots,n\}) \cup (B' \cap \{j,\ldots,k\})$
and $s'':=(A',B'',j)$.
Then we have a well defined map $$z^r_{F_{s'}}(C_{s'}) \to z^r_{\{g_{j,B'',k,B''\cap \{k,\ldots,n\}}\}}(C_{s''})$$ since cycles meet in the correct
codimension. Furthermore we have a well defined map $$z^r_{\{g_{j,B'',k,B''\cap \{k,\ldots,n\}}\}}(C_{s''}) \to
z^r_{\{g_{j,B',k,B'\cap \{k,\ldots,n\}}\}}(C_s)$$ for the same reason. Altogether we see that cycles meet as claimed.
\end{proof}

For $f \colon t \to s$ a map in $\caE_n$ we let $\alpha_K(f) \colon z^r_{F_t}(C_t) \to z^r_{F_s}(C_s)$ be the map
constructed above using Lemma \ref{jg45fj}.

\begin{lemma}
\label{gftrg66}
For $f \colon t \to s$ and $g \colon s \to r$ two maps in $\caE_n$ we have $\alpha_K(g \circ f) = \alpha_K(g) \circ \alpha_K(f)$.
\end{lemma}

\begin{proof}
Let $t=(A,B,i)$, $s=(A',B',j)$ and $r=(A'',B'',k)$. Then the map $\alpha_K(f)$ is defined by pulling back cycles via
the map $\Spec(g_{i,B,j,B'})$ and the map $\alpha_K(g)$ by pull back via  $\Spec(g_{j,B',k,B''})$. Thus the claim
follows from (\ref{jhe56d}) and Theorem \ref{h5t423t}.
\end{proof}

Setting $\alpha_K(t):=z^r_{F_t}(C_t)$ for $t$ an object of $\caE_n$ and using Lemma \ref{gftrg66} we get a functor
$\alpha_K \colon \caE_n \to \Cpx(\Ab)$.

By restricting everything to opens $U$ in $S_\Zar$ we get a functor $$\tilde{\alpha}_K \colon \caE_n \to \Cpx(\Sh(S_\Zar,\integers)).$$

\begin{lemma}
\label{grfe4d}
The functor $\tilde{\alpha}_K$ sends weak equivalences in $\caE_n$ to quasi isomorphisms.
\end{lemma}

\begin{proof}
This follows from Theorem \ref{6gfrfd}, Theorem \ref{htehg} and the fact that for $X \in \Sm_S$ the push forward
of $(X_\Zar \ni Y \mapsto z^r(Y))$ via the structure morphism $X \to S$ computes the derived push forward,
which follows from \cite[Theorem 1.7]{levine.techniques}.
\end{proof}

\begin{lemma}
\label{4nhzg}
Let $f \colon [m] \to [n]$ be a monomorphism in $\bigtriangleup$ and $K$ a $n$-simplex in the nerve of $\caS$.
Then the composition $\caE_m \overset{f_*}{\longrightarrow}
\caE_n \overset{\tilde{\alpha}_K}{\longrightarrow}
\Cpx(\Sh(S_\Zar,\integers))$ is equal to $\tilde{\alpha}_{f^*K}$. 
\end{lemma}
\begin{proof}
We use a superscript $K$ or $f^*K$ to distinguish between the objects which are defined above for $K$ respectively $f^*K$.
We have $C^{f^*K}_t=C^K_{f_*t}$ and $F^{f^*K}_t=F^K_{f_*t}$ for $t$ an object of $\caE_m$. Thus the claim follows on objects. 
The definitions of the two functors on morphisms also coincide, thus the claim follows.
\end{proof}

For a category $I$ we let $\hat{I}$ be the subcategory of $I \times \naturals$ (where $\naturals$ is a category in the usual way)
which has all objects and where a map
$(A,n) \to (B,m)$ belongs to $\hat{I}$ if and only if the map $A \to B$ is the identity or if $m>n$. Note
that a composition of non-identity maps is again a non-identity map in $\hat{I}$.

We let a map $(A,n) \to (B,m)$ in $\hat{I}$ be a weak equivalence if and only if the map $A \to B$ is the identity.
We have a canonical projection functor $p \colon \hat{I} \to I$.

\begin{proposition}
\label{gvshzu6}
For any category $I$ the canonical functor $L^H \hat{I} \to I$ is a weak equivalence of simplicial categories.
\end{proposition}

\begin{proof}
We use Lemma \ref{gfvre5}. Let $\caC$ be a model category and let $\caC^{\hat{I}}$ be equipped with the projective model
structure (which exists since $\hat{I}$ has the structure of a direct category).
Let $D \colon \hat{I} \to \caC$ be a cofibrant diagram which preserves weak equivalences. For $i \in I$ the diagram $D|_{p/i}$ is also cofibrant by
\cite[Lemma 4.2]{spitzweck.fund} (it is not used here that $\caC^I$ also should
have a model structure). The full subcategory $J$ comprised by the $((i,n),p((i,n)) \overset{\id}{\longrightarrow} i)$
in $p/i$ is homotopy right cofinal, thus $\colim (D|_{p/i}) \simeq \hocolim (D|_J)$ from which it follows that
$D \to r(l(D))$, where $l$ is the left adjoint to $r \colon \caC^I \to \caC^{\hat{I}}$, is a weak equivalence.
\end{proof}

Let $\caN$ be the nerve of $\hat{\caS}$ and $\pi$ the nerve of the map $p$ from
above. For any $K \in\caN_n$ we let $f_K \colon [n] \to [n']$ be the unique epimorphism
in $\bigtriangleup$ such that $K=f_K^*(K')$ with $K' \in \caN_{n'}$ non-degenerate. $K'$ is then also uniquely determined.
We let $\beta_K$ be the composition
$$\xymatrix{\caE_n \ar[r]^{f_{K,*}} & \caE_{n'} \ar[r]^(.28){\tilde{\alpha}_{\pi(K')}} &
\Cpx(\Sh(S_\Zar,\integers)).}$$

The reason for introducing $\hat{\caS}$ is the following observation.

\begin{lemma}
\label{gvvd3ee}
Let $h \colon [m] \to [n]$ be a map in $\bigtriangleup$ and $K \in \caN_n$. Then the composition $\caE_m \overset{h_*}{\longrightarrow}
\caE_n \overset{\beta_K}{\longrightarrow}
\Cpx(\Sh(S_\Zar,\integers))$ is equal to $\beta_{h^*K}$. 
\end{lemma}

\begin{proof}
Since every composition of non-identity maps in $\hat{\caS}$ is a non-identity map
we have a commutative diagram
$$\xymatrix{[m] \ar[r]^h \ar[d]^{f_{h^*K}} & [n] \ar[d]^{f_K} \\
[m'] \ar[r] & [n']}$$
where the bottom horizontal map is a monomorphism.
Thus the claim follows from Lemma \ref{4nhzg} and the definition of the maps $\beta_K$ and $\beta_{h^*K}$.
\end{proof}

Let $\Gamma \colon \Cpx(\Sh(S_\Zar,\integers)) \to \Cpx(\Ab)$ be a fibrant replacement functor followed
by the global sections functor.
We denote by $qi$ the subcategory of quasi isomorphisms of $\Cpx(\Ab)$.
By Lemma \ref{grfe4d} we get for any $K \in \caN_n$ induced functors $L^H(\Gamma \circ \beta_K) \colon L^H \caE_n \to L^H_{qi} \Cpx(\Ab)$
which are compatible with maps in $\bigtriangleup$ by Lemma \ref{gvvd3ee}.

Let $q_\bullet \colon Q_\bullet \to L^H \caE$ be a map between cosimplicial objects in $\sCat$.
For $K \in \caN_n$ let $\gamma_K:=L^H(\Gamma \circ \beta_K) \circ q_n$. Then the $\gamma_K$ are again
compatible with maps in $\bigtriangleup$.
By a coend construction we can pair a simplicial set $L$ and any cosimplicial object $P_\bullet$ in $\sCat$ to obtain an object of $\sCat$ which we denote by
$\caD^L_{P_\bullet}$.
In the case $L=\caN$ we just write $\caD_{P_\bullet}$. If $L$ is the nerve of a category $C$ we let $\caD^C_{P_\bullet}:=\caD^L_{P_\bullet}$.
We have $\caD_{[\bullet]} \cong \hat{\caS}$, thus we have a natural map $\caD_{Q_\bullet} \to \hat{\caS}$.

\begin{lemma}
If $Q_\bullet$ is Reedy cofibrant and the map $Q_\bullet \to [\bullet]$ is
a weak equivalence then the natural map $\caD_{Q_\bullet} \to \hat{\caS}$ is a weak equivalence of simplicial categories.
\end{lemma}
\begin{proof}
One deduces the result from the analogous statement for the usual adjunction between simplicial sets and simplicial categories
involving the simplicial nerve functor, \cite[Theorem 2.2.0.1]{lurie.HTT}.
\end{proof}

From now on suppose that $Q_\bullet$ is Reedy cofibrant and that the map
$Q_\bullet \to [\bullet]$ is a weak equivalence (which can always be achieved by a cofibrant replacement
of $L^H\caE$ in $\sCat^\bigtriangleup$).
The compatible maps $\gamma_K$ give rise to an induced map $\gamma \colon \caD_{Q_\bullet} \to L^H_{qi} \Cpx(\Ab)$.

\begin{lemma}
The map $\gamma$ gives rise to a diagram $\gamma' \in \Ho (\Cpx(\Ab)^{\hat{\caS}})$ which is well-defined up
to canonical isomorphism.
\end{lemma}
\begin{proof}
This follows from a strictification result, see \cite[Proposition 4.2.4.4]{lurie.HTT}.
\end{proof}

We define the motivic complex $\caM(r)$ to be the push forward of $\gamma'[-2r]$ with respect to
the composition $\Ho (\Cpx(\Ab)^{\hat{\caS}}) \to \Ho(\Cpx(\Ab)^{\caS}) \to \D(\Sh(\Sm_{S,\Zar},\integers))$,
where the first map is induced by $p \colon \hat{\caS} \to \caS$ and the second map is the Zariski localization map.

\subsection{Properties of the motivic complexes}
\label{5gffgrw}

Let $\caC, \caS'$ be categories and $I$ a small category.
Let $\caE'$ be a cosimplicial object in $\sCat$ over $[\bullet]$.
Let for any $n$-simplex $K$ of the nerve of $\caS'$ be
a functor $\alpha_K \colon \caE_n' \times I \to \caC$ be given. Suppose these functors
are compatible for monomorphisms in $\bigtriangleup$, i.e. that for $f \colon [m] \to [n]$ a monomorphism
we have $\alpha_K \circ (f_* \times \mathrm{id})= \alpha_{f^* K}$. Then for $\tilde{K}$ a $n$-simplex of the nerve of $\caS' \times I$
we let $T(\alpha)_{\tilde{K}}$ be the composition $\caE_n' \to \caE_n' \times I \overset{\alpha_K}{\longrightarrow} \caC$,
where the second component of the first map is the composition $\caE_n' \to [n] \to I$ (the second map being
the second component of $\tilde{K}$) and where $K$ is the first component of $\tilde{K}$.
The $T(\alpha)_{\tilde{K}}$ are then again compatible for monomorphisms in $\bigtriangleup$.

Let $p \colon \caS'' \to \caS' \times I$ be a functor and suppose that the composition in $\caS''$ of two
non-identity maps is a non-identity map.
Let $K$ be a $n$-simplex of the nerve of $\caS''$. Let $f \colon [n] \to [n']$ be the unique epimorphism in
$\bigtriangleup$ such that $K=f^*(K')$ for a non-degenerate $n'$-simplex $K'$.
Let $T^p(\alpha)_K$ be the composition $\caE_n' \overset{f_*}{\longrightarrow} \caE_{n'}'
\overset{T(\alpha)_{\tilde{K}}}{\longrightarrow} \caC$, where $\tilde{K}$ is the image of $K'$ in the nerve
of $\caS' \times I$. Then the $T^p(\alpha)_K$ are compatible for all maps in $\bigtriangleup$.

In our applications $\caS''$ will be $\hat{\caS'} \times I$.

\vskip.4cm

\subsubsection{Comparison to flat maps}
\label{hht55rd}

Let the notation be as in the last section.
We denote by $\caS^\fl$ the subcategory of $\caS$ consisting of flat maps.

Let $K$ be a $n$-simplex in the nerve of $\caS^\fl$. In particular we have a chain
$A_0 \to \cdots \to A_n$ of smooth $D$-algebras where each map is flat.
We associate to this the functor $\alpha_K' \colon [n] \to \Cpx(\Sh(S_\Zar,\integers))$
which sends $i$ to $(U \mapsto z^r(\Spec(A_i) \times_S U))$ and where the maps are induced by flat pullback of cycles.
We denote by $\tilde{\alpha}_K'$ the composition $\caE_n
\overset{\varphi_n}{\longrightarrow} [n] \overset{\alpha_K'}{\longrightarrow} \Cpx(\Sh(S_\Zar,\integers))$.

Recall the maps $\tilde{\alpha}_K$.
We have a natural transformation $\tilde{\alpha}_K' \to \tilde{\alpha}_K$ which is induced by the maps
$A_i \to C_t$ for $t = (A,B,i) \in \caE_n$. We note that the cycle conditions given by the
$F_t$ are fulfilled since for a map $t \to s$ in $\caE_n$ with $s=(A',B',j)$ the diagram
$$\xymatrix{C_t \ar[r] & C_s \\
A_i \ar[u] \ar[r] & A_j \ar[u]}$$
commutes.

We denote by $\alpha_K \colon \caE_n \times [1] \to \Cpx(\Sh(S_\Zar,\integers))$ the functor corresponding to this
natural transformation.


Thus as in the beginning of section \ref{5gffgrw} we get a compatible family of maps
$T^p(\alpha)_K$, where $p$ is the functor $\hat{\caS}^\fl \times [1] \to \caS^\fl \times [1]$.


For $K$ a $n$-simplex in the nerve of $\hat{\caS}^\fl \times [1]$ let $\tilde{\gamma}_K:=L^H(\Gamma \circ T^p(\alpha)_K) \circ q_n$.

The $\tilde{\gamma}_K$ glue to give a map $$\tilde{\gamma} \colon \caD^{\hat{\caS}^\fl \times [1]}_{Q_\bullet} \to L^H_{qi} \Cpx(\Ab).$$


We denote by $\tilde{\gamma}' \in \Ho(\Cpx(\Ab)^{\hat{\caS}^\fl \times [1]})$
the diagram canonically associated to $\tilde{\gamma}$.

\begin{lemma}
\label{6grth}
$\gamma'|_{\hat{\caS}^\fl}$ and $\tilde{\gamma}'|_{\hat{\caS}^\fl \times \{1\}}$ are canonically isomorphic.
\end{lemma}

\begin{proof}
This follows by construction of $\gamma'$ and $\tilde{\gamma}'$.
\end{proof}

\begin{lemma}
\label{ghnjs4}
$\tilde{\gamma}'|_{\hat{\caS}^\fl \times \{0\}}$ is canonically isomorphic to the diagram on $\hat{\caS}^\fl$ which associates to an $(A,n,(a_1,\ldots,a_n),m)$ the
cycle complex $z^r(A)$.
\end{lemma}

\begin{proof}
This follows by construction of $\tilde{\gamma}'$.
\end{proof}

Let $\Sm_S^\fl$ be the subcategory of $\Sm_S$ of flat maps.

\begin{corollary}
\label{nderfd}
The complex $\caM(r)|_{\Sm_S^\fl}$ is canonically isomorphic to the diagram $X \mapsto z^r(X)[-2r]$ in
$\D(\Sh(\Sm_{S,\Zar}^\fl,\integers))$.
\end{corollary}

\begin{proof}
This follows from Lemmas \ref{6grth} and \ref{ghnjs4}, the fact that the map in $\Ho(\Cpx(\Ab)^{\hat{\caS}^\fl})$ associated to
$\tilde{\gamma}'$ is an isomorphism,
the fact (which follows from these Lemmas)
that the push forward of $\gamma'$ with respect to $\Ho (\Cpx(\Ab)^{\hat{\caS}}) \to \Ho(\Cpx(\Ab)^{\caS})$
has Zariski descent and from Proposition \ref{gvshzu6} (or better its proof).
\end{proof}

\begin{corollary}
For $X \in \Sm_S$ there is a canonical isomorphism
$\caM^X(r) \cong \caM(r)|_{X_\Zar}$ in $\D(\Sh(X_\Zar,\integers))$.
\end{corollary}
\begin{proof}
This follows from Corollary \ref{nderfd}.
\end{proof}

\subsubsection{Some localization triangles}
\label{h4u4r7u}

We still keep the notation of section \ref{kdtjjz}. Let $A$ be a smooth $D$-algebra.
Let $K$ be a $n$-simplex in the nerve of $\caS$. For $t \in \caE_n$ set $C_t^A:=A \otimes_D C_t$
and $F_t^A:=\{\Spec(A) \times_S a | a \in F_t\}$.
Then as in section \ref{kdtjjz} we get functors $$\alpha_K^A \colon \caE_n \to \Cpx(\Ab),
t \mapsto z^r_{F_t^A}(C_t^A),$$ and $$\tilde{\alpha}_K^A \colon \caE_n \to
\Cpx(\Sh(S_\Zar,\integers)).$$ 

Now let $a_1,\ldots,a_k \in A$ be generators of $A$.
Set $\underline{A}:=(A,k,(a_1,\ldots,a_k)) \in \caS$.
For $(A',k',(a_1',\ldots,a_{k'}')) \in \caS$
let $$\underline{A} \otimes (A',k',(a_1',\ldots,a_{k'}'))
:=(A \otimes_D A',k+k',(a_1 \otimes 1,\ldots,a_k \otimes 1, 1 \otimes a_1',
\ldots, 1 \otimes a_{k'})) \in \caS.$$ Similarly for a $n$-simplex $K$ in the nerve of
$\caS$ the $n$-simplex $\underline{A} \otimes K$ is defined.


For $K$ a $n$-simplex in the nerve of $\caS$
we have a natural transformation $\tilde{\alpha}_K^A \to \tilde{\alpha}_{\underline{A} \otimes K}$ induced
by the obvious inclusion maps of algebras.
We denote by $$\overline{\alpha}_K^A \colon \caE_n \times [1] \to \Cpx(\Sh(S_\Zar,\integers))$$ the functor
corresponding to this natural transformation. 


For $K$ a $n$-simplex in the nerve of $\hat{\caS} \times [1]$ we let $\gamma_K^A:=L^H(\Gamma \circ T^p(\overline{\alpha}^A)_K) \circ q_n$,
where $p$ is the functor $\hat{\caS} \times [1] \to \caS \times [1]$.

The $\gamma_K^A$ glue to give a map $$\gamma^A \colon \caD^{\hat{\caS} \times [1]}_{Q_\bullet} \to L^H_{qi} \Cpx(\Ab).$$


We denote by $\gamma_A' \in \Ho(\Cpx(\Ab)^{\hat{\caS} \times [1]})$
the diagram canonically associated to $\gamma^A$.

\begin{lemma}
The push forward of $\gamma_A'[-2r]|_{\hat{\caS} \times \{1\}}$ to $\D(\Sh(\Sm_{S,\Zar},\integers))$ is canonically isomorphic
to $\underline{\R \Hom}(\integers[\Spec(A)]_\Zar,\caM(r))$.
\end{lemma}
\begin{proof}
This follows from the definition of $\gamma_A'$.
\end{proof}

\begin{corollary}
\label{bdedzhe}
The push forward of $\gamma_A'[-2r]|_{\hat{\caS} \times \{0\}}$ to $\D(\Sh(\Sm_{S,\Zar},\integers))$ is canonically isomorphic
to $\underline{\R \Hom}(\integers[\Spec(A)]_\Zar,\caM(r))$.
\end{corollary}
\begin{proof}
This follows from the fact that the map in $\Ho(\Cpx(\Ab)^{\hat{\caS}})$ associated to $\gamma_A'$ is an isomorphism.
\end{proof}

Now let $f \colon A \to A'$ be a flat map to a smooth $D$-algebra $A'$, let
$a_1',\ldots, a_{k'}' \in A'$ be generators. We have functors
$$a \colon \caS \to \caS, (B,l,(b_1,\ldots,b_l)) \mapsto
(A \otimes_D B, k+l,(a_1 \otimes 1,\ldots,a_k \otimes 1, 1 \otimes b_1, \ldots, 1 \otimes b_l))$$
and $$b \colon \caS \to \caS, (B,l,(b_1,\ldots,b_l)) \mapsto
(A' \otimes_D B, k'+l,(a_1' \otimes 1,\ldots,a_{k'}' \otimes 1,
1 \otimes b_1, \ldots, 1 \otimes b_l))$$ and a natural transformation
$a \to b$ induced by $f$. We let $G  \colon \caS \times [1] \to \caS$
be the corresponding functor.

Let $K$ be a $n$-simplex in the nerve of $\caS \times [1]$.
Let $\alpha_{2,K}^f := \tilde{\alpha}_{G(K)}$.

We have a natural transformation $\tilde{\alpha}_K^A \to \tilde{\alpha}_K^{A'}$ induced
by $f$. Let $\alpha_{1,K}^f \colon \caE_n \times [1] \to \Cpx(\Sh(S_\Zar,\integers))$ be the corresponding
functor.

We have a natural transformation $\alpha_{1,K}^f \to \alpha_{2,K}^f$ induced by the natural inclusion maps of algebras.
We denote by $\overline{\alpha}_K^f \colon \caE_n \times [1]^2 \to \Cpx(\Sh(S_\Zar,\integers))$ the corresponding functor.

For $K$ a $n$-simplex in the nerve of $\hat{\caS} \times [1]^2$ we let
$\gamma_K^f:= L^H(\Gamma \circ T^p(\overline{\alpha}^f)_K) \circ q_n$,
where $p$ is the functor $\hat{\caS} \times [1]^2 \to \caS \times [1]^2$.

The $\gamma_K^f$ glue to give a map $$\gamma^f \colon \caD^{\hat{\caS} \times [1]^2}_{Q_\bullet} \to L^H_{qi} \Cpx(\Ab).$$

We denote by $\gamma_f' \in \Ho(\Cpx(\Ab)^{\hat{\caS} \times [1]^2})$
the diagram canonically associated to $\gamma^f$.

\begin{lemma}
The push forward to $\D(\Sh(\Sm_{S,\Zar},\integers))$ of the map in $\Ho(\Cpx(\Ab)^{\hat{\caS}})$ associated to
$\gamma_f'[-2r]|_{\hat{\caS} \times [1] \times \{1\}}$ is canonically isomorphic to the map
$$\underline{\R \Hom}(\integers[\Spec(A)]_\Zar, \caM(r)) \to
\underline{\R \Hom}(\integers[\Spec(A')]_\Zar, \caM(r)).$$
\end{lemma}

\begin{proof}
This follows from the definition of $\gamma_f'$.
\end{proof}

\begin{corollary}
\label{hgdtz5z}
 The push forward to $\D(\Sh(\Sm_{S,\Zar},\integers))$ of the map in $\Ho(\Cpx(\Ab)^{\hat{\caS}})$ associated to
$\gamma_f'[-2r]|_{\hat{\caS} \times [1] \times \{0\}}$ is canonically isomorphic to the map
$$\underline{\R \Hom}(\integers[\Spec(A)]_\Zar, \caM(r)) \to
\underline{\R \Hom}(\integers[\Spec(A')]_\Zar, \caM(r)).$$
\end{corollary}

\begin{proposition}
\label{bddrzhf}
Let $i \colon Z \to X$ be a closed immersion of affine schemes in $\Sm_S$ of codimension $1$
with open affine complement $U$.
Then there is an exact triangle
$$\underline{\R \Hom}(\integers[Z]_\Zar,\caM(r-1))[-2] \to
\underline{\R \Hom}(\integers[X]_\Zar,\caM(r))$$
$$\to \underline{\R \Hom}(\integers[U]_\Zar,\caM(r))
\to \underline{\R \Hom}(\integers[Z]_\Zar,\caM(r-1))[-1]$$
in $\D(\Sh(\Sm_{S,\Zar},\integers))$, where the second map is
induced by the morphism $U \to X$.
\end{proposition}

\begin{proof}
 Let $A \to A''$ be the map of function algebras corresponding to $i$ and
$A \to A'$ the map corresponding to the open inclusion $U \subset X$.

For $K$ a $n$-simplex in the nerve of $\caS$ we define a functor $\alpha_K^\smallsquare \colon \caE_n \times [1]^2
\to \Cpx(\Ab)$ by sending $(t,0,0)$ to $z^{r-1}_{F_t^{A''}}(C_t^{A''})$, $(t,1,0)$ to 
$z^r_{F_t^A}(C_t^A)$, $(t,1,1)$ to $z^r_{F_t^{A'}}(C_t^{A'})$ and $(t,0,1)$ to $0$.
Sheafification on $S$ yields a functor $\tilde{\alpha}_K^\smallsquare \colon \caE_n \times [1]^2 \to \Cpx(\Sh(S_\Zar,\integers))$.

For $K$ a $n$-simplex in the nerve of $\hat{\caS} \times [1]^2$ let
$\gamma_K^\smallsquare:=L^H(\Gamma \circ T^p(\tilde{\alpha}^\smallsquare)_K) \circ q_n$.

The $\gamma_K^\smallsquare$ glue to give a map $$\gamma^\smallsquare \colon \caD^{\hat{\caS} \times [1]^2}_{Q_\bullet} \to L^H_{qi} \Cpx(\Ab).$$
We denote by $\gamma_\smallsquare' \in \Ho(\Cpx(\Ab)^{\hat{\caS} \times [1]^2})$
the diagram canonically associated to $\gamma^\smallsquare$.

The square in $\D(\Sh(\Sm_{S,\Zar},\integers))$ associated to the push forward of $\gamma_\smallsquare'[-2r]$
is exact by \cite[Theorem 1.7]{levine.techniques}.
Moreover by Corollary \ref{bdedzhe} the entries in this square in the places
$(0,0)$, $(1,0)$ and $(1,1)$ are $\underline{\R \Hom}(\integers[Z]_\Zar,\caM(r-1))[-2]$,
$\underline{\R \Hom}(\integers[X]_\Zar,\caM(r))$ and $\underline{\R \Hom}(\integers[U]_\Zar,\caM(r))$,
and the map from entry $(1,0)$ to $(1,1)$ is the one induced by the map $U \subset X$
by Corollary \ref{hgdtz5z}.
Thus by \cite[Definition 1.1.2.11]{lurie.higheralgebra} we get the exact triangle as required.
\end{proof}

\subsubsection{The {\'e}tale cycle class map}

\label{jrtzuutre}

For $X \in \Sm_S$ and $F$ a finite set of closed immersions in $\Sm_S$ with target $X$
we denote by $c_F^r(X,n)$ the set of cycles (closed integral subschemes) of $X \times \Delta^n$ which
intersect all $Z \times Y$ with $Z \in F \cup \{X\}$ and $Y$ a face of $\Delta^n$ properly.

Let $U \subset S$ open. Let $m$ be an integer which is invertible on $U$.
Let $\mu_m^{\otimes r} \to \caG$ be an injectively fibrant
replacement in $\Cpx(\Sh(\Sm_{U,\et},\integers/m))$.

Let $X \in \Sm_U$. For $W$ a closed subset of $X$ such that each irreducible component has codimension greater
or equal to $r$ set $\caG^W(X):= \ker(\caG(X) \to \caG(X \setminus W))$.

As in \cite[12.3]{levine.schemes} there is a canonical isomorphism of $H^{2r}(\caG^W(X))$ with the free $\integers/m$-module
on the irreducible components of $W$ of codimension $r$ and the map $\tau_{\le 2r} \caG^W(X) \to H^{2r}(\caG^W(X))[-2r]$
is a quasi isomorphism.

For $F$ a finite set of closed immersions in $\Sm_U$ with target $X$ denote by
$\caG^r_F(X,n)$ the colimit of the $\caG^W(X \times \Delta^n)$ where $W$ runs through the finite unions of
elements of $c_F^r(X,n)$. The simplicial complex of $\integers/m$-modules $\tau_{\le 2r}\caG^r_F(X,\bullet)$ augments
to the simplicial abelian group $z^r_F(X,\bullet)/m[-2r]$. This augmentation is a levelwise quasi isomorphism.
We denote by $\caG^r_F(X)$ the total complex associated to the double complex which is the normalized complex
associated to $\tau_{\le 2r}\caG^r_F(X,\bullet)$. Thus we get a quasi isomorphism
$\caG^r_F(X) \to z^r_F(X)/m[-2r]$.

On the other hand we have a canonical map $\caG^r_F(X,n) \to \caG(X \times \Delta^n)$ compatible with
the simplicial structure. We denote by $\caG'(X)$ the total complex associated to the double complex
which is the normalized complex associated to $\caG(X \times \Delta^\bullet)$.
We have a canonical quasi isomorphism $\caG(X) \to \caG'(X)$ and a canonical map
$\caG^r_F(X) \to \caG'(X)$.

Thus in $\D(\integers/m)$ we get a map $$z^r_F(X)/m[-2r] \cong \caG^r_F(X) \to \caG'(X) \cong \caG(X).$$

Our next aim is to make this assignment functorial in $X$ for all maps in $\Sm_U$.
In the following we sometimes insert into the above definitions $A$ instead of $\Spec(A)$.
Let $I$ be the category $0 \leftarrow 1 \to 2 \leftarrow 3 \leftarrow 4$.
We denote by $\caS_m$ the full subcategory of $\caS$ such that $m$ is invertible in the algebras
belonging to the objects.
We use the notation of section \ref{kdtjjz}. Let $K$ be a $n$-simplex in the nerve of $\caS_m$.
We assign to $K$ the following functor $\alpha_K' \colon \caE_n \times I \to \Cpx(\Ab)$:
$(t,0) \mapsto \alpha_K(t)/m[-2r]$, $(t,1) \mapsto \caG^r_{F_t}(C_t)$, $(t,2) \mapsto \caG'(C_t)$, $(t,3) \mapsto \caG(C_t)$,
$(t,4) \mapsto \caG(A_{\varphi_n(t)})$.
Sheafifying on $U_\Zar$ yields a functor $\tilde{\alpha}_K' \colon \caE_n \times I \to \Cpx(\Sh(U_\Zar,\integers))$.
These functors are compatible for monomorphisms in $\bigtriangleup$.


For $K$ a $n$-simplex of the nerve of $\hat{\caS}_m \times I$ let $\gamma_K^I:=L^H(\Gamma \circ T^p(\tilde{\alpha}')_K) \circ q_n$,
where $p$ is the functor $\hat{\caS}_m \times I \to \caS_m \times I$.
The $\gamma_K^I$ glue to give a map $$\gamma^I \colon \caD^{\hat{\caS}_m \times I}_{Q_\bullet} \to L^H_{qi} \Cpx(\Ab).$$
We denote by $\gamma_I' \in \Ho(\Cpx(\Ab)^{\hat{\caS}_m \times I})$
the diagram canonically associated to $\gamma^I$.

The push forward of the $I$-diagram in $\Ho(\Cpx(\Ab)^{\hat{\caS}_m})$ corresponding to the diagram $\gamma_I'$
to $\D(\Sh(\Sm_{U,\Zar},\integers))$ is an $I$-diagram
of the form $(\caM(r)/m)|_U \cong \bullet \to \bullet \cong \bullet \cong \R \epsilon_* \mu_m^{\otimes r}$ which yields the
cycle class map.

Next we wish to show the compatibility of this cycle class map with the original cycle class map defined for flat morphisms.

If in the following notation a collection $F$ of closed subschemes is missing we assume that this $F$ is empty.
For $K$ a $n$-simplex in the nerve of $\caS_m^\fl$ (with the obvious notation) we define a functor $\alpha_K'' \colon \caE_n \times I \to \Cpx(\Ab)$
in the following way: $(t,0) \mapsto z^r(A_{\varphi_n(t)})[-2r]$, $(t,1) \mapsto \caG^r(A_{\varphi_n(t)})$,
$(t,2) \mapsto \caG'(A_{\varphi_n(t)})$, $(t,3), (t,4) \mapsto \caG(A_{\varphi_n(t)})$.
Sheafifying on $U$ yields a functor $\tilde{\alpha}_K'' \colon \caE_n \times I \to \Cpx(\Sh(U_\Zar,\integers))$.
There is an obvious natural transformation $\tilde{\alpha}_K'' \to \tilde{\alpha}_K'$.
We denote by $\overline{\alpha}_K \colon \caE_n \times I \times [1] \to \Cpx(\Sh(U_\Zar,\integers))$ the corresponding functor.

For $K$ a $n$-simplex of the nerve of $\hat{\caS}_m^\fl \times I \times [1]$ let $\gamma_K^{I \times [1]}:=L^H(\Gamma \circ T^p(\overline{\alpha})_K) \circ q_n$,
where $p$ is the functor $\hat{\caS}_m^\fl \times I \times [1] \to \caS_m^\fl \times I \times [1]$.
The $\gamma_K^{I \times [1]}$ glue to give a map $$\gamma^{I \times [1]} \colon \caD^{\hat{\caS}_m^\fl \times I \times [1]}_{Q_\bullet} \to L^H_{qi} \Cpx(\Ab).$$
We denote by $\gamma_{I \times [1]}' \in \Ho(\Cpx(\Ab)^{\hat{\caS}_m^\fl \times I \times [1]})$
the diagram canonically associated to $\gamma^{I \times [1]}$.

The push forward of the $I \times [1]$-diagram in $\Ho(\Cpx(\Ab)^{\hat{\caS}_m^\fl})$ corresponding to the diagram $\gamma_{I \times [1]}'$
to $\D(\Sh(\Sm_{U,\Zar}^\fl,\integers))$ is an $I \times [1]$-diagram
where the subdiagram indexed on $I \times \{0\}$ gives the old cycle class map and the subdiagram
indexed on $I \times \{1\}$ the new cycle class map restricted to flat maps. Thus the two
cycle class maps are canonically isomorphic (over flat maps).

\begin{corollary}
For $X \in \Sm_
U$ the cycle class map $\caM^X(r)/m \to \R \epsilon_* \integers/m(r)$ from section \ref{45zgf} is canonically isomorphic
to the cycle class map $(\caM(r)/m)|_U \to \R \epsilon_* \mu_m^{\otimes r}$ restricted to $X_\Zar$.
\end{corollary}

Now we also use the notation of section \ref{h4u4r7u}. We assume $m$ is invertible in $A$.
For a $n$-simplex $K$ of the nerve of $\caS_m$ we define a functor $(\alpha')_K^A \colon \caE_n \times I \to \Cpx(\Ab)$
in the following way: $(t,0) \mapsto \alpha_K^A(t)/m[-2r]$, $(t,1) \mapsto \caG^r_{F_t^A}(C_t^A)$, $(t,2) \mapsto \caG'(C_t^A)$, $(t,3) \mapsto \caG(C_t^A)$,
$(t,4) \mapsto \caG(A \otimes_D A_{\varphi_n(t)})$.
Sheafifying on $U_\Zar$ yields a functor $(\tilde{\alpha}')_K^A \colon \caE_n \times I \to \Cpx(\Sh(U_\Zar,\integers))$.
These functors are compatible for monomorphisms in $\bigtriangleup$.

We have a natural transformation $(\tilde{\alpha}')_K^A \to \tilde{\alpha}_{\underline{A} \otimes K}'$
induced by the obvious inclusion maps of algebras.
We denote by $\overline{\alpha}_K^A \colon \caE_n \times I \times [1] \to \Cpx(\Sh(U_\Zar,\integers))$ the
corresponding functor.

For $K$ a $n$-simplex of the nerve of $\hat{\caS}_m \times I \times [1]$ let $\gamma_K^{A, I \times [1]}:=L^H(\Gamma \circ T^p(\overline{\alpha}^A)_K) \circ q_n$,
where $p$ is the functor $\hat{\caS}_m \times I \times [1] \to \caS_m \times I \times [1]$.
The $\gamma_K^{A, I \times [1]}$ glue to give a map $$\gamma^{A, I \times [1]} \colon \caD^{\hat{\caS}_m \times I \times [1]}_{Q_\bullet} \to L^H_{qi} \Cpx(\Ab).$$
We denote by $\gamma_{A, I \times [1]}' \in \Ho(\Cpx(\Ab)^{\hat{\caS}_m \times I \times [1]})$
the diagram canonically associated to $\gamma^{A, I \times [1]}$.

The push forward of the $I \times [1]$-diagram in $\Ho(\Cpx(\Ab)^{\hat{\caS}_m})$ corresponding to the diagram $\gamma_{A, I \times [1]}'$
to $\D(\Sh(\Sm_{U,\Zar},\integers))$ is an $I \times [1]$-diagram
where the subdiagram indexed on $I \times \{1\}$ gives the functor
$\underline{\R \Hom}(\Spec(A),\_)$ applied to the cycle class map $(\caM(r)/m)|_U \to \R \epsilon_* \mu_m^{\otimes r}$. 

\begin{corollary}
\label{m4u53d}
The subdiagram indexed on $I \times \{0\}$ of the above $I \times [1]$ diagram in $\D(\Sh(\Sm_{U,\Zar},\integers))$
yields a map canonically isomorphic to the map $$\underline{\R \Hom}(\Spec(A),(\caM(r)/m)|_U) \to
\underline{\R \Hom}(\Spec(A),\R \epsilon_* \mu_m^{\otimes r})$$
induced by the cycle class map.
\end{corollary}

Let $i \colon Z \to X$ be a closed immersion of affine schemes in $\Sm_S$ of codimension $1$
with open affine complement $V$.
The exact triangle
$$\underline{\R \Hom}(\integers[Z]_\Zar,\caM(r-1))[-2] \to
\underline{\R \Hom}(\integers[X]_\Zar,\caM(r))$$
$$\to \underline{\R \Hom}(\integers[U]_\Zar,\caM(r))
\to \underline{\R \Hom}(\integers[Z]_\Zar,\caM(r-1))[-1]$$
in $\D(\Sh(\Sm_{S,\Zar},\integers))$ from Proposition \ref{bddrzhf} yields
an exact triangle
$$\underline{\R \Hom}(\integers[Z_U]_\Zar,(\caM(r-1)/m)|_U)[-2] \to
\underline{\R \Hom}(\integers[X_U]_\Zar,(\caM(r)/m)|_U)$$
$$\to \underline{\R \Hom}(\integers[V_U]_\Zar,(\caM(r)/m)|_U)
\to \underline{\R \Hom}(\integers[Z_U]_\Zar,(\caM(r-1)/m)|_U)[-1]$$
in $\D(\Sh(\Sm_{U,\Zar},\integers))$.
\begin{proposition}
\label{hedtt5t5t}
Let the notation be as above. Then the diagram
$$\xymatrix{\underline{\R \Hom}(\integers[Z_U]_\Zar,(\caM(r-1)/m)|_U)[-2] \ar[r] \ar[d] &
\R \epsilon_* \underline{\R \Hom}(\integers[Z_U]_\et,\mu_m^{\otimes (r-1)})[-2] \ar[d] \\
\underline{\R \Hom}(\integers[X_U]_\Zar,(\caM(r)/m)|_U) \ar[r] \ar[d] &
\R \epsilon_* \underline{\R \Hom}(\integers[X_U]_\et,\mu_m^{\otimes r}) \ar[d] \\
\underline{\R \Hom}(\integers[V_U]_\Zar,(\caM(r)/m)|_U) \ar[r] \ar[d] &
\R \epsilon_* \underline{\R \Hom}(\integers[V_U]_\et,\mu_m^{\otimes r}) \ar[d] \\
\underline{\R \Hom}(\integers[Z_U]_\Zar,(\caM(r-1)/m)|_U)[-1] \ar[r] &
\R \epsilon_* \underline{\R \Hom}(\integers[Z_U]_\et,\mu_m^{\otimes (r-1)})[-1],}$$
where the first vertical row is the exact triangle from above, the second vertical
row is the corresponding exact triangle for {\'e}tale sheaves and where
the horizontal maps are induced by the cycle class maps, commutes. 
\end{proposition}
\label{gtrf5rdd}
\begin{proof}
Let $A \to A''$ be the map of function algebras corresponding to $i$ and $A \to A'$
the map corresponding to the open inclusion $V \to X$.
We let $J$ be the category which is defined by gluing the object $(0,0)$ of $[1]^2$ to
the object $0$ of $[1]$. We call $c$ the object $1$ of $[1]$ viewed as object of $J$,
the other objects are numbered $(k,l)$, $k,l \in \{0,1\}$.
Let $\mu_m^{\otimes (r-1)} \to \tilde{\caG}$ be an injectively fibrant
replacement in $\Cpx(\Sh(\Sm_{U,\et},\integers/m))$.
We let $\caH_t(n)$ be the colimit of the $\caG^W((\Spec C_t^A) \times \Delta^n)$,
where $W$ runs through the finite unions of elements of $c^{r-1}(\Spec C_t^{A''},n)$.
We denote by $\caH_t$ the total complex associated to the double complex which is the normalized
complex associated to $\tau_{\le 2r} \caH_t(\bullet)$.
We have an absolute purity isomorphism $\varphi$ from the sheaf
$$\Sm_U \ni Y \mapsto \ker(\caG(Y \times_S X) \to \caG(Y \times_X V))$$ to
$$\Sm_U \ni Y \mapsto \tilde{\caG}(Y \times_S Z)[-2]$$ in $\D(\Sh(\Sm_{U,\et},\integers/m))$.
This can be lifted to a map of (complexes of) sheaves since the target of the map is
injectively fibrant. We denote any such lift also by $\varphi$.
We let $\tilde{\caG}^{r-1}$ and $\tilde{\caG}'$ be the analogues of
$\caG^r$ and $\caG'$ (and in the first case also with the cycle conditions).

For $K$ a $n$-simplex
in the nerve of $\caS_p$ we define, using $\varphi$, a functor
$$\alpha_K^\heartsuit \colon \caE_n \times I \times J \to \Cpx(\Ab)$$
by sending

$(t,0,c)$ to $z^{r-1}_{F_t^{A''}}(C_t^{A''})[-2r]$, $(t,0,(0,0))$ to $z^{r-1}_{F_t^{A''}}(C_t^{A''})[-2r]$,

$(t,0,(1,0))$ to
$z^r_{F_t^A}(C_t^A)[-2r]$, $(t,0,(1,1))$ to $z^r_{F_t^{A'}}(C_t^{A'})[-2r]$,
$(t,0,(0,1))$ to $0$, 

$(t,1,c)$ to $\tilde{\caG}^{r-1}_{F_t^{A''}}(C_t^{A''})[-2]$,
$(t,1,(0,0))$ to $\caH_t$,
$(t,1,(1,0))$ to $\caG^r_{F_t^A}(C_t^A)$, $(t,1,(1,1))$ to $\caG^r_{F_t^{A'}}(C_t^{A'})$,
$(t,1,(0,1))$ to $0$,

$(t,2,c)$ to $\tilde{\caG}'(C_t^{A''})[-2]$,
$(t,2,(0,0))$ to $\ker(\caG'(C_t^A) \to \caG'(C_t^{A'}))$,

$(t,2,(1,0))$ to $\caG'(C_t^A)$, $(t,2,(1,1))$ to $\caG'(C_t^{A'})$,
$(t,2,(0,1))$ to $0$,

$(t,3,c)$ to $\tilde{\caG}(C_t^{A''})[-2]$,
$(t,3,(0,0))$ to $\ker(\caG(C_t^A) \to \caG(C_t^{A'}))$,
$(t,3,(1,0))$ to $\caG(C_t^A)$, $(t,3,(1,1))$ to $\caG(C_t^{A'})$,
$(t,3,(0,1))$ to $0$,

$(t,4,c)$ to $\tilde{\caG}(A'' \otimes_D A_{\varphi_n(t)})[-2]$,
$(t,4,(0,0))$ to $\ker(\caG(A \otimes_D A_{\varphi_n(t)}) \to \caG(A' \otimes_D A_{\varphi_n(t)}))$,
$(t,4,(1,0))$ to $\caG(A \otimes_D A_{\varphi_n(t)})$,
$(t,4,(1,1))$ to $\caG(A' \otimes_D A_{\varphi_n(t)})$ and
$(t,4,(0,1))$ to $0$.

Sheafifying we obtain a functor $\tilde{\alpha}_K^\heartsuit \colon \caE_n \times I \times J \to \Cpx(\Sh(S_\Zar,\integers))$.
For $K$ a $n$-simplex in the nerve of $\hat{\caS}_m \times I \times J$ we let
$\gamma_K^\heartsuit:= L^H(\Gamma \circ T^p(\tilde{\alpha}^\heartsuit)_K) \circ q_n$, where $p$ is the
functor $\hat{\caS}_m \times I \times J \to \caS_m \times I \times J$.

The $\gamma_K^\heartsuit$ glue to give a map 
$$\gamma^\heartsuit \colon \caD_{Q_\bullet}^{\hat{\caS}_m \times I \times J} \to L^H_{qi} \Cpx(\Ab).$$

We denote by $\gamma_\heartsuit' \in \Ho(\Cpx(\Ab)^{\hat{\caS}_m \times I \times J})$ the diagram canonically
associated to $\gamma^\heartsuit$.

The commutativity of the push forward of the corresponding $I \times J$-diagram in $\Ho(\Cpx(\Ab)^{\hat{S}_m})$ to
$\D(\Sh(\Sm_{U,\Zar},\integers))$ shows the claim, using Corollary \ref{m4u53d}.

\end{proof}

\subsection{The naive $\mathbb{G}_m$-spectrum}
\label{gtre5rdd}

\begin{proposition}
\label{ht4ztr}
There is a canonical isomorphism $$\caM(r-1)[-1] \cong
\underline{\R \Hom}(\integers[\GmS,\{1\}]_\Zar,\caM(r))$$ in $\D(\Sh(\Sm_{S,\Zar},\integers))$.
\end{proposition}

\begin{proof}
By Proposition \ref{bddrzhf} there is an exact triangle
$$\caM(r-1)[-2] \to \underline{\R \Hom}(\integers[\mathbb{A}^1_S]_\Zar,\caM(r))
\to \underline{\R \Hom}(\integers[\GmS]_\Zar,\caM(r)) \to \caM(r-1)[-1].$$

There is a split $\underline{\R \Hom}(\integers[\GmS]_\Zar,\caM(r)) \to 
\underline{\R \Hom}(\integers[\mathbb{A}^1_S]_\Zar,\caM(r))$ induced by
$\{1\} \subset \GmS$ and the $\mathbb{A}^1$-invariance of $\caM(r)$.
This induces the required isomorphism.
\end{proof}

We thus get a naive $\integers[\GmS,\{1\}]_\Zar$-spectrum $\caM$ (in the sense of \cite[I 6]{riou.these}) in
$\D(\Sh(\Sm_{S,\Zar},\integers))$ with entry $\caM(r)[r]$ in level $r$.
We also denote a lift of $\caM$ to the homotopy category of $\integers[\GmS,\{1\}]_\Zar$-spectra
by $\caM$. If we want to emphasize the dependence of $\caM$ on $S$ we write
$\caM_S$ instead of $\caM$.

\section{Motivic complexes over a field}
\label{hjrerrt}

We first note that the material from section \ref{h5643dfgf} carries over verbatim
to the case of smooth schemes over a field $k$, except that we do not have to use
the functor $\Gamma$ in the constructions. We denote the resulting motivic complexes
in $\D(\Sh(\Sm_{k,\Zar},\integers))$ by $\caM(r)_k$. The resulting naive $\mathbb{G}_m$-spectrum
is denoted by $\caM_k$, the same notation is used for a lift to a spectrum.
In this section we will use the notation of section \ref{h5643dfgf} (like $\caS$, $C_t^A$ etc.) carried over to the field case.

We let $\tilde{z}^r(X)=C_*(z_\equi(\mathbb{A}^r,0))(X)$ (for notation see e.g.
\cite{mazza-voevodsky-weibel}) be the complex introduced by
Friedlander and Suslin (\cite{friedlander-suslin}), so $\tilde{z}^r \in \Cpx(\Sh(\Sm_{k,\Zar},\integers))$
and the Zariski hypercohomology of $\tilde{z}^r$ computes Bloch's higher Chow groups (see loc. cit.).

Let $A=k[T_1,\ldots,T_r]$. For $K$ a $n$-simplex in the nerve of $\caS$
(with corresponding chain $A_0 \to \cdots \to A_n$ of $k$-algebras) we define a functor
$\alpha_K^e \colon \caE_n \times [1] \to \Cpx(\Ab)$ by sending $(t,0)$ to $\tilde{z}^r(A_{\varphi_n(t)})$ and
$(t,1)$ to $z^r_{F_t^A}(C_t^A)$.

For $K$ a $n$-simplex in the nerve of $\hat{\caS} \times [1]$ let $\gamma^e_K := L^H(T^p(\alpha^e)_K)  \circ q_n$.
The $\gamma^e_K$ glue to give a map
$$\gamma^e \colon \caD^{\hat{\caS} \times [1]}_{Q_\bullet} \to L^H_{qi} \Cpx(\Ab).$$
We denote by $\gamma_e' \in \Ho(\Cpx(\Ab)^{\hat{\caS} \times [1]})$ the diagram canonically associated to $\gamma^e$.

The map in $\D(\Sh(\Sm_{k,\Zar},\integers))$ associated to the push forward of $\gamma_e'[-2r]$ is an isomorphism.
Moreover the target is canonically isomorphic to $\caM(r)_k$. We get the

\begin{proposition}
\label{h46tef}
The complexes $\tilde{z}^r[-2r]$ and $\caM(r)_k$ are canonically isomorphic in $\D(\Sh(\Sm_{k,\Zar},\integers))$.
\end{proposition}

Pairing of cycles gives us pairings $\tilde{z}^r \otimes \tilde{z}^{r'} \to \tilde{z}^{r+r'}$
(involving the Eilenberg-Zilber map), using Proposition \ref{h46tef}
this gives us pairings
\begin{equation}
\label{h5rrgtt}
\caM(r)_k \otimes^\bL \caM(r')_k \to \caM(r+r')_k
\end{equation}
in $\D(\Sh(\Sm_{k,\Zar},\integers))$ which are unital, associative and commutative.
Using the diagonal we obtain for any $X \in \Sm_k$ pairings
$$\underline{\R \Hom}(\integers[X]_\Zar,\caM(r)_k) \otimes^\bL \underline{\R \Hom}(\integers[X]_\Zar,\caM(r')_k)
\to \underline{\R \Hom}(\integers[X]_\Zar,\caM(r+r')_k)$$ such that the map
$$\caM(\bullet)_k \to \underline{\R \Hom}(\integers[X]_\Zar,\caM(\bullet)_k)$$
is a map of $\naturals$-graded algebras.

We derive an action 
\begin{equation}
\label{ghu54}
\caM(r)_k \otimes^\bL \underline{\R \Hom}(\integers[X]_\Zar,\caM(r')_k)
\to \underline{\R \Hom}(\integers[X]_\Zar,\caM(r+r')_k).
\end{equation}

In order to achieve a compatibility between the localization triangle (Proposition \ref{bddrzhf}) and
this action we study an action of $\tilde{z}^r$ directly on Bloch's complexes which appear in the
definition of the $\caM(r')_k$.

Let $A$ be a smooth $k$-algebra and let $A':=A[T_1,\ldots,T_r]$.
For $K$ a $n$-simplex in the nerve of $\caS$
(with corresponding chain $A_0 \to \cdots \to A_n$ of $k$-algebras) we define a functor
$\alpha_K^a \colon \caE_n \times [1] \to \Cpx(\Ab)$ by sending $(t,0)$ to $\tilde{z}^r(A_{\varphi_n(t)}) \otimes z^{r'}_{F_t^A}(C_t^A)$ and
$(t,1)$ to $z^{r+r'}_{F_t^{A'}}(C_t^{A'})$ (the transition maps from $0$ to $1$ are induced by pairing of cycles).

For $K$ a $n$-simplex in the nerve of $\hat{\caS} \times [1]$ let $\gamma^a_K := L^H(T^p(\alpha^a)_K)  \circ q_n$.
The $\gamma^a_K$ glue to give a map
$$\gamma^a \colon \caD^{\hat{\caS} \times [1]}_{Q_\bullet} \to L^H_{qi} \Cpx(\Ab).$$
We denote by $\gamma_a' \in \Ho(\Cpx(\Ab)^{\hat{\caS} \times [1]})$ the diagram canonically associated to $\gamma^a$.

The map in $\D(\Sh(\Sm_{k,\Zar},\integers))$ associated to the push forward of $\gamma_a'[-2r-2r']$ yields an action map
\begin{equation}
\label{h5th4gh}
\caM(r)_k \otimes^\bL \underline{\R \Hom}(\integers[X]_\Zar,\caM(r')_k)
\to \underline{\R \Hom}(\integers[X]_\Zar,\caM(r+r')_k),
\end{equation}
where $X=\Spec(A)$.
We will show in Lemma \ref{ghh4rfrf} that the action maps
(\ref{ghu54}) and (\ref{h5th4gh}) coincide.

\begin{lemma}
\label{h453hss}
The pairing (\ref{h5th4gh}) for $A=k$ coincides with the pairing (\ref{h5rrgtt}).
\end{lemma}

\begin{proof}
Let $A:=k[T_1,\ldots,T_{r'}]$ and $A':=A[T_1,\ldots,T_r]$.
For $K$ a $n$-simplex in the nerve of $\caS$
(with corresponding chain $A_0 \to \cdots \to A_n$ of $k$-algebras) we define a functor
$\alpha_K^c \colon \caE_n \times [1]^2 \to \Cpx(\Ab)$ by sending $(t,0,0)$ to 
$\tilde{z}^r(A_{\varphi_n(t)}) \otimes \tilde{z}^{r'}(A_{\varphi_n(t)})$,
$(t,0,1)$ to $\tilde{z}^{r+r'}(A_{\varphi_n(t)})$,
$(t,1,0)$ to $\tilde{z}^r(A_{\varphi_n(t)}) \otimes z^{r'}_{F_t^A}(C_t^A)$ and
$(t,1,1)$ to $z^{r+r'}_{F_t^{A'}}(C_t^{A'})$.

For $K$ a $n$-simplex in the nerve of $\hat{\caS} \times [1]^2$ let $\gamma^c_K := L^H(T^p(\alpha^c)_K)  \circ q_n$.
The $\gamma^c_K$ glue to give a map
$$\gamma^c \colon \caD^{\hat{\caS} \times [1]^2}_{Q_\bullet} \to L^H_{qi} \Cpx(\Ab).$$
We denote by $\gamma_c' \in \Ho(\Cpx(\Ab)^{\hat{\caS} \times [1]^2})$ the diagram canonically associated to $\gamma^c$.

The commutativity of the square in $\D(\Sh(\Sm_{k,\Zar},\integers))$ associated to the push forward of $\gamma_c'[-2r-2r']$
shows the claim for the $A$ under consideration.
The claim for $A=k$ is shown in a similar manner.
\end{proof}

\begin{lemma}
\label{ghh4rfrf}
For affine $X$ the action maps
(\ref{ghu54}) and (\ref{h5th4gh}) coincide.
\end{lemma}

\begin{proof}
Let $X=\Spec(A)$ and let $A':=A[T_1,\ldots,T_r]$.
For $K$ a $n$-simplex in the nerve of $\caS$
(with corresponding chain $A_0 \to \cdots \to A_n$ of $k$-algebras) we define a functor
$\alpha_K^g \colon \caE_n \times [2] \to \Cpx(\Ab)$ by sending $(t,0)$ to $\tilde{z}^r(A_{\varphi_n(t)}) \otimes z^{r'}_{F_t^A}(C_t^A)$,
$(t,1)$ to $\tilde{z}^r(A \otimes_k A_{\varphi_n(t)}) \otimes z^{r'}_{F_t^A}(C_t^A)$ and
$(t,2)$ to $z^{r+r'}_{F_t^{A'}}(C_t^{A'})$.

For $K$ a $n$-simplex in the nerve of $\hat{\caS} \times [2]$ let $\gamma^g_K := L^H(T^p(\alpha^g)_K)  \circ q_n$.
The $\gamma^g_K$ glue to give a map
$$\gamma^g \colon \caD^{\hat{\caS} \times [2]}_{Q_\bullet} \to L^H_{qi} \Cpx(\Ab).$$
We denote by $\gamma_g' \in \Ho(\Cpx(\Ab)^{\hat{\caS} \times [2]})$ the diagram canonically associated to $\gamma^g$.

Using methods as in the beginning of section \ref{h4u4r7u}
one shows that the second map in the diagram $[2] \to \Ho(\Cpx(\Ab)^\caS)$ associated to $\gamma_g'$ 
is $\underline{\R \Hom}([X],f)$, where $f$ is the map induced by $\gamma_a'$ with $A$ the $k$-algebra $k$.
The composite map associated to $\gamma_g'$ gives the action map (\ref{h5th4gh}),
thus the claim follows from Lemmma \ref{h453hss}.
\end{proof}

\begin{proposition}
\label{bh5z6u5}
Let $i \colon Z \to X$ be a closed immersion of affine schemes in $\Sm_k$ of codimension $1$ with
open affine complement $U$. Then the diagram
$$\xymatrix{
\caM(r)_k \otimes^\bL \underline{\R \Hom}([Z]_\Zar,\caM(r'-1)_k)[-2] \ar[r] \ar[d] &
\underline{\R \Hom}([Z]_\Zar,\caM(r+r'-1)_k)[-2] \ar[d] \\
\caM(r)_k \otimes^\bL \underline{\R \Hom}([X]_\Zar,\caM(r')_k) \ar[r] \ar[d] &
\underline{\R \Hom}([X]_\Zar,\caM(r+r')_k) \ar[d] \\
\caM(r)_k \otimes^\bL \underline{\R \Hom}([U]_\Zar,\caM(r')_k) \ar[d] \ar[r] &
\underline{\R \Hom}([U]_\Zar,\caM(r+r')_k) \ar[d] \\
\caM(r)_k \otimes^\bL \underline{\R \Hom}([Z]_\Zar,\caM(r'-1)_k)[-1] \ar[r] &
\underline{\R \Hom}([Z]_\Zar,\caM(r+r'-1)_k)[-1],}$$
where the horizontal maps are the above action maps and the columns
are the triangles from Proposition \ref{bddrzhf},
commutes.
\end{proposition}

\begin{proof}
For $K$ a $n$-simplex in the nerve of $\caS$ one defines a functor $\caE_n \times [1] \times [1]^2 \to \Cpx(\Ab)$
combining the action maps from above and the functors used in the proof of Proposition \ref{bddrzhf}.
The commutativity of the diagram associated to the resulting functor
$\caD^{\hat{\caS} \times [1] \times [1]^2}_{Q_\bullet} \to L^H_{qi} \Cpx(\Ab)$ shows the claim.
\end{proof}

The isomorphism
$$\integers[-1] \cong \caM(0)_k[-1] \cong \underline{\R \Hom}(\integers[\Gmk,\{1\}]_\Zar,\caM(1)_k)$$
in $\D(\Sh(\Sm_{k,\Zar},\integers))$ from Proposition \ref{ht4ztr}
induces a map $$\iota_1 \colon \integers[\Gmk,\{1\}]_\Zar[-1] \to \caM(1)_k.$$

\begin{lemma}
\label{gh4tzred}
The composition $$\integers[\Gmk,\{1\}]_\Zar[-1] \otimes^\bL \caM(r-1)_k \to \caM(1)_k \otimes^\bL \caM(r-1)_k
\to \caM(r)_k,$$
where the first map is induced by $\iota_1$ and the second map is the above multiplication,
is adjoint to the isomorphism from Proposition \ref{ht4ztr}.
\end{lemma}

\begin{proof}
This follows from Proposition \ref{bh5z6u5}.
\end{proof}

Recall the isomorphisms 
\begin{equation}
\label{h46zrde}
\tilde{z}^r[-2r] \cong C_*(\integers_\tr((\Gmk,\{1\})^{\wedge r}))[-r]
\end{equation}
in $\D(\Sh(\Sm_{k,\Zar},\integers))$ constructed in \cite{voevodsky.allagree}.
We get a natural map $$\iota_2 \colon \integers[\Gmk,\{1\}]_\Zar[-1] \to C_*(\integers_\tr((\Gmk,\{1\})))[-1] \cong
\tilde{z}^1[-2] \cong \caM(1)_k.$$

\begin{lemma}
\label{ght55rfr}
The maps $\iota_1$ and $\iota_2$ agree.
\end{lemma}

\begin{proof}
Let $i_1 \colon \Gmk \to \bA^1_k$ be the natural inclusion and $i_2 \colon \Gmk \to \bA^1_k$ the inversion
followed by the natural inclusion. Let $Q$ be the sheaf cokernel of the map $$C_*(\integers_\tr((\Gmk,\{1\})))
\overset{i_1 \ominus i_2}{\longrightarrow} C_*(\integers_\tr((\bA^1_k,\{1\}))) \oplus C_*(\integers_\tr((\bA^1_k,\{1\}))).$$
Since this map is injective and the target is acyclic $Q$ is a representative of
the shifted complex  $C_*(\integers_\tr((\Gmk,\{1\})))[1]$.

For $X \in \Sm_k$ and maps $f,g \colon X \to \bA^1_k$ let $h(f,g)$ be the map
$X \times \Delta^1 \to \bA^1_k$ given by $sf + (1-s)g$, where $s$ is the standard coordinate
on the algebraic $1$-simplex $\Delta^1$.

Let $c \colon \Gmk \to \bA^1_k$ be the constant map to $1$.
Let $\varphi \in C_1(\integers_\tr((\bA^1_k,\{1\})))(\Gmk)$ be given by $h(i_1,c)$,
in a similar manner let $\psi$ be given by $-h(i_2,c)$. 
Then $$\partial (\varphi,\psi) \in C_0(\integers_\tr((\bA^1_k,\{1\})))(\Gmk) \oplus C_0(\integers_\tr((\bA^1_k,\{1\})))(\Gmk)$$
is the image of $\overline{\mathrm{id}}_{\Gmk} \in C_0(\integers_\tr((\Gmk,\{1\})))(\Gmk)$ with respect to the map
$i_1 \ominus i_2$.
Thus the canonical map $\integers[\Gmk,\{1\}]_\Zar[1] \to Q$ is represented by the image of $(\varphi,\psi)$ in
$Q_1(\Gmk)$.

Note there is a canonical map $Q \to C_*(\integers_\tr((\P^1_k,\{1\})))$ which is induced by the two canonical covering
maps $\bA^1_k \to \P^1_k$.
We denote the image of $(\varphi,\psi)$ in the group $C_1(\integers_\tr((\P^1_k,\{1\})))(\Gmk)$ by $\eta$.
Thus $\eta$ induces a map
\begin{equation}
\label{h4eztwer}
\integers[\Gmk,\{1\}]_\Zar[1] \to C_*(\integers_\tr((\P^1_k,\{1\}))).
\end{equation}
The comparison isomorphism (\ref{h46zrde}) is constructed using the natural map
$$C_*(\integers_\tr((\P^1_k,\{1\}))) \to C_*(z_\equi(\P^1_k \setminus \{1\},0)) \cong C_*(z_\equi(\bA^1_k,0)),$$
and precomposition with (\ref{h4eztwer}) gives the map $\iota_2$ (modulo the identification $\tilde{z}^1[-2] \cong \caM(1)_k$
and a shift).
Let us denote the image of $\eta$ in $z^1(\Gmk \times_k \bA^1_k,1)$ by $\eta'$. The cycle (in the sense of
homological algebra) $\eta'$ is a sum
$\varphi'+\psi'$ of chains (where each summand is a chain constituted
by an algebraic cycle or the negative thereof). Here $\varphi'$ (resp. $\psi'$) is the image of $\varphi$ (resp. $\psi$).
We want to compute the boundary of $\eta'$ for the triangle defined by the sequence
\begin{equation}
\label{gtu6z4r}
z^0(\{0\} \times_k \bA^1_k) \to z^1(\bA^1_k \times_k \bA^1_k) \to z^1(\Gmk \times_k \bA^1_k).
\end{equation}
Therefore we lift $\eta'$ to the middle complex, take the boundary and view it as an element of the left complex.
We first give a lift of $\varphi'$.

Let $c' \colon \bA^1_k \to \bA^1_k$ be the constant map to $\{1\}$. We let $\tilde{\varphi} \in
C_1(\integers_\tr((\bA^1_k,\{1\})))(\bA^1_k)$ be given by $h(\mathrm{id}_{\bA^1_k},c')$ and
$\tilde{\varphi}'$ be the image of $\tilde{\varphi}$ with respect to the composition
$$C_1(\integers_\tr((\bA^1_k,\{1\})))(\bA^1_k) \to C_1(\integers_\tr((\P^1_k,\{1\})))(\bA^1_k)
\to \tilde{z}^1(\bA^1_k,1) \to z^1(\bA^1_k \times_k \bA^1_k,1).$$
Then $\tilde{\varphi}'$ is a lift of $\varphi'$.
The boundary of the image of $\tilde{\varphi}$ in $C_1(\integers_\tr((\P^1_k,\{1\})))(\bA^1_k)$
is the graph of the canonical embedding $\bA^1_k \to \P^1_k$.

We let $t$ be the standard coordinate on $\Gmk$, $s$ the standard coordinate on $\Delta^1$ and
$[x_0:x_1]$ homogeneous coordinates on $\P^1_k$. Then the effective cycle corresponding to the image of
$-\psi$ in $C_1(\integers_\tr((\P^1_k,\{1\})))(\Gmk)$ is given by the homogeneous equation
$$sx_0+t(1-s)x_0=tx_1.$$
The closure $Z$ in $\bA^1_k \times \Delta^1 \times_k \P^1_k$ is given by the same equation.
Intersecting with $s=0$ (resp. $s=1$) gives the closed subscheme with equation
$t(x_0-x_1)=0$ (resp. $x_0=tx_1$). This shows that the intersections of this closure with the faces of $\Delta^1$ are proper.
We view the restriction of $Z$ to $\P^1_k \setminus \{1\}$ as a cycle in $z^1(\bA^1_k \times_k \bA^1_k,1)$ and
denote its negative by $\tilde{\psi}'$. Thus $\tilde{\psi}'$ is a lift of $\psi'$.

We see that the boundary of $\tilde{\varphi}'$ cancels with the contribution
of the boundary of $\tilde{\psi}'$ for $s=1$. Thus the boundary of $\tilde{\varphi}' + \tilde{\psi}'$ is
given by the equation $t=0$. It follows that the boundary of $\eta'$ for the triangle defined
by (\ref{gtu6z4r}) corresponds to $1$. 
The claim follows.
\end{proof}

\begin{theorem}
\label{ji67fhll}
The spectrum $\caM_k$ is isomorphic to the motivic Eilenberg-MacLane spectrum
$\MZ_k$ over $k$.
\end{theorem}

\begin{proof}
The isomorphisms (\ref{h46zrde}) are compatible with the product structures,
see \cite[Proposition 3.3]{kondo-yasuda} for the case of a perfect
ground field and \cite{kondo-yasuda-letter} for
the general case. Thus the claim follows
from Lemmas \ref{gh4tzred} and \ref{ght55rfr}.
\end{proof}

\section{Comparisons}
\label{gf556tf}

\subsection{The exceptional inverse image of $\caM$}

We let $x$ be a closed point of $S$, $k$ its residue field and
$i \colon \Spec(k) \hookrightarrow S$ the corresponding closed inclusion.
Set $U:= S \setminus \{x\}$, $U=\Spec(D')$, and let $j$ be the open inclusion $U \to S$.
We view $k$ as a $D$-algebra in the canonical way.

We use the notation of section \ref{h5643dfgf}. Let $K$ be a $n$-simplex in the nerve
of $\caS$. For $t \in \caE_n$ set $C_t':=D' \otimes_D C_t$ and
$F_t':=\{U \times_S a | a \in F_t\}$. Also set $C_t'':=k \otimes_D C_t$ and
$F_t'':=\{\{x\}\times_S a | a \in F_t\}$.
Define a functor $\alpha_K^! \colon \caE_n \times [1]^2 \to \Cpx(\Ab)$ by sending
$(t,0,0)$ to $z_{F_t''}^{r-1}(C_t'')$, $(t,1,0)$ to $z_{F_t}^r(C_t)$,
$(t,1,1)$ to $z_{F_t'}^r(C_t')$ and $(t,0,1)$ to $0$. Sheafification on $S$ yields a functor
$\tilde{\alpha}_K^! \colon \caE_n \times [1]^2 \to \Cpx(\Sh(S_\Zar,\integers))$.

For $K$ a $n$-simplex in the nerve of $\hat{\caS} \times [1]^2$ let
$\gamma_K^! := L^H(\Gamma \circ T^p(\tilde{\alpha}^!)_K) \circ q_n$, where $p$ is the functor
$\hat{\caS} \times [1]^2 \to \caS \times [1]^2$.

The $\gamma_K^!$ glue to give a map
$\gamma^! \colon \caD_{Q_\bullet}^{\hat{\caS} \times [1]^2} \to L_{qi}^H \Cpx(\Ab)$.

We denote by $\gamma_!' \in \Ho(\Cpx(\Ab)^{\hat{\caS} \times [1]^2})$ the diagram
canonically associated to $\gamma^!$.

The square in $\D(\Sh(\Sm_{S,\Zar},\integers))$ associated to the push forward of $\gamma_!'[-2r]$ is exact.
We thus obtain the

\begin{proposition}
\label{hgrtgz54}
There is an exact triangle
$$i_* \caM(r-1)_k[-2] \to \caM(r) \to j_*j^* \caM(r) \to i_* \caM(r-1)_k[-1]$$
in $\D^{\mathbb{A}^1}(\Sh(\Sm_{S,\Nis},\integers))$.
\end{proposition}

\begin{corollary}
\label{h46reee}
There is a canonical isomorphism $$i^!\caM(r) \cong \caM(r-1)_k[-2]$$
in $\D^{\mathbb{A}^1}(\Sh(\Sm_{k,\Nis},\integers))$.
\end{corollary}

\begin{corollary}
\label{lplp87d}
There is a canonical isomorphism of naive $\mathbb{G}_m$-spectra $$i^!\caM \cong \caM_k(-1)[-2]$$
and also such an isomorphism of spectra.
\end{corollary}

\begin{proof}
The bonding maps are the same.
\end{proof}

\begin{theorem}
\label{h47zt5e}
There is an isomorphism of spectra $$i^!\caM \cong \MZ_k(-1)[-2].$$
\end{theorem}

\begin{proof}
This follows from Corollary \ref{lplp87d} and Theorem \ref{ji67fhll}.
\end{proof}

\subsection{Pullback to the generic point}

Let $K$ be the fraction field of $D$ and $f \colon \Spec(K) \to S$ the canonical morphism.

\begin{lemma}
\label{bh464red}
There is an isomorphism $f^* \caM(r) \cong \caM(r)_K$ in $\D(\Sh(\Sm_{K,\Zar},\integers))$.
\end{lemma}

\begin{theorem}
\label{jd5ru655}
There is an isomorphism $f^* \caM \cong \MZ_K$
in $\SH(K)$.
\end{theorem}

\begin{proof}
There is an isomorphism $f^* \caM \cong \caM_K$. The result now follows from
Theorem \ref{ji67fhll}.
\end{proof}

\subsection{Weight $1$ motivic complexes}

We keep the notation of the last section.

\begin{proposition}
\label{hfhzjgg}
Let $k$ be a field, $\integers(1)=C_*(\integers_\tr((\Gmk,\{1\})))[-1]$ be the motivic complex of weight $1$
(in the notation of \cite{voevodsky.allagree} or \cite{mazza-voevodsky-weibel}). Then there is an isomorphism
$\integers(1) \cong \caO^*[-1]$ in $\D(\Sh(\Sm_{k,\Zar},\integers))$. Moreover the map
$\integers[\Gmk,\{1\}]_\Zar \to \caO^*$ induced by this map is the canonical one.
\end{proposition}

\begin{proof}
The first part is \cite[Theorem 4.1]{mazza-voevodsky-weibel}, the second part is contained
in the proof of \cite[Lemma 4.4]{mazza-voevodsky-weibel}.
\end{proof}

We denote by $S^{(1)}$ the set of codimension $1$ points of $S$, and for each $\p \in S^{(1)}$ we let
$\kappa(\p)$ be the residue field of $\p$ and $i_\p$ the corresponding inclusion $\Spec(\kappa(\p)) \to S$.

\begin{lemma}
\label{h4ehzu4}
There is an exact triangle
$$\bigoplus_{\p \in S^{(1)}} i_{\p,*} \caM(r-1)_{\kappa(\p)}[-2] \to \caM(r) \to f_* \caM(r)_K \to \bigoplus_{\p \in S^{(1)}} i_{\p,*} \caM(r-1)_{\kappa(\p)}[-1]$$
in $\D(\Sh(\Sm_{S,\Zar},\integers))$.
\end{lemma}

\begin{proof}
This follows from Corollary \ref{h46reee} and Lemma \ref{bh464red}.
\end{proof}

\begin{corollary}
\label{h4rtz55r}
We have $\caH^i(\caM(1)) \cong 0$ for $i \neq 1$ and there is an exact sequence
$$0 \to \caH^1(\caM(1)) \to f_* \caO^*_{/K} \to \bigoplus_{\p \in S^{(1)}} i_{\p,*} \integers \to 0$$
in $\Sh(\Sm_{S,\Zar},\integers)$.
\end{corollary}

\begin{proof}
This follows from Lemma \ref{h4ehzu4} and Proposition \ref{hfhzjgg}.
\end{proof}

\begin{theorem}
\label{jr5t4565}
There is a canonical isomorphism $$\caM(1) \cong \caO^*_{/S}[-1]$$ in $\D(\Sh(\Sm_{S,\Zar},\integers))$.
\end{theorem}

\begin{proof}
We have a canonical map $\integers[\GmS]_\Zar \to \caH^1(\caM(1))$ whose composition with the map
$\caH^1(\caM(1)) \to f_* \caO^*_{/K}$ is the canonical map by Proposition \ref{hfhzjgg}. The image of this
canonical map $\integers[\GmS]_\Zar \to f_* \caO^*_{/K}$ is $\caO^*_{/S}$. The claim follows now from Corollary \ref{h4rtz55r}.
\end{proof}

Let $\HBo \in \D(\Sh(\Sm_{S,\Zar},\integers))$ be the first $\bA^1$- and Nisnevich-local space in a $\Omega$-$\GmS$-spectrum model of $\HB$.

\begin{theorem}
\label{hnehztrr}
There is a canonical isomorphism $$\HBo \cong (\caO^*_{/S})_\Q$$ in $\D(\Sh(\Sm_{S,\Zar},\integers))$.
\end{theorem}

\begin{proof}
The proof is similar to the proof of Theorem \ref{jr5t4565}.
\end{proof}

\subsection{Rational spectra}
\label{hfdgjtr}

We keep the notation of the last sections.

\begin{corollary}
\label{j675gbbb}
There is an exact triangle
$$\bigoplus_{\p \in S^{(1)}} i_{\p,*} \MZ_{\kappa(\p)}(-1)[-2] \to \caM \to f_* \MZ_K \to \bigoplus_{\p \in S^{(1)}} i_{\p,*} \MZ_{\kappa(\p)}(-1)[-1]$$
in $\SH(S)$.
\end{corollary}

\begin{proof}
This follows from Theorems \ref{h47zt5e} and \ref{jd5ru655}.
\end{proof}

We call a spectrum $E \in \SH(S)$ a Beilinson motive if it is $\HB$-local
(compare with \cite[Definition 13.2.1]{cisinski-deglise}). This is the case if and only if the canonical map
$E \to \HB \wedge E$ is an isomorphism (\cite[Corollary 13.2.15]{cisinski-deglise}).

\begin{corollary}
\label{hn46tzr}
The rationalization $\caM_\Q$ is a Beilinson motive.
\end{corollary}

\begin{proof}
The rational motivic Eilenberg-MacLane spectra $\MQ_{\kappa(\p)}$ for $\p \in S^{(1)}$ and $\MQ_K$ are orientable, thus their
push forwards to $S$ are Beilinson motives (\cite[Corollary 13.2.15 (Ri)]{cisinski-deglise}). Now the claim
follows from Corollary \ref{j675gbbb}.
\end{proof}

By construction we have $\caM_{0,0}=\integers$. By Corollary \ref{hn46tzr} the elements of $(\caM_\Q)_{0,0}$ correspond bijectively
to maps $\HB \to \caM_\Q$, and we let $u \colon \HB \to \caM_\Q$ be the map corresponding to
$1 \in \Q = (\caM_\Q)_{0,0}$.

\begin{theorem}
\label{jf356zgf}
The map $u$ is an isomorphism.
\end{theorem}


\begin{lemma}
\label{fhgt4rztr}
For $\p \in S^{(1)}$ the map
$$i_\p^! \HB \overset{i_\p^!u}{\longrightarrow} i_\p^! \caM_\Q$$
is an isomorphism.
\end{lemma}

\begin{proof}
The definition of $u$ is in such a way that the map $u_1 \colon \HBo \to \caM(1)_\Q[1]$ induced by $u$ is compatible
with the natural maps from $\Q[\GmS]_\Zar$ to $\HBo$ and $\caM(1)_\Q[1]$. Since these latter maps are surjections
it follows that the composition $$(\caO^*_{/S})_\Q \cong \HBo \overset{u_1}{\longrightarrow} \caM(1)_\Q[1] \cong (\caO^*_{/S})_\Q,$$
where the first resp. third map is the identification from Theorem \ref{hnehztrr} resp. Theorem \ref{jr5t4565},
is the identity. It follows that the map
$$\varphi \colon \HBkappap(-1)[-2] \cong i_\p^! \HB \overset{i_\p^!u}{\longrightarrow} i_\p^! \caM_\Q \cong \MQ_{\kappa(\p)}(-1)[-2],$$
where the first isomorphism is \cite[Theorem 13.4.1]{cisinski-deglise},
induces an isomorphism on zeroth spaces of $\Omega$-$\mathbb{G}_{m,\kappa(\p)}$-spectra, thus $\varphi(1)[2]$ corresponds to a nonzero
element in $\Q=\Hom_{\SH(\kappa(\p))}(\HBkappap,\MQ_{\kappa(\p)})$. The claim follows.
\end{proof}

\begin{proof}[Proof of Theorem \ref{jf356zgf}]
The map $f^* u$ is an isomorphism. Now the claim follows from Lemma \ref{fhgt4rztr} and the map between triangles of the form as in
Corollary \ref{j675gbbb} induced by $u$.
\end{proof}

\subsection{The isomorphism between $\MZ$ and $\caM$}
\label{gtz54zr}

First we recast the definition of $\MZ$ working purely in triangulated categories.
We use the notation of section \ref{htr46u}.

The canonical triangle $$L_n(r-1)[-2] \to \underline{\R \Hom}(\integers/p^n[\bA^1_U]_\et,L_n(r)) \to$$
$$\underline{\R \Hom}(\integers/p^n[\GmU]_\et,L_n(r)) \to L_n(r-1)[-1]$$
in $\D(\Sh(\Sm_{U,\et},\integers/p^n))$ induces a canonical isomorphism
$$L_n(r-1)[-1] \cong \underline{\R \Hom}(\integers/p^n[\GmU,\{1\}]_\et,L_n(r)).$$
We thus get a naive $\integers/p^n[\GmU,\{1\}]_\et$-spectrum $\caL_n$ with entry
$L_n(r)[r]$ in level $r$.
By the choice of the map $\iota$ the underlying naive $\integers/p^n[\GmU,\{1\}]_\et$-spectrum
of the spectrum $\Sym(\caT)$ is $\caL_n$.

Since we have canonical isomorphisms
$$\underline{\R \Hom}(\integers/p^n[\GmU,\{1\}]_\Zar,\R \epsilon_* \caL_{n,r})
\cong \R \epsilon_* \underline{\R \Hom}(\integers/p^n[\GmU,\{1\}]_\et, \caL_{n,r})$$
we get a naive $\integers/p^n[\GmU,\{1\}]_\Zar$-spectrum $\R \epsilon_* \caL_n$.
The naive prespectrum (with the obvious definition of naive prespectrum)
$\tau_{\le 0} \R \epsilon_* \caL_n$ with entries $\tau_{\le 0} \R \epsilon_* \caL_{n,r}$
is a naive spectrum by Proposition \ref{bgerhh4}.

We also have canoncial isomorphisms
$$\underline{\R \Hom}(\integers/p^n[\GmS,\{1\}]_\Zar,\R j_* \tau_{\le 0} \R \epsilon_* \caL_{n,r})$$
$$\cong \R j_* \underline{\R \Hom}(\integers/p^n[\GmU,\{1\}]_\Zar, \tau_{\le 0} \R \epsilon_* \caL_{n,r}),$$
thus $H_n:=\R j_* \tau_{\le 0} \R \epsilon_* \caL_n$ is a naive
$T_n:=\integers/p^n[\GmS,\{1\}]_\Zar$-spectrum.

By construction the underlying naive $T_n$-spectrum of $B'$ (see section \ref{htr46u})
equals $H_n$.

Proposition \ref{htgree} furnishes canonical maps
$$H_{n,r} \to i_* \nu_n^{r-1}$$
whose homotopy fibers we denote by $F_{n,r}$.

As after equation (\ref{ht46jk}) it follows that the $F_{n,r}$
are determined up to canonical isomorphisms.

Using the structure maps of $H_n$
we get maps
$$T_n \otimes F_{n,r} \to T_n \otimes H_{n,r} \to H_{n,r+1} \to i_* \nu_n^r$$
which are $0$ since we know already from section \ref{htr46u}
that the $F_{n,r}$ organize themselves into a naive $T$-spectrum such that the maps
$F_{n,r} \to H_{n,r}$ form a map of naive $T_n$-spectra.

Thus in turn the maps $T_n \otimes F_{n,r} \to H_{n,r+1}$
factorize through $F_{n,r+1}$. These factorizations are unique
since there are no non-trivial maps $T_n \otimes F_{n,r} \to i_* \nu_n^r[-1]$.

Thus we see that the $F_{n,r}$ assemble in a unique way into a naive $T_n$-spectrum
$F_n$ together with a map of naive $T_n$-spectra $F_n \to H_n$.

Of course the underlying naive $T_n$-spectrum of $C$ (see section \ref{htr46u})
is $F_n$.

Set $\caM_n(r):=\caM(r)/p^n$ and $\caM_n:=\caM/p^n$. We have {\'e}tale cycle class maps
$$\caM_n(r)|_U \to \R \epsilon_* L_n(r)$$
(see section \ref{jrtzuutre}).

Proposition \ref{hedtt5t5t} (applied with $X=\bA_S^1$ and $Z=\{1\}$)
implies that these cycle class maps combine to give a map
of naive $\integers/p^n[\GmU,\{1\}]_\Zar$-spectra
$j^* \caM_n \to \R \epsilon_* \caL_n$. 

By Theorem \ref{gtfn} this map factors uniquely through
$\tau_{\le 0} \R \epsilon_* \caL_n$ (by an isomorphism, see
Theorem \ref{edfgth76})
and by adjointness
we get a map of naive $T_n$-spectra $\caM_n \to H_n$
which factors as $$\caM_n \to \R j_* j^* \caM_n \overset{\cong}{\longrightarrow} H_n.$$

We have commutative diagrams
$$\xymatrix{\caM_n(r)[r] \ar[r] & \R j_* j^* \caM_n(r)[r] \ar[d]^\cong \ar[r] &
i_* \caM_{n,Z}(r-1)[r-1] \ar[d]^\cong \ar[r] & \caM_n(r)[r+1] \\
& H_{n,r} \ar[r] & i_* \nu_n^{r-1} & }$$
(the vertical maps being isomorphisms) with an exact triangle as upper row:
Th exact triangle is induced by a variant of Proposition \ref{hgrtgz54}. The right vertical
isomorphism is a global version of Theorem \ref{fbmtzzt}. The diagram
commutes by a global version of Proposition \ref{fbmtzzt}.

Thus we get a unique factorization $\caM_n \to F_n$ which is an isomorphism.

Set $T:=\integers[\GmS,\{1\}]_\Zar$. By adjointness we get a map of
naive $T$-spectra $\caM \to F_n$.

For $F,G \in \D(\Sh(\Sm_{S,\Zar},\integers/p^n))$ we denote by $\map_{\integers/p^n}(F,G) \in \Ho \sSet$ the mapping space
between $F$ and $G$. We use the same notation for {\'e}tale sheaves and also for spectra. If the coefficients are
$\integers$ we use the notation $\map$.

\begin{lemma}
\label{get6654e}
For any $r \ge 0$ we have $\map_{\integers/p^n}(F_{n,r},F_{n,r}) \cong \integers/p^n$.
\end{lemma}

\begin{proof}
We have
$$\map_{\integers/p^n}(F_{n,r},\R j_* \tau_{\le 0} \R \epsilon_* \caL_{n,r})$$
$$\cong \map_{\integers/p^n}(j^*F_{n,r},\tau_{\le 0} \R \epsilon_* \caL_{n,r})$$
$$\cong \map_{\integers/p^n}(\tau_{\le 0} \R \epsilon_* \caL_{n,r},\tau_{\le 0} \R \epsilon_* \caL_{n,r})$$
$$\cong \map_{\integers/p^n}(\tau_{\le 0} \R \epsilon_* \caL_{n,r},\R \epsilon_* \caL_{n,r})$$
$$\cong \map_{\integers/p^n}(\epsilon^* \tau_{\le 0} \R \epsilon_* \caL_{n,r},\caL_{n,r})$$
$$\cong \map_{\integers/p^n}(\caL_{n,r},\caL_{n,r})$$
$$\cong \integers/p^n.$$

We have a long exact sequence
$$\cdots \to \Hom(F_{n,r}[i+1],i_* \nu_n^{r-1})
\to \Hom(F_{n,r}[i],F_{n,r}) \to$$
$$\to \Hom(F_{n,r}[i],\R j_* \tau_{\le 0} \R \epsilon_* \caL_{n,r})
\to \Hom(F_{n,r}[i],i_* \nu_n^{r-1}) \to \cdots$$
Thus the maps $$\Hom(F_{n,r}[i],F_{n,r}) \to
\Hom(F_{n,r}[i],\R j_* \tau_{\le 0} \R \epsilon_* \caL_{n,r})$$
are isomorphisms for $i>0$ and injective for $i=0$. 
But since $\Hom(F_{n,r},\R j_* \tau_{\le 0} \R \epsilon_* \caL_{n,r})\cong \integers/p^n$
the map for $i=0$ is also surjective, so the claim follows.
\end{proof}

\begin{corollary}
The naive spectra $F_n$ have lifts to spectra which are unique up to canonical isomorphism.
Denoting these lifts also by $F_n$ we have $\map_{\integers/p^n}(F_n,F_n) \cong \integers/p^n$.
Moreover we have $\map(\caM,F_n) \cong \integers/p^n$ (the latter mapping space is computed in $T$-spectra).
\end{corollary}

Clearly we have $F_{n+1} \otimes_{\integers/p^{n+1}}^\bL \integers/p^n \cong F_n$.

Thus we get compatible maps $\caM \to F_n$ for all $n$ in the homotopy category of $T$-spectra.
This furnishes a map $\caM \to \holim_n F_n$ which is the $p$-completion map.
Note that the homotopy limit is uniquely determined up to canonical isomorphism.

We have a canonical isomorphism $\holim_n F_n \cong D_p$, where we use the notation of
section \ref{jfezztrzz}. Thus we get a canonical map $\caM \to D$ (notation from
section \ref{bfdtzt}) in the homotopy category of $T$-spectra.

Moreover the diagram
$$\xymatrix{\caM \ar[r] \ar[d]^{u^{-1} \circ f} & D \ar[d] \\
\HB \ar[r] & D_\Q,}$$
where $f \colon \caM \to \caM_\Q$
is the rationalization map and
and $u$ is from section \ref{hfdgjtr}, commutes (maps out of $\caM_\Q \cong \HB$ into Beilinson motives correspond to elements in
$\pi_{0,0}$). Thus we obtain a map $\caM \to \MZ$ (for the latter see definition
\ref{ght544ede}). This map is an isomorphism since $u$ is an isomorphism (Theorem \ref{jf356zgf})
and each map $\caM \to D_p$ is the $p$-completion map.

We have shown

\begin{theorem}
\label{grferzt}
There is a canonical isomorphism $\caM \cong \MZ$. We have
$\map(\caM,\caM) \cong \integers$, where $\map$ can denote the mapping space in $T$-spectra or
in $\SH(S)$.
\end{theorem}

We leave the last assertion as an exercise to the reader.

\begin{corollary}
\label{nrerz6u6}
For $X \in \Sm_S$ there is a canonical isomorphism
$$\Hom_{\SH(S)}(\Sigma^\infty X_+,\MZ(n)[i]) \cong H^i_{\mathrm{mot}}(X,n),$$
where the latter group denotes Levine's motivic cohomology.
\end{corollary}

For $X \in \Sm_S$ we denote by $\DM(X)$ the homotopy category of the category of
$f^* \MZ$-modules, where $f$ is the structural morphism of $X$.

\begin{corollary}
\label{nghrrz5}
For $X \in \Sm_S$ there is a canonical isomorphism
$$\Hom_{\DM(X)}(\integers,\integers(n)[i]) \cong H^i_{\mathrm{mot}}(X,n).$$
\end{corollary}

\section{The periodization of $\MZ$}
\label{gheu5444}

We set ourselves in the situation of section \ref{jfezztrzz} before the definition
of $A$. Since $L_\bullet(r) = L_\bullet(1)^{\otimes r}$ for any $r \in \integers$
the collection of the $L_\bullet(r)[2r]$ gives rise to a strictly commutative
algebra $L_\bullet(*)[2*]$ in $\Cpx(\Sh(\Sm_{U,\et},\integers/p^\bullet))^\integers$.
We denote by $e$ the embedding
$$\Cpx(\Sh(\Sm_{U,\et},\integers/p^\bullet)) \to \Cpx(\Sh(\Sm_{U,\et},\integers/p^\bullet))^\integers$$
which sets everything into outer degree $0$
and by the same symbol the induced embedding of $\integers/p^\bullet[\GmU,\{1\}]_\et$-spectra into
$e(\integers/p^\bullet[\GmU,\{1\}]_\et)$-spectra. The tensor product in 
$e(\integers/p^\bullet[\GmU,\{1\}]_\et)$-spectra of $e(\Sym(\caT))$ with the suspension spectrum
of $L_\bullet(*)[2*]$ can be written as the outer tensor product
$\Sym(\caT) \otimes L_\bullet(*)[2*]$. We let
$$\Sym(\caT) \otimes L_\bullet(*)[2*] \to R(\Sym(\caT) \otimes L_\bullet(*)[2*])$$
be a fibrant replacement in $E_\infty$-algebras in $e(\integers/p^\bullet[\GmU,\{1\}]_\et)$-spectra
(i.e. in the semi model category $E_\infty((\Sp_{\integers/p^\bullet[\GmU,\{1\}]_\et}^\Sigma)^\integers)$).
Set $A:=\epsilon_*(R(\Sym(\caT) \otimes L_\bullet(*)[2*]))$.

For $k \in \integers$ we denote by $A_k$ the contribution of $A$ in outer $\integers$-degree $k$,
so $A_k$ is a $\integers/p^\bullet[\GmU,\{1\}]_\et$-spectrum.
We set $A_k':= \tau_{\le (-k)}(A_k)$. The $A_k'$ assemble to an $E_\infty$-algebra
$A' \in E_\infty((\Sp_{\integers/p^\bullet[\GmU,\{1\}]_\Zar}^\Sigma)^\integers)$.
We set $B:=j_*A'$.

As in section \ref{jfezztrzz} we have canonical epimorphisms
$$B_{k,r} \to i_* \nu_\bullet^{k+r-1}[k].$$
We denote by $C_{k,r}$ the kernels of these epimorphisms. A variant of Lemma
\ref{hnzt4} implies that the collection of the $C_{k,r}$ gives
rise to an $E_\infty$-algebra $C \in E_\infty((\Sp_{\integers/p^\bullet[\GmS,\{1\}]_\Zar}^\Sigma)^\integers)$.
Let $C \to C'$ be a fibrant replacement in the latter semi model category and set
$D_p := \lim_n C_{*,\bullet,n}' \in E_\infty((\Sp_{\integers_p[\GmS,\{1\}]_\Zar}^\Sigma)^\integers)$.

Set $D:= \prod_p D_p \in E_\infty((\Sp_{\hat{\integers}[\GmS,\{1\}]_\Zar}^\Sigma)^\integers)$.
We let $\PHB$ be the periodic version of $\HB$, then there is a canonical
$E_\infty$-map $\PHB \to D_\Q$.

\begin{definition}
\label{ght54ede}
We let $\PMZ$ denote the homotopy pullback in $E_\infty$-spectra of the diagram
$$\xymatrix{& D \ar[d] \\ \PHB \ar[r] & D_\Q.}$$
\end{definition}

Clearly we have

\begin{theorem}
\label{jrfh56zr}
The $E_\infty$-spectrum $\PMZ$ is a strong periodization of $\MZ$ in the sense of
\cite{spitzweck.per}.
\end{theorem}

For $X \in Sm_S$ let $\DMT(X)$ be the full localizing triangulated subcategory
of $\DM(X)$ spanned by the $\integers(n)$, $n \in \integers$.
We denote by $\DMT_{\mathrm{gm}}(X)$ the full subcategory of $\DMT(X)$ of
compact objects.

\begin{corollary}
For $X \in \Sm_S$ there is a $E_\infty$-algebra $A$ in $\Cpx(\Ab)^\integers$ and a tensor
triangulated equivalence $\DMT(X) \simeq \D(A)$.
\end{corollary}

\begin{proof}
This follows now from \cite[Theorem 4.3]{spitzweck.per}.
\end{proof}

\begin{corollary}
\label{redhz5eed}
Let $X \in \Sm_S$ such that for any $n$ we have $H^i_{\mathrm{mot}}(X,n)_\Q=0$ for $i \ll 0$
(for example $X=\Spec(R)$, $R$ the localization of a number ring, or $X=\P_R^1 \setminus \{0,1,\infty\}$).
Then there is an affine derived group scheme $G$ over $\integers$ such that $\mathrm{Perf}(G)$,
the (derived) category of perfect representations of $G$, is tensor triangulated equivalent
to $\DMT_{\mathrm{gm}}(X)$.
\end{corollary}

\begin{proof}
This follows from \cite[Theorem 6.21]{spitzweck.fund}.
\end{proof}

\section{Base change}
\label{gf4z4z4}

\begin{proposition}
\label{gfrtg5434}
Let $f \colon T \to S$ be a morphism of schemes, $R$ a commutative ring and $t \in \{\Zar,\Nis,\et\}$. Let
$F \in \D(\Sh(\Sm_{S,t},R))$. For each $X \in \Sm_S$ let $f_X$ be the map
$X_T:=T \times_S X \to X$. Suppose that for each $X \in \Sm_S$ the object
$f_X^*(F|_{X_t}) \in \D(\Sh(X_{T,t},R))$ is zero. Then
$\bL f^* F \in \D(\Sh(\Sm_{T,t},R))$ is zero.
\end{proposition}

\begin{proof}
We use the language of $\infty$-categories. For any scheme $U$ let
$\theta(U)$ be the functor on $\Sm_U^\op$ which associates to any $X \in \Sm_U$
the $\infty$-category associated to the model category $\Cpx(\Sh(X_t,R))$.
Let $\Sect(\theta(U))$ be the category of sections of $\theta(U)$ and $\Sect(\theta(U))_\etcart$
the full subcategory which consists of objects which are cartesian for \'etale morphisms in $\Sm_U$.
Then $\Sect(\theta(U))_\etcart$ is canonically equivalent to the $\infty$-category
associated to $\Cpx(\Sh(\Sm_{U,t},R))$. Note that the inclusion
$$\Sect(\theta(U))_\etcart \hookrightarrow \Sect(\theta(U))$$
preserves limits and colimits, so it has both a left and a right adjoint.
For a morphism $g \colon V \to U$ of schemes one has a base change left adjoint
$g^* \colon  \Sect(\theta(U)) \to \Sect(\theta(V))$, and the base change
$\bL f^* \colon \Cpx(\Sh(\Sm_{U,t},R)) \to \Cpx(\Sh(\Sm_{V,t},R))$ is modelled
by the composition
$$\Sect(\theta(U))_\etcart \hookrightarrow \Sect(\theta(U)) \to
\Sect(\theta(V)) \to \Sect(\theta(V))_\etcart,$$
where the last morphism is the left adjoint to the inclusion.

The assumption on $F$ implies that the composition
$$\Sect(\theta(S))_\etcart \hookrightarrow \Sect(\theta(S)) \to
\Sect(\theta(T))$$ applied to $F$ gives zero, hence the claim.
\end{proof}

\begin{proposition}
\label{g34t65te}
Let $f \colon Y \to X$ be a morphism of schemes which induces
isomorphisms on residue fields and $R$ a commutative ring.
Let $\varepsilon$ denote the maps of sites $X_\et \to X_\Nis$ and $Y_\et \to Y_\Nis$.
Let $F \in \D(\Sh(X_\et,R))$. Then the canonical map
$f^* \R \varepsilon_* F \to \R \varepsilon_* f^* F$ is an isomorphism in
$\D(\Sh(Y_\Nis,R))$.
\end{proposition}

\begin{proof}
Note first that for a scheme $U$, $F \in \D(\Sh(U_\et,R))$, field $K$ and map
$x \colon \Spec K \to U$ which is an
isomorphism on residue fields we have $$(x^*(\R \varepsilon_* F))(\Spec K)
\cong \R \Gamma(\Spec K, x^* F),$$ $\varepsilon$ the map of sites $U_\et \to U_\Nis$.
Thus for a map $x \colon \Spec K \to Y$
inducing an isomorphism on residue fields we have
$$(x^*(\R \varepsilon_*(f^* F)))(\Spec K) \cong \R \Gamma(\Spec K,x^*f^*F) \cong
(x^*f^* \R \varepsilon_*F)(\Spec K).$$
This shows the claim.
\end{proof}

\begin{proposition}
\label{hg788ou}
Let $i \colon Z \hookrightarrow X$ be a closed immersion between separated Noetherian schemes
of finite Krull dimension and $R$ a commutative ring. Let $F \in \D(\Sh(\Sm_{X,\Nis},R))$,
$G \in \D(\Sh(\Sm_{Z,\Nis},R))$ and $\varphi \colon \bL i^*F \to G$ be a map.
Suppose that for any $Y \in \Sm_X$ the map $i_Z^*(F|_{Y_\Nis}) \to G|_{Y_{Z,\Nis}}$ induced by
$\varphi$ is an isomorphism in $\D(\Sh(Y_{Z,\Nis},R))$. Then $\varphi$ is an
$\bA^1$-weak equivalence.
\end{proposition}

\begin{proof}
The cofiber of the adjoint $F \to \R i_*G$ of $\varphi$ satisfies the assumption
of Proposition \ref{gfrtg5434}, thus the map $\bL i^*F \to \bL i^* \R i_* G$ is an
isomorphism. But the map $i^* \R i_* G \to G$ is an $\bA^1$-weak equivalence,
because $\R i_*$ preserves $\bA^1$-weak equvalences (since $i$ is finite)
and the composition $$\D^{\bA^1}(\Sh(\Sm_{Z,\Nis},R)) \to \D^{\bA^1}(\Sh(\Sm_{X,\Nis},R))
\to \D^{\bA^1}(\Sh(\Sm_{Z,\Nis},R))$$ is naturally equivalent to the identity.
\end{proof}

Let now $U$ be the spectrum of a Dedekind domain of mixed characteristic and $p$ a
prime which is invertible on $U$. Let $x \in U$ be a closed point of positive residue characteristic
and $\kappa:=\kappa(x)$. We denote by $i$ the closed inclusion $\{x\} \hookrightarrow U$.
We let $L_{U,n}(r)=\mu_{p^n}^{\otimes r}$ viewed as object of
$\D(\Sh(\Sm_{U,\et},\integers/p^n))$, similarly $L_{\kappa,n}(r)=\mu_{p^n}^{\otimes r}$
viewed as object of $\D(\Sh(\Sm_{\kappa,\et},\integers/p^n))$.
We have a natural map $\varphi \colon L_{U,n}(r) \to \R i_* L_{\kappa,n}(r)$.
Let $\varepsilon$ denote the maps of sites $\Sm_{U,\et} \to \Sm_{U,\Nis}$ and
$\Sm_{\kappa,\et} \to \Sm_{\kappa,\Nis}$. The adjoint 
of $\varphi$ induces the second map in the composition
$$\bL i^* \R \varepsilon_* L_{U,n}(r) \to \R \varepsilon_* \bL i^* L_{U,n}(r)
\to \R \varepsilon_* L_{\kappa,n}(r).$$
Applying $\tau_{\le r}$ to this composition yields the second
map in the composition
$$g_{n,r} \colon \bL i^* \tau_{\le r} \R \varepsilon_* L_{U,n}(r) \to
\tau_{\le r} \bL i^* \R \varepsilon_* L_{U,n}(r) \to
\tau_{\le r} \R \varepsilon_* L_{\kappa,n}(r),$$
whereas the first map canonically exists since $\bL i^*$ preserves $(-r)$-connected
objects.

The $(\tau_{\le r} \R \varepsilon_* L_{U,n}(r))[r]$ assemble into
a naive $\integers/p^n[\GmU,\{1\}]_\Nis$-spectrum $F_n$, and the
$(\tau_{\le r} \R \varepsilon_* L_{\kappa,n}(r))[r]$ into
a naive $\integers/p^n[\Gmkappa,\{1\}]_\Nis$-spectrum $G_n$,
and the $g_{n,r}$ give a map of naive prespectra $g_n \colon \bL i^* F_n \to G_n$.

\begin{proposition}
\label{ge5454dd}
The maps $g_n$ are levelwise $\bA^1$-weak equivalences.
\end{proposition}

\begin{proof}
Let $X \in \Sm_U$, $X_\kappa$ the inverse image of $x$ and $i_X \colon X_\kappa \hookrightarrow X$
the closed inclusion. Proposition \ref{g34t65te} implies that
the canonical map $$i_X^* \R \varepsilon_* \mu_{p^n}^{\otimes r} \to \R \varepsilon_*
\mu_{p^n}^{\otimes r}$$ is an isomorphism in $\D(\Sh(X_{\kappa,\Nis},\integers/p^n))$
(here the first $\mu_{p^n}^{\otimes r}$ denotes an object in $\D(\Sh(X_\et,\integers/p^n))$,
whereas the second an object in $\D(\Sh(X_{\kappa,\et},\integers/p^n))$).
By the exactness of $i_X^*$ the canonical map
$$i_X^* \tau_{\le r} \R \varepsilon_* \mu_{p^n}^{\otimes r} \to \tau_{\le r} \R \varepsilon_*
\mu_{p^n}^{\otimes r}$$ is thus
also an isomorphism (in the same category).
Thus the claim follows from Proposition \ref{hg788ou}.
\end{proof}

Note that $\map_{\integers/p^n}(G_{n,r},G_{n,r}) \cong \integers/p^n$
(compare to Lemma \ref{get6654e}), thus the $G_n$ lift to spectra $G_n$ which
are unique up to canonical isomorphism. 

Using the techniques of section \ref{h5643dfgf} one constructs as in section
\ref{gtz54zr} a map of naive spectra $\caM_\kappa \to G_n$
($\caM_\kappa$ as in section \ref{hjrerrt}) which is induced
by \'etale cycle class maps. The induced map $\caM_\kappa/p^n \to G_n$ is
an isomorphism by Theorem \ref{edfgth76}. The object $D_p:=\holim_n G_n$ is canonically
defined and the canonically induced map $\caM_\kappa \to D_p$ is the $p$-completion map.

Let $p:=\mathrm{char}(\kappa)$.
For each $n \in \naturals$ define a naive $\integers/p^n[\Gmkappa,\{1\}]_\Nis$-spectrum $E_n$ by
$E_{n,r}:=\nu_n^r \in \D(\Sh(\Sm_{\kappa,\Nis},\integers/p^n))$. The bonding maps
are the compositions
$$E_{n,r} \otimes^\bL \integers/p^n[\Gmkappa,\{1\}]_\Nis \to E_{n,r} \otimes^\bL \caO_{/\kappa}^*
\to E_{n,r} \otimes^\bL \nu_n^1 \to \nu_n^{r+1} \cong E_{n,r+1}.$$

There is a map of naive spectra $\caM_\kappa \to E_n$ such that the induced map
$\caM_\kappa/p^n \to E_n$ is an isomorphism.

Note that $\map_{\integers/p^n}(E_{n,r},E_{n,r})$ is (homotopy) discrete, so that
$E_n$ has a canonical model as spectrum which we also denote by $E_n$.
Moreover $E:=\holim_n E_n$ is well-defined
up to canonical isomorphism, and the canonical map $\caM_\kappa \to E$ is the
$p$-completion map.

We denote by $\HBkappa$ the Beilinson spectrum over $\kappa$. There is a canonical map
$$\HBkappa \to (E \times \prod_{p \neq \mathrm{char}(\kappa)} D_p)_\bQ,$$
and the canonical diagram
\begin{equation}
\label{gdsrhee}
\xymatrix{\caM_\kappa \ar[r] \ar[d] & E \times \prod_{p \neq \mathrm{char}(\kappa)} D_p \ar[d] \\
\HBkappa \ar[r] & (E \times \prod_{p \neq \mathrm{char}(\kappa)} D_p)_\bQ}
\end{equation}
is homotopy cartesian.

Suppose now that the $U$ from above is the spectrum of a complete discrete
valuation ring $\Lambda$ and $x$ is the closed point of $U$. Let $p=\mathrm{char}(\kappa)$.
Above we have for any prime $l \neq p$ constructed maps of naive spectra
$$\bL i^* \MZ_U/l^n \cong \bL i^* F_n \to G_n$$
(here the dependence of $G_n$ and $F_n$ on $l$ is suppressed)
which are isomorphisms by Proposition \ref{ge5454dd} (here we view the naive spectra
taking values in the $\bA^1$-local categories). Thus these maps
lift uniquely to isomorphisms between the corresponding spectra. We get canonical maps
$\bL i^* \MZ_U \to D_l$ for all primes $l \neq p$.

Let $\eta$ be the complement of $\{x\}$ in $U$ and $j \colon \eta \hookrightarrow U$
the open inclusion. We again have the objects $L_{\eta,n}(r):= \mu_{p^n}^{\otimes r}
\in \D(\Sh(\Sm_{\eta,\et},\integers/p^n))$ and the map of sites
$\varepsilon \colon \Sm_{\eta,\et} \to \Sm_{\eta,\Nis}$.
By Lemma \ref{jh44444} we get the first isomorphism in the chain
of isomorphisms

\begin{equation}
\label{gfrthrr}
\R j_* \tau_{\le r} \R \varepsilon_* L_{\eta,n}(r) \cong
\tau_{\le r} \R j_* \R \varepsilon_* L_{\eta,n}(r)
\cong \tau_{\le r} \R \varepsilon_* \R j_* L_{\eta,n}(r)
\end{equation}

in $\D(\Sh(\Sm_{U,\Nis},\integers/p^n))$ (the second isomorphism is obvious).

Mapping to $\caH^r$ we get a map
$$\R j_* \tau_{\le r} \R \varepsilon_* L_{\eta,n}(r) \to
\caH^r(\R \varepsilon_* \R j_* L_{\eta,n}(r))[-r]$$
$$ \to \varepsilon_* \varepsilon^* \caH^r(\R \varepsilon_* \R j_* L_{\eta,n}(r))[-r]
\cong \varepsilon_* \R^r j_* L_{\eta,n}(r)[-r].$$

For any $X \in \Sm_U$ we have a map $$(i_X)^* \R^r j_{X,*} (L_{\eta,n}(r)|_{X_{\eta,\et}})
\to \nu_n^r \oplus \nu_n^{r-1}$$ in $\Sh(X_{\kappa,\et},\integers/p^n)$ constructed
in \cite[\S (6.6)]{bloch-kato.p-adic}. The second projection of this map was already
described in section \ref{htr46u}, the first projection is similar:
it sends a symbol $\{f_1, \ldots, f_r\}$, $f_1, \ldots, f_r \in (i_X)^* \caO_X^*$,
to $\dlog \overline{f}_1 \ldots \dlog \overline{f}_r$ and a symbol
$\{f_1, \ldots, f_{r-1}, \pi\}$, $\pi$ a fixed uniformizer of $\Lambda$, to $0$.

As in section \ref{htr46u} these maps glue to give a map
$$\R^r j_* L_{\eta,n}(r) \overset{\varphi \oplus \psi}{\longrightarrow}
i_* \nu_n^r \oplus i_* \nu_n^{r-1}.$$
As in section \ref{gtz54zr} we denote by $F_{n,r}$ the homotopy fiber
of the composition
$$\R j_* \tau_{\le r} \R \varepsilon_* L_{\eta,n}(r)[r] \to
\varepsilon_* \R^r j_* L_{\eta,n}(r) \to i_* \nu_n^r \oplus i_* \nu_n^{r-1}
\to i_* \nu_n^{r-1},$$
and the $F_{n,r}$ assemble to a naive spectrum $F_n$.

Note that the maps $\varphi$ give rise to maps $F_{n,r} \to i_* \nu_n^r$,
thus to maps $$\alpha_{n,r} \colon \bL i^* F_{n,r} \to \nu_n^r=E_{n,r}.$$

\begin{lemma}
\label{dgvjerhh}
The maps $\alpha_{n,r}$ assemble to a map $\alpha_n \colon \bL i^* F_n \to E_n$
of naive prespectra.
\end{lemma}

\begin{proof}
We leave the verification to the reader.
\end{proof}

Our next goal is to show that the $\alpha_{n,r}$ are $\bA^1$-weak equivalences.

\begin{proposition}
\label{ge5t4ded}
Let $X$ be a scheme of characteristic $p$ and $F \in \Sh(X_\et,\integers)$ a $p$-torsion sheaf.
Then $\R^i \varepsilon_* F \in \Sh(X_\Nis,\integers)$ is zero for $i > 1$.
\end{proposition}

\begin{proof}
Let $K$ be a field and $x \colon \Spec K \to X$ a map
inducing an isomorphism on residue fields. Then $$\Gamma(\Spec K, x^* \R \varepsilon_* F)
\cong \R \Gamma(\Spec K, x^* F)$$ in $\D(\Ab)$. The latter complex (which is Galois cohomology)
vanishes in cohomological degrees $> 1$ by \cite[II, Prop. 3]{serre.galois-coh}.
\end{proof}

\begin{proposition}
\label{grtretht}
Let $X$ be a scheme and $F$ a quasi coherent sheaf on $X_\et$.
Then $\R^i \varepsilon_* F \in \Sh(X_\Nis,\integers)$ is zero for $i > 0$.
\end{proposition}

\begin{proof}
Let $K$ be a field and $x \colon \Spec K \to X$ a map
inducing an isomorphism on residue fields. Then $$\Gamma(\Spec K, x^* \R \varepsilon_* F)
\cong \R \Gamma(\Spec K, x^* F)$$ in $\D(\Ab)$. But for a finite Galois extension
$L/K$ the object $(x^*F)(\Spec L)$ is an induced $\mathrm{Gal}(L/K)$-module, thus its
cohomology vanishes in degrees $>0$.
\end{proof}

For $X \in \Sm_\kappa$ let $\Omega_X^1$ be the sheaf on $X_\et$ (and thus also on $X_\Nis$ and $X_\Zar$)
of absolute K\"ahler differentials on $X$. It is quasi coherent. Let $\Omega_X^\bullet$ be
the exterior algebra over $\caO_X$ of $\Omega_X^1$. Define subsheaves
$$B_X^i := \im(d \colon \Omega_X^{i-1} \to \Omega_X^i) \text{ and}$$
$$Z_X^i := \ker(d \colon \Omega_X^i \to \Omega_X^{i+1})$$
of $\Omega_X^i$ on $X_\et$.

\begin{lemma}
\label{gbdrrgggd}
For $X \in \Sm_\kappa$ we have $\R^i \varepsilon_* \Omega_X^j=\R^i \varepsilon_* B_X^j =
\R^i \varepsilon_* Z_X^j=0$ in $\Sh(X_\Nis,\F_p)$ for $j \ge 0$ and $i>0$.
\end{lemma}

\begin{proof}
We have $\R^i \varepsilon_* \Omega_X^j=0$ for $j \ge 0$ and $i>0$ by Proposition \ref{grtretht}.
We have isomorphisms
$$\Omega_X^j \cong Z_X^j/B_X^j$$
given by the inverse Cartier operator, see \cite[top of p. 112]{bloch-kato.p-adic}, thus the claim
follows for $j=0$.
Suppose by induction the claim for $j$.
The exact sequence
$$0 \to Z_X^j \to \Omega_X^j \to B_X^{j+1} \to 0$$
shows the claim for $B_X^{j+1}$, the above isomorphism for $j+1$ shows
the claim for $Z_X^{j+1}$. This finishes the proof.
\end{proof}

Let $L_\eta(r):=L_{\eta,1}(r)$.
For any $X \in \Sm_U$ we have the following isomorphisms

$$i_X^*(\R j_* \tau_{\le r} \R \varepsilon_* L_\eta(r))|_{X_\Nis} \cong
i_X^*(\tau_{\le r} \R \varepsilon_* \R j_* L_\eta(r)|_{X_\Nis})$$
$$\cong \tau_{\le r} i_X^*(\R \varepsilon_* \R j_* L_\eta(r))|_{X_\Nis})
\cong \tau_{\le r} \R \varepsilon_* i_X^* \tau_{\le r} (\R j_* L_\eta(r)|_{X_\et}).$$

The first isomorphism uses (\ref{gfrthrr}), the second the exactness
of $i_X^*$ and the third also Proposition \ref{g34t65te} (strictly speaking
we do not need the third isomorphism, but we give it
for motivation).

Set $K_{X,0}:=i_X^* \tau_{\le r} (\R j_* L_\eta(r)|_{X_\et})$.
We will define filtrations on the $\caH^k(K_{X,0})$, $0 \le k \le r$ (compare
with \cite{bloch-kato.p-adic}). We start with $k=r$.

For $m \ge 1$ let $U^m \caH^r(K_{X,0})$ be the subsheaf of $\caH^r(K_{X,0})$ generated
\'etale locally by sections of the form $\{x_1,\ldots,x_r\}$,
$x_i \in i_X^*j_{X,*} \caO_{X_\eta}^*$, such that  $x_1 -1 \in \pi^m i_X^* \caO_X$,
see \cite[p. 111]{bloch-kato.p-adic}. We define $U^0 \caH^r(K_{X,0}):=\caH^r(K_{X,0})$.

Let $e$ be the absolute ramification index of $\Lambda$ and
$e':=\frac{ep}{p-1}$.

The $0$-th graded piece of the filtration $U^\bullet$ on 
$\caH^r(K_{X,0})$
is $\nu_1^r \oplus \nu_1^{r-1}$,
the $m$-th graded piece for $m \ge e'$ is $0$,
and for $1 \le m < e'$ and $m$ prime to $p$
is $\Omega_{X_\kappa}^{r-1}$
and for $p \mid m$ is
$B_{X_\kappa}^r \oplus B_{X_\kappa}^{r-1}$,
see \cite[Cor. (1.4.1)]{bloch-kato.p-adic}.
We denote these graded pieces by $Q_X^m$.
These sheaves $Q_X^m$ glue to sheaves $Q^m$ on $\Sm_{\kappa,\et}$.

To define the filtrations for $k<r$ we have to adjoin a $p$-th root
of unity and descend a filtration upstairs.

Let $\widetilde{\Lambda}:=\Lambda[\zeta_p]$ be the integral closure of $\Lambda$ in $\widetilde{K}:=K(\zeta_p)$,
where $K$ is the quotient field of $\Lambda$ and $\zeta_p$ is a primitive
$p$-th root of unity. Let $d$ be the degree of
$\widetilde{K}$ over $K$ and $G:=\mathrm{Gal}(\widetilde{K}/K)$. We have $d \mid p-1$.
The group $G$ is canonically identified with a subgroup of $\F_p^* \cong \mu_{p-1} \subset K^*$.

There exists a uniformizer $\tilde{\pi}$ of $\widetilde{\Lambda}$ such that $\tilde{\pi}^d \in \Lambda$,
since this is true in the case $\Lambda=\integers_p$ (take $\tilde{\pi}=\sqrt[p-1]{-p}$) (use Kummer
theory).

Let $\widetilde{U}:=\Spec \widetilde{\Lambda}$, $\tilde{\eta}$ the generic point of $\widetilde{U}$,
$\tilde{\kappa}$ the residue field of $\widetilde{\Lambda}$.

For $X \in \Sm_U$ denote by $\widetilde{X}$ the base change to $\widetilde{U}$.

The notations $\widetilde{X}_{\tilde{\eta}}$, $\widetilde{X}_{\tilde{\kappa}}$,
$j_{\widetilde{X}}$ and $i_{\widetilde{X}}$ explain themselves.

We fix now $X \in \Sm_U$.

We have the commutative diagram

$$\xymatrix{\widetilde{X}_{\tilde{\kappa}} \ar@{^(->}[r]^{i_{\widetilde{X}}} \ar[d]^{f''} &
\widetilde{X} \ar[d]^f & \widetilde{X}_{\tilde{\eta}} \ar@{_(->}[l]_{j_{\widetilde{X}}}
\ar[d]^{f'} \\
X_\kappa \ar@{^(->}[r]^{i_X} & X & X_\eta. \ar@{_(->}[l]_{j_X}}$$

We denote by $\widetilde{L}(k)$ the sheaf $\mu_p^{\otimes k}$ viewed as object of
$\D(\Sh(\widetilde{X}_{\tilde{\eta},\et},\F_p))$ and make the same definition without tildas.
We have a canonical isomorphism $L(r) \cong (\R f'_* \widetilde{L}(r))^G$, where
$(-)^G$ denotes homotopy fixed points. Since the order of $G$ is prime to $p$
the cohomology sheaves of the homotopy fixed points are the fixed points of
the cohomology sheaves. Also homotopy fixed points commutes with
$\R j_{X,*}$ and $i_X^*$, thus we have
$$i_X^* \R j_{X,*} L(r) \cong (i_X^* \R j_{X,*} \R f'_* \widetilde{L}(r))^G
\cong (i_X^* \R f_* \R j_{\widetilde{X},*} \widetilde{L}(r))^G
\cong (\R f''_* i_{\widetilde{X}}^* \R j_{\widetilde{X},*} \widetilde{L}(r))^G.$$

For the last isomorphism we have used the proper base change formula.

So we have $$\caH^k(i_X^* \R j_{X,*} L(r)) \cong
(f''_* \caH^k(i_{\widetilde{X}}^* \R j_{\widetilde{X},*} \widetilde{L}(r)))^G.$$

For an object $X$ with $\F_p$-coefficients and $G$-action we denote by $X\{k\}$
the same object with the $G$-action twisted by the $k$-fold tensor power
of the canonical action of $G$ on $\F_p$.

We have an isomorphism $\widetilde{L}(0) \cong \widetilde{L}(1)$ sending
$1$ to $\zeta_p$. This gives us an isomorphism $\widetilde{L}(0) \cong \widetilde{L}(k)$
for any $k$. We get the isomorphism $$\widetilde{L}(r) \cong \widetilde{L}(k) \otimes
\widetilde{L}(r-k) \cong \widetilde{L}(k) \otimes \widetilde{L}(0) \cong
\widetilde{L}(k).$$

We thus have $$f''_* \caH^k(i_{\widetilde{X}}^* \R j_{\widetilde{X},*} \widetilde{L}(r))
\cong (f''_* \caH^k(i_{\widetilde{X}}^* \R j_{\widetilde{X},*} \widetilde{L}(k)))\{r-k\}$$
as $G$-objects.

We have a filtration on $\caH^k(i_{\widetilde{X}}^* \R j_{\widetilde{X},*} \widetilde{L}(k))$
by the $U^m \caH^k(i_{\widetilde{X}}^* \R j_{\widetilde{X},*} \widetilde{L}(k))$,
where the latter subsheaf is generated as above by sections of the form
$\{x_1,\ldots,x_r\}$ such that $x_1 -1 \in \tilde{\pi}^m i_{\widetilde{X}}^* \caO_{\widetilde{X}}$ ($m \ge 1$,
for $m=0$ we again take the whole sheaf).
Clearly the $f''_* U^m \caH^k(i_{\widetilde{X}}^* \R j_{\widetilde{X},*} \widetilde{L}(k))$
are invariant under the $G$-action, thus the
$$U^{k,m}:=((f''_* U^m \caH^k(i_{\widetilde{X}}^* \R j_{\widetilde{X},*} \widetilde{L}(k)))\{r-k\})^G$$
filter $\caH^k(i_X^* \R j_{X,*} L(r)) \cong \caH^k(K_{X,0})$.

Let $\tilde{e}$ be the absolute ramification index of $\widetilde{\Lambda}$ and
$\tilde{e}':=\frac{\tilde{e}p}{p-1}$.

The $0$-th graded piece of the filtration $U^\bullet$ on 
$\caH^k(i_{\widetilde{X}}^* \R j_{\widetilde{X},*} \widetilde{L}(k))$
is $\nu_1^k \oplus \nu_1^{k-1}$,
the $m$-th graded piece for $m \ge \tilde{e}'$ is $0$,
and for $1 \le m < \tilde{e}'$ and $m$ prime to $p$
is $\Omega_{\widetilde{X}_{\tilde{\kappa}}}^{k-1}$
and for $p \mid m$ is
$B_{\widetilde{X}_{\tilde{\kappa}}}^k \oplus B_{\widetilde{X}_{\tilde{\kappa}}}^{k-1}$,
see \cite[Cor. (1.4.1)]{bloch-kato.p-adic}.

Let us denote these graded pieces by $\widetilde{Q}_X^{k,m}$.
Let us equip the $P_X^{k,m}:=f''_*  \widetilde{Q}_X^{k,m}$ with $G$-actions coming from $\widetilde{L}(k)$.
On the $f''_* (\nu_1^k \oplus \nu_1^{k-1})$, $f''_* \Omega_{\widetilde{X}_{\tilde{\kappa}}}^{k-1}$
and $f''_* (B_{\widetilde{X}_{\tilde{\kappa}}}^k \oplus B_{\widetilde{X}_{\tilde{\kappa}}}^{k-1})$
there is a canonical $\mathrm{Gal}(\tilde{\kappa}/\kappa)$-action, thus a canonical $G$-action.
These also induce $G$-actions on the $f''_* \widetilde{Q}_X^{k,m}$.
We denote these $G$-objects by $R_X^{k,m}$.
The formulas in \cite[(4.3)]{bloch-kato.p-adic} and the definition
of the map to the $0$-graded part show that there are isomorphisms
$$P_X^{k,m} \cong R_X^{k,m}\{m\}$$
as $G$-objects.

We let $Q_X^{k,m}:= \mathrm{gr}^m(U_X^{k,\bullet}) \cong (P_X^{k,m}\{r-k\})^G$.

The considerations made show the following

\begin{proposition}
The sheaves $Q_X^{k,m}$ on $X_{\kappa,\et}$ only depend on
$X_\kappa$ and glue to a sheaf $Q^{k,m}$ on $\Sm_{\kappa,\et}$.
\end{proposition}

Define inductively objects $K_{X,m} \in \D(\Sh(X_{\kappa,\et},\F_p))$ in the
following way. For $m=0$ we have already defined the object.
We define $K_{X,1}$ to be the homotopy fiber of
the composition $$K_{X,0} \to \caH^r(K_{X,0})[-r] \to Q_X^0[-r].$$
There is a canoncial map $K_{X,1} \to Q_X^1[-r]$.
Suppose $K_{X,m}$ together with a map $$K_{X,m} \to Q_X^m[-r]$$ is already defined
for $m < e'$. Define then $K_{X,m+1}$ to be the homotopy fiber of this last map.
If $m+1<e'$ there is a map $K_{X,m+1} \to Q_X^{m+1}[-r]$.
If $m+1 \ge e'$ we have $K_{X,m+1} \cong \tau_{\le (r-1)} K_{X,0}$ and
there is a map $$K_{X,m+1} \to Q_X^{r-1,0}[-r+1].$$
Keep going this way splitting off successively 
the $$Q_X^{r-1,0}[-r+1], Q_X^{r-1,1}[-r+1],\ldots, Q_X^{r-k,m}[-k],\ldots,Q_X^{0,0}$$
(where for this $m$ we require $0 \le m < \tilde{e}'$) obtaining the $K_{X,m+2}, \ldots,K_{X,N}=0$.

By construction we have triangles
$$K_{X,m+1} \to K_{X,m} \to Q_X^{k(m),m'(m)}[-k(m)] \to K_{X,m+1}[1],$$
where $k(m)$ and $m'(m)$ depend in a way on $m$ which
we do not make explicit.

Set $H_0 := \tau_{\le r} \R j_* L_\eta(r)$.
The maps $H_0|_{X_\et} \to i_{X,*} Q_X^0[-r]$ glue to a map of
sheaves $H_0 \to Q^0[-r]$. We let $H_1$ be the homotopy fiber
of this last map. We have a map $i_X^* H_1|_{X_\et} \to K_{X,0}$ which
factors uniquely through $K_{X,1}$, thus we get a map
$i_X^* H_1|_{X_\et} \to Q_X^1$ with adjoint $H_1|_{X_\et} \to i_{X,*} Q_X^1$.
These maps glue to a map $H_1 \to i_* Q^1$ whose homotopy fiber
we denote by $H_2$. Inductively one constructs objects $H_m \in \D(\Sh(\Sm_{U,\et},\F_p))$,
$0 \le m \le N$,
with maps $$i_X^* H_m|_{X_\et} \to K_{X,m} \to Q_X^{k(m),m'(m)}[-k(m)]$$
(here we suppose $m > e'$, the other case is similar)
whose adjoints glue to a map
$$H_m \to \caH^{k(m)}(H_m)[-k(m)] \to i_* Q^{k(m),m'(m)}[-k(m)].$$
$H_{m+1}$ is then defined to be the homotopy fiber of this map.
By construction we have triangles
$$H_{m+1} \to H_m \to i_*Q^{k(m),m'(m)}[-k(m)] \to H_{m+1}[1].$$
Moreover for $X \in \Sm_U$ we have $$i_X^* (H_m|_{X_\et}) \cong K_{X,m}.$$

Note that we have $\R^j \varepsilon_* i_* Q^m \cong
i_* \R^j \varepsilon_* Q^m =0$ for $j >0$ and $m \ge 1$
by Lemma \ref{gbdrrgggd}, similarly we have $\R^j \varepsilon_* i_* Q^{k,m} =0$
for $j>1$ by Proposition \ref{ge5t4ded}.
Thus the canonical maps $$\tau_{\le r} \R \varepsilon_* H_m \to \R \varepsilon_* H_m$$
are isomorphisms for $m \ge 1$.
Since $\caH^{r+1}(\R \varepsilon_* H_1)=0$ it also follows that
the map $$\caH^r(\R \varepsilon_* H_0) \to \varepsilon_* i_*Q^0 \cong i_* \nu_1^r \oplus i_* \nu_1^{r-1}$$
is an epimorphism. It follows that we have an exact triangle
\begin{equation}
\label{gfdrrhe}
\R \varepsilon_* H_1 \to \tau_{\le r} \R \varepsilon_* H_0 \to \varepsilon_* i_* Q^0 \to
\R \varepsilon_* H_1[1].
\end{equation}

\begin{lemma}
\label{gfwr5rt}
For any $F \in \D(\Sh(\Sm_{\kappa,\et},\F_p))$ we have
$\R \varepsilon_* F \in \D(\Sh(\Sm_{\kappa,\Nis},\F_p))$
is $\bA^1$-weakly contractible.
\end{lemma}

\begin{proof}
Tensoring $F$ with the Artin-Schreier exact triangle
$$\F_p \to \mathbb{G}_a \to \mathbb{G}_a \to \F_p[1]$$
shows that it is sufficient to show that $\R \varepsilon_*(F \otimes \mathbb{G}_a)$
is $\bA^1$-contractible. The standard $\bA^1$-contraction of
$\mathbb{G}_a$ does the job.
\end{proof}

\begin{proposition}
\label{gbfdrfhf}
The homotopy cofiber of the map
$\R \varepsilon_* H_N \to \R \varepsilon_* H_1$
is $\bA^1$-weakly contractible in $\D(\Sh(\Sm_{U,\Nis},\F_p))$.
\end{proposition}

\begin{proof}
This homotopy cofiber is filtered with graded pieces
the $\R \varepsilon_* i_* Q^m \cong i_* \R \varepsilon_* Q^m$, $m \ge 1$,
and the $\R \varepsilon_* i_* Q^{k,m} \cong i_* \R \varepsilon_* Q^{k,m}$,
so it is sufficient to show that these are $\bA^1$-weakly contractible.
But $i_*$ preserves $\bA^1$-weak equivalences since it is finite,
so the claim follows from Lemma \ref{gfwr5rt}.
\end{proof}

With Proposition \ref{g34t65te} for any $X \in \Sm_U$ we have
$i_X^* (\R \varepsilon_* H_N)|_{X_\Nis} \cong \R \varepsilon_* i_X^* (H_N|_{X_\et})
\cong \R \varepsilon_* K_{X,N} \cong 0$, thus with Proposition \ref{gfrtg5434} we get
$\bL i^* \R \varepsilon_* H_N \cong 0$.
With Proposition \ref{gbfdrfhf} we get that $\bL i^* \R \varepsilon_* H_1$
is $\bA^1$-weakly contractible.
With (\ref{gfdrrhe}) it follows that 
$$\bL i^* \tau_{\le r} \R \varepsilon_* H_0 \to \bL i^* \varepsilon_* i_* Q^0$$
is an $\bA^1$-weak equivalence.

We get

\begin{theorem}
The maps $\alpha_{n,r}$ defined before Lemma \ref{dgvjerhh} are $\bA^1$-weak
equivalences.
\end{theorem}

So we have isomophisms of naive spectra
$$\bL i^* \MZ_U/p^n \cong \bL i^* F_n \cong E_n$$
which lift uniquely to isomorphisms of spectra. The induced map
$$\bL i^* \MZ \to E$$ 
is the $p$-completion map.

We get that $\bL i^* \MZ_U$ sits in the same homotopy cartesian square
as $\caM_\kappa$ (see diagram (\ref{gdsrhee})),
whence (using Theorem \ref{ji67fhll})

\begin{corollary}
\label{grgrerh4}
There are canonical isomorphisms $\bL i^* \MZ_U \cong \caM_\kappa \cong \MZ_\kappa$.
\end{corollary}

We are next going to construct natural comparison maps for
our spectra for morphisms between Dedekind domains of
mixed characteristic.

So let $D \to \widetilde{D}$ be a map of Dedekind domains of mixed
characeristic. We use the notation of section \ref{jfezztrzz}, without
tildas for the situation over $S=\Spec(D)$ and with tildas
for the situation over $\widetilde{S}=\Spec (\widetilde{D})$.
Note that for the various categories of complexes of sheaves we
have Quillen adjunctions between the categories attached to $S$ and
$\widetilde{S}$. We let $f$ denote the various maps from the situation with
tildas to the situation without, e.g. $f \colon \widetilde{S} \to S$
or $f \colon \widetilde{U} \to U$.

We have an isomorphism $\varphi \colon f^* L_\bullet(r) \cong \widetilde{L}_\bullet(r)$.
Choose map $\psi \colon f^* \caT \to \widetilde{\caT}$ lifting the image
of $\varphi$ in the homotopy category. Thus we get a map
$$\psi' \colon f^* \Sym(\caT) \to \Sym(\widetilde{\caT})$$
of commutative monoids in symmetric
$\integers/p^\bullet[\mathbb{G}_{m,\widetilde{U}},\{1\}]_\et$-spectra.
Using lifting arguments one gets a map
$$\psi'' \colon f^* RQ\Sym(\caT) \to RQ\Sym(\widetilde{\caT}).$$
One gets induced maps
$f^*A \to \widetilde{A}$, $f^* A' \to \widetilde{A}'$,
$f^* B \to \widetilde{B}$, $f^* C \to \widetilde{C}$,
$f^* C' \to \widetilde{C}'$ and $f^* D_p \to \widetilde{D}_p$.

By the definition of $\MZ_S$ and $\MZ_{\widetilde{S}}$
it is then clear that we get the comparison map

$$\Phi_f \colon \bL f^* \MZ_S \to \MZ_{\widetilde{S}}$$
which is a map of $E_\infty$-spectra.



\begin{lemma}
\label{gfdrreg}
If $\widetilde{D}$ is a filtered colimit of smooth $D$-algebras then
the comparison map $\Phi_f$ is an isomorphism.
\end{lemma}

\begin{proof}
Let $\widetilde{D}=\colim_\alpha D_\alpha$, where each $D_\alpha$ is a smooth $D$-algebra
and set $S_\alpha:=\Spec(D_\alpha)$. Let $f_\alpha \colon S_\alpha \to S$
be the canonical maps.
We show that for any $X \in \Sm_{\widetilde{S}}$ and integers $p,q$ the induced map
\begin{equation}
\label{h4z7t44}
\Hom_{\SH(\widetilde{S})}(\Sigma^{p,q} \Sigma_+^\infty X,f^* \MZ_S) \to
\Hom_{\SH(\widetilde{S})}(\Sigma^{p,q} \Sigma_+^\infty X, \MZ_{\widetilde{S}})
\end{equation}
is an isomorphism. By the the remarks after \cite[Definition A.1.]{hoyois.hopkins-morel}
we can write $X=\lim_\alpha X_\alpha$, where each $X_\alpha$ is a smooth and separated $S_\alpha$-scheme
of finite type. By \cite[Lemma A.7.(1)]{hoyois.hopkins-morel} the left side
of (\ref{h4z7t44}) can be written as
$$\colim_\alpha \Hom_{\SH(S_\alpha)}(\Sigma^{p,q} \Sigma_+^\infty X_\alpha ,f_\alpha^* \MZ_S).$$
A similar formula holds for the right hand side, using the isomorphism
$\MZ_{\widetilde{S}} \cong \caM_{\widetilde{S}}$ and observing that the cycle complexes
defining $\caM_{\widetilde{S}}$ evaluated over $X$ are a colimit of the cycle complexes
over the $X_\alpha$.
\end{proof}

\begin{corollary}
\label{grf44thh}
Suppose $\widetilde{D}$ is the completion of a local ring of $S$ at a closed point
of positive residue characteristic. Then $\Phi_f$ is an isomorphism.
\end{corollary}

\begin{proof}
It is known that then $\widetilde{D}$ is a filtered colimit of smooth $D$-algebras,
hence Lemma \ref{gfdrreg} applies.
\end{proof}

\begin{theorem}
\label{grerhrh3}
Let $S=\Spec (D)$, $D$ a Dedekind domain of mixed characteristic, $x \in S$
a closed point of positive residue characteristic and $i \colon \{x\} \to S$
the inclusion. Then there is a canonical isomorphism $\bL i^* \MZ_S \cong \MZ_{\kappa(x)}$
which respects the $E_\infty$-structures.
\end{theorem}

\begin{proof}
The isomorphism as spectra follows now from Corollary \ref{grgrerh4} and Corollary \ref{grf44thh}.
The isomorphism can be made to respect the $E_\infty$-structures by the uniquness
of $E_\infty$-structures on $\MZ_{\kappa(x)}$, which holds sinces this spectrum is the
zero-slice of the sphere spectrum.
\end{proof}

\begin{lemma}
\label{bht5gergg}
Let $g \colon k \to k'$ be a field extension. Then the natural map
$g^*\MZ_k \to \MZ_{k'}$ is an isomorphism.
\end{lemma}

\begin{proof}
This follows from \cite[Theorem 4.18]{hoyois.hopkins-morel} (taking $U$ to be
the spectrum of the prime field contained in $k$).
\end{proof}

\begin{lemma}
\label{vdgther}
Let $\widetilde{S} = \Spec(\widetilde{D})$, $\widetilde{D}$ a Dedekind domain, and let $\varphi \colon E \to F$ be any map
in $\SH(\widetilde{S})$. Suppose for any $x \in \widetilde{S}$ that $\bL i_x^* \varphi$ is
an isomorphism, where $i_x$ denotes the inclusion $\{x\} \hookrightarrow \widetilde{S}$.
Then $\varphi$ is an isomorphism.
\end{lemma}

\begin{proof}
This follows from \cite[Proposition 4.3.9]{cisinski-deglise} and localization.
\end{proof}

\begin{theorem}
For any $f$ as above the comparison map $\Phi_f$ is an isomorphism.
\end{theorem}

\begin{proof}
This follows from Theorem \ref{grerhrh3}, Lemma \ref{bht5gergg} and Lemma \ref{vdgther}.
\end{proof}

\section{The motivic functor formalism}
\label{g5r3tfx}

For any Noetherian separated scheme $X$ of finite Krull dimension (we call such
schemes from now on base schemes) we let
$\MZ_X := f^* \MZ_{\Spec(\integers)}$, where $f \colon X \to \Spec(\integers)$
is the structure morphism. We let $\MZ_X-\Mod$ be the model category of
highly structured $\MZ_X$-module spectra and set $\DM(X):=\Ho(\MZ_X-\Mod)$.
This is done e.g. along the lines of \cite{roendigs-oestvaer}.
For any map of base schemes $f \colon X \to Y$ we get an adjunction
$$f^* \colon \DM(Y) \rightleftharpoons \DM(X) \colon f_*.$$

The categories $\DM(X)$ are closed tensor triangulated and the functors
$f^*$ are symmetric monoidal.

If $f$ is smooth the functor $f^*$ has a left adjoint $f_\sharp$.

Note that all these functors commute with the forgetful functors
$$\DM(X) \to \SH(X).$$ It follows that the assignment
$$X \mapsto \DM(X)$$
has the structure of a stable homotopy functor in the sense of
\cite{ayoub.I}.

Thus the main results of loc. cit. are valid for this assignment.

In particular for a morphism of finite type $f \colon X \to Y$ we have an adjoint pair
$$f_! \colon \DM(X) \rightleftharpoons \DM(Y) \colon f^!.$$
Moreover the projective base change theorem holds.

We also have the
\begin{theorem}
\label{ht346z4}
Let $i \colon Z \hookrightarrow X$ be a closed inclusion of base schemes and
$j \colon U \hookrightarrow X$ the open complement. Then for any $F \in \DM(X)$
there is an exact triangle
$$ j_!j^*F \to F \to i_* i^*F \to j_!j^*F[1]$$
in $\DM(X)$.
\end{theorem}

\section{Further applications}
\label{g54r3grq}

\subsection{The Hopkins-Morel isomorphism}

We first equip $\MZ$ with an orientation.

\begin{proposition} Let $X$ be a smooth scheme over Dedekind domain of
mixed characteristic. Then
there is a unique orientation on $\MZ_X$. The corresponding formal group law
is the additive one.
\end{proposition}

\begin{proof}
Let $S$ be the spectrum of a Dedekind domain of mixed characteristic. Let $P \in \D(\Sh(\Sm_{S,\Zar},\integers))$
be the first $\bA^1$- and Nisnevich-local space of an $\Omega$-$\P^1$-spectrum model of $\MZ_S$.
Then by Theorem \ref{jr5t4565} there is a canonical isomorphism $P \cong \caO^*_{/S}[1]$. Moreover by the proof
of this theorem the canonical map $\integers[\P^1,\{\infty\}]_\Zar \to P$ induced by the first bonding map
is induced by the suspension of the canonical map $\integers[\GmS,\{1\}]_\Zar \to \caO^*_{/S}$, using the canonical
isomorphism $(\P^1,\{\infty\}) \cong \GmS \wedge S^1$ in $\caH_\bullet(S)$. Thus our map
$\integers[\P^1,\{\infty\}]_\Zar \to \caO^*_{/S}[1]$ classifies the line bundle $\caO(-1)$.
So the map $\Sigma^{-2,-1} \Sigma_+^\infty \P^\infty \to \MZ_S$ corresponding to the map
$\integers[\P^\infty]_\Zar \to \caO^*_{/S}[1]$ which classifies the tautological
line bundle is an orientation of $\MZ_S$. Pulling back to any smooth scheme $X$ over $S$ gives an
orientation of $\MZ_X$. Since motivic cohomology of $X$ with negative weight vanishes
this orientation is unique and the corresponding formal group law is the additive one.
\end{proof}

By pulling back the unique orientation of $\MZ_{\Spec(\integers)}$ to any base scheme $X$ we
see that $\MZ_X$ has a canonical additive orientation.

\begin{remark}
\label{ht63423}
We note that over smooth schemes $X$ over Dedekind domains of mixed characteristic or over fields
the orientation map $\MGL_X \to \MZ_X$ has a unique structure of an $E_\infty$-map.
This $E_\infty$-map is achieved as the composition of the $E_\infty$-maps
$$\MGL_X \to s_0 \MGL_X \cong s_0 \unit \to \s_0 \MZ_X \to \MZ_X,$$
where the last map exists since the map $f_0 \MZ_X \to s_0 \MZ_X$ is an isomorphism
since $\MZ_X$ is $0$-truncated.

Thus for any base scheme $X$ the orientation $\MGL_X \to \MZ_X$ has a canonical
$E_\infty$-structure. Since $\MGL_X$ has a strong periodization this gives a second
proof that $\MZ_X$ is strongly periodizable.
\end{remark}

We see that we can factor the orientation map $\MGL \to \MZ_{\Spec(\integers)}$ through
the quotient $\MGL/(x_1,x_2,\ldots)\MGL$, where the $x_i$ are images of generators of $\MU_*$
with respect to the natural map $\MU_* \to \MGL_{2*,*}$.
Pulling back this factorization to any base scheme $X$ we get a map
$\Phi_X \colon \MGL_X / (x_1,x_2,\ldots)\MGL_X \to \MZ_X$.

\begin{theorem}
\label{gh543tz}
Let $R$ be a commutative ring and $X$ a base scheme whose positive residue characteristics are all invertible
in $R$. Then $\Phi_X \wedge M_R$, where $M_R$ denotes the Moore spectrum on $R$, is an isomorphism. 
\end{theorem}

\begin{proof}
We only have to show this statement for $X$ being equal to the spectrum of a localization of $\integers$.
Then it follows by pullback to the points of $X$ and \cite[Theorem 7.12]{hoyois.hopkins-morel}
using Theorem \ref{grerhrh3} and Lemma \ref{vdgther}.
\end{proof}

\begin{corollary}
\label{h5r324we}
Let $R$ be a commutative ring and $X$ a base scheme whose positive residue characteristics are all invertible
in $R$. Then $\MR_X$ (which denotes $\MZ_X$ with $R$-coefficients) is cellular.
\end{corollary}

\subsection{The motivic dual Steenrod algebra}

In the whole section we fix a prime $l$.

For a base scheme $S$ we denote by $\underline{\Pic}_S$ a strictification
of the $2$-functor which assigns to any $X \in \Sm_S$ the Picard groupoid
of line bundles on $X$. We denote by $\nu \underline{\Pic}_S$ the motivic space which
assigns to any $X \in \Sm_S$ the nerve of $\underline{\Pic}_S(X)$.

\begin{proposition}
\label{ghtr543e}
Let $S$ be a regular base scheme and let $f \colon \P_S^\infty \to \nu \underline{Pic}_S$
be a map classifying a $\mathbb{G}_m$-torsor $P$ on $\P_S^\infty$. Then there is an
$\bA^1$-fiber sequence
$$P \to \P_S^\infty \to \nu \underline{\Pic}_S$$
of motivic spaces.
\end{proposition}

\begin{proof}
The sequence is a fiber sequence in simplicial presheaves equipped with a model
structure with objectwise weak equivalences. Thus the claim follows from
the $\bA^1$- and Nisnevich-locality of $\nu \underline{\Pic}_S$
(and e.g. right properness of motivic model structures). 
\end{proof}

For a base scheme $S$ we let $W_{S,n,k}$ be the $\mathbb{G}_m$-torsor on
$\P_S^k$ corresponding to the line bundle $\caO_{\P^k}(-n)$. We let
$W_{S,n}:= \colim_k W_{S,n,k}$ be the corresponding $\mathbb{G}_m$-torsor
on $\P_S^\infty$.

We are going to compute the motivic cohomology of $W_{S,n}$ with $\integers/m$-coefficients
for $m|n$
relative to the motivic cohomology of the base $S$. We orient ourselves
along the lines of \cite[\S 6]{voevodsky.power}.

We have a cofibration sequence
\begin{equation}
\label{ghrez545}
W_{S,n,+} \to \caO_{\P_S^\infty}(-n)_+ \to \Th(\caO_{\P_S^\infty}(-n)).
\end{equation}

For a motivic space $\caX$ over $S$ let $$H^{p,q}(\caX):=\Hom_{\SH(S)}(\Sigma^\infty \caX_+,\Sigma^{p,q} \MZ)$$
be the motivic cohomology of $\caX$. More generally for an abelian group $A$ we set
$$H^{p,q}(\caX,A):=\Hom_{\SH(S)}(\Sigma^\infty \caX_+,\Sigma^{p,q} \MA).$$
We denote the respective reduced motivic cohomology groups of pointed motivic spaces
by $\widetilde{H}$.

Then (\ref{ghrez545}) gives a long exact sequence
\begin{equation}
\label{ju907tff}
\cdots \to H^{*-2,*-1}(S) \poauf \sigma \pozu \overset{n \sigma}{\longrightarrow}
H^{*,*}(S) \poauf \sigma \pozu \to H^{*,*}(W_{S,n}) \to
H^{*-1,*-1}(S) \poauf \sigma \pozu \to \cdots.
\end{equation}

Here $\sigma$ is the class of $\caO_{\P_S^\infty}(-1)$ in $H^{2,1}(\P_S^\infty)$.

For any $m>0$ let $\beta_m \colon H^{*,*}(\_,\integers/m) \to H^{*+1,*}(\_)$
be the Bockstein homomorphism.

Let $v_n \in H^{2,1}(W_{S,n})$ be the pullback of $\sigma$ under the canonical map
$W_{S,n} \to \P_S^\infty$.

\begin{lemma}
\label{gr4thgrw}
For any $m>0$ there is a $u \in H^{1,1}(W_{S,n},\integers/m)$ such that
the restriction of $u$ to $*$ is $0$ and such that $\beta_m(u)=\frac{n}{\mathrm{gcd}(m,n)} \cdot v_n$.
If $S$ is smooth over a Dedekind domain of mixed characeristic or over a field then this
$u$ is unique with these properties.
\end{lemma}

\begin{proof}
Let $\tilde{v} \in H^{2,1}(W_{S,n})$ be any $m$-torsion class which restricts to $0$ on $*$.
Note that $\frac{n}{\mathrm{gcd}(m,n)} \cdot v_n$ is such a class. We will prove that
then there is a unique $\tilde{u} \in H^{1,1}(W_{S,n},\integers/m)$ which restricts to $0$ on $*$
such that $\beta_m(\tilde{u})=\tilde{v}$, assuming $S$ is smooth over a Dedekind domain
of mixed characteristic or over a field. The general statement about existence follows then
by base change (e.g. from $\Spec(\integers)$ to $S$).

Consider the commutative diagram
$$\xymatrix{H^{1,1}(W_{S,n}) \ar[r] \ar[d] & H^{1,1}(W_{S,n},\integers/m) \ar[r] \ar[d] &
H^{2,1}(W_{S,n}) \ar[r]^{\cdot m} \ar[d] & H^{2,1}(W_{S,n}) \\
H^{1,1}(S) \ar[r] & H^{1,1}(S,\integers/m) \ar[r] & H^{2,1}(S) &}$$
with exact rows and where the vertical maps are restriction to $*$
which split the maps on cohomology induced by the structure map $W_{S,n} \to S$.
The exact sequence (\ref{ju907tff}) around $H^{1,1}(W_{S,n})$ shows that
the first vertical map is an isomorphism.
A diagram chase then shows existence and uniqueness of $\tilde{u}$
with the required properties.
\end{proof}

We denote the canonical class in $\widetilde{H}^{1,1}(W_{S,n},\integers/m)$ obtained this way by
$u_{n,m}$ (by demanding that these classes are compatible with base change). We set $u_n:=u_{n,n}$.

We let $$K(\integers/n(1),1)_S, K(\integers(1),2)_S \in \caH_\bullet(S)$$ be the motivic Eilenberg-MacLane spaces which 
represent the functors $\widetilde{H}^{1,1}(\_,\integers/n)$ and $\widetilde{H}^{2,1}(\_)$
on $\caH_\bullet(S)$ respectively.

\begin{proposition}
\label{ghr34twe}
If $S$ is smooth over a Dedekind domain of mixed characteristic or over a field then
we have $K(\integers(1),2)_S \cong \nu \underline{\Pic}_S \cong B \GmS \cong \P_S^\infty$ in $\caH_\bullet(S)$.
\end{proposition}

\begin{proof}
This follows from the fact the motivic sheaf of weight $1$ is in this case $\caO_{/S}^*[-1]$.
\end{proof}

\begin{proposition}
\label{vgrehrr}
If $S$ is smooth over a Dedekind domain of mixed characteristic or over a field then
we have $K(\integers/n(1),1)_S \cong W_{S,n}$ in $\caH_\bullet(S)$.
The isomorphism is given by the class $u_n$.
\end{proposition}

\begin{proof}
Let $f \colon \P_S^\infty \to \nu \underline{\Pic}_S$ be the map classifying the line bundle
$\caO_{\P_S^\infty}(-n)$. $W_{S,n}$ is the corresponding $\mathbb{G}_m$-torsor over $\P_S^\infty$.
Then the diagram
$$\xymatrix{W_{S,n} \ar[r] & \P_S^\infty \ar[r]^f \ar[d] & \nu \underline{\Pic}_S \ar[d] \\
K(\integers/n(1),1)_S \ar[r] & K(\integers(1),2)_S \ar[r]^{\cdot n} & K(\integers(1),2)_S}$$
in $\caH_\bullet(S)$,
where the vertical maps are the canonical identifications,
commutes. Moreover the rows are fiber sequences: the first one by Proposition \ref{ghtr543e},
the second one by definition. It follows that there is a vertical isomorphism
$u' \colon W_{S,n} \to K(\integers/n(1),1)_S$ in $\caH_\bullet(S)$ making the whole diagram
commutative. The uniqueness clause of Lemma \ref{gr4thgrw} shows that $u'=u_n$
finishing the proof.
\end{proof}

For any $X \in \Sm_S$ consider the functor $T_n \colon \underline{\Pic}(X) \to \underline{\Pic}(X)$,
$\caL \mapsto \caL^{\otimes n}$. Its homotopy fiber is the Picard groupoid $\underline{G}_n(X)$ whose objects
are pairs $(\caL,\varphi)$, where $\caL$ is a line bundle on $X$ and $\varphi \colon \caL^{\otimes n} \to \caO_X$
is an isomorphism, and whose morphisms are isomorphisms of line bundles compatible with the
trivializations. Note that we have a fiber sequences
$$\nu \underline{G}_n(X) \to \nu \underline{\Pic}(X) \overset{\nu T_n}{\longrightarrow}
\nu \underline{\Pic}(X)$$
functorial in $X$ and that these fiber sequences also make sence for $X \in \Set^{\Sm_S^\op}$.

As in the proof of Proposition \ref{vgrehrr} it follows that we have a canonical
equivalence $K(\integers/n(1),1)_S \cong \nu \underline{G}_n$ in $\caH_\bullet(S)$,
provided that $S$ is smooth over a Dedekind ring of mixed characteristic or over a field.

Since $\nu \underline{G}_n$ is Nisnevich- and $\bA^1$-local it follows

\begin{proposition}
\label{grf4z5we}
Suppose $S$ is smooth over a Dedekind ring of mixed characteristic or over a field
and let $X$ be in $\Sm_S$ or $\Set^{\Sm_S^\op}$.
There is a canonical group isomorphism between $H^{1,1}(X,\integers/n)$ and the group of
isomorphism classes of $\underline{G}_n(X)$. The boundary map
$H^{1,1}(X,\integers/n) \to H^{2,1}(X)$ corresponds to the map on groupoids which forgets
the trivialization.
\end{proposition}

\begin{lemma}
\label{54zwefwe}
Suppose $S$ is smooth over a Dedekind ring of mixed characteristic or over a field.
The class $u_{n,m}$ corresponds under the isomorphism of Proposition \ref{grf4z5we} to
the isomorphism class of the object
$$(p^* \caO_{\P_S^\infty}(-\frac{n}{\mathrm{gcd}(m,n)}),
(p^* \caO_{\P_S^\infty}(-\frac{n}{\mathrm{gcd}(m,n)}))^{\otimes m} \cong
p^* \caO_{\P_S^\infty}(-\mathrm{lcm}(m,n)) \cong \caO_{W_{S,n}}),$$
where the last isomorphism is the $\frac{m}{\mathrm{gcd}(m,n)}$-th tensor power
of the canonical isomorphism $p^* \caO_{\P_S^\infty}(-n) \cong \caO_{W_{S,n}}$.
\end{lemma}

\begin{proof}
This element clearly satisfies the requirements of Lemma \ref{gr4thgrw}.
\end{proof}

\begin{lemma}
\label{gh455zz}
The image of the constant function on $1$ under the isomorphisms
$$(\integers/m)^{\pi_0(S)} \cong H^{0,0}(S, \integers/m) \cong \widetilde{H}^{1,1}(\GmS,\integers/m)$$
corresponds under the isomorphism of Proposition \ref{grf4z5we} to
the object $(\caO_{\GmS}, \varphi)$, where $\varphi$ is given by multiplication with the canonical
unit in $\caO(\GmS)$.
\end{lemma}

\begin{proof}
This unit corresponds to $1$ under the map
$$\caO^*(\GmS) \cong H^{1,1}(\GmS) \to \widetilde{H}^{1,1}(\GmS) \cong H^{0,0}(S) \cong \integers^{\pi_0(S)}.$$
\end{proof}

\begin{corollary}
\label{g544tte}
Suppose $S$ is smooth over a Dedekind ring of mixed characteristic or over a field.
Then the image of $u_{n,m}|_{W_{S,n,0}}$ under the isomorphisms
$$\widetilde{H}^{1,1}(W_{S,n,0},\integers/m) \cong \widetilde{H}^{1,1}(\GmS,\integers/m)
\cong H^{0,0}(S, \integers/m) \cong (\integers/m)^{\pi_0(S)}$$
is the constant function on the class of $\frac{m}{\mathrm{gcd}(m,n)}$.
\end{corollary}

\begin{proof}
This follows from Lemmas \ref{54zwefwe} and \ref{gh455zz}.
\end{proof}

\begin{definition}
Let $E$ be a motivic ring spectrum (i.e. a commutative monoid in $\SH(S)$) such that
$E_{0,0}$ is a $\integers/n$-algebra. A \emph{mod-$n$ orientation} on $E$ consists
of an orientation $c \in E^{2,1}(\P_S^\infty)$ and a class
$u \in \widetilde{E}^{1,1}(W_{S,n})$ which restricts to $1$ under the map
$$\widetilde{E}^{1,1}(W_{S,n}) \to \widetilde{E}^{1,1}(W_{S,n,0})
\cong \widetilde{E}^{1,1}(\GmS) \cong E_{0,0}.$$ 
\end{definition}

It follows from Corollary \ref{g544tte} that the usual orientation of $\MZm$ together
with the class $u_{n,m}$ defines a mod-$n$ orientation on $\MZm$ provided $m|n$.
We call this orientation the canonical mod-$n$ orientation of $\MZm$.

Note also that any mod-$n$ orientation gives rise to a mod-$n'$ orientation
for $n|n'$.

\begin{theorem}
Let $E$ be a motivic ring spectrum such that $E_{0,0}$ is a $\integers/n$-algebra
with a mod-$n$ orientation given by classes $c \in E^{2,1}(\P_S^\infty)$ and
$u \in E^{1,1}(W_{S,n})$. Let $v \in E^{2,1}(W_{S,n})$ be the pullback of $c$ under
the canonical projection $W_{S,n} \to \P_S^\infty$.
Let $\caX$ be a motivic space. Denote by $u$ and $v$ also the pullbacks of $u$ and
$v$ to the $E$-cohomology of $\caX \times W_{S,n}$.
Then the elements $v^i$, $uv^i$, $i \ge 0$ form a topological basis of
$E^{*,*}(\caX \times W_{S,n})$ over $E^{*,*}(\caX)$. More precisely,
the elements $v^i$, $uv^i$, $0 \le i \le k$, form a basis
of $E^{*,*}(\caX \times W_{S,n,k})$ over $E^{*,*}(\caX)$, $v^{k+1}$ is zero
in $E^{*,*}(\caX \times W_{S,n,k})$ and the canonical map
$$E^{*,*}(\caX \times W_{S,n}) \to \lim_k E^{*,*}(\caX \times W_{S,n,k}),$$
where the transition maps are surjective, is an isomorphism.
\end{theorem}

\begin{proof}
By writing $\caX$ has the homotopy colomit over $\bigtriangleup^\op$ of a diagram with entries
disjoint unions of objects from $\Sm_S$ and replacing cohomology groups by mapping spaces
we reduce to the case where $\caX=X \in \Sm_S$.
The induced long exact sequences in $E$-cohomology from the cofiber sequence
$$W_{X,n,k,+} \to \caO_{\P_X^k}(-n)_+ \to \Th(\caO_{\P_X^k}(-n))$$
split into short exact sequences
$$0 \to E^{*,*}(X)[\sigma]/(\sigma^{k+1}) \to E^{*,*}(X \times W_{S,n,k})
\to E^{*-1,*-1}(X)[\sigma]/(\sigma^{k+1}) \to 0$$
since $E_{0,0}$ is a $\integers/n$-algebra.
The image of $u$ in the right group is of the form
$1 + \sigma \cdot r$.
Using the fact that these sequences are $E^{*,*}(X)[\sigma]/(\sigma^{k+1})$-module
sequences the claim follows.
\end{proof}

It follows that every element in $E^{*,*}(\caX \times W_{S,n})$ can be uniquely
written as a power series
$$\sum_{i \ge 0} (a_i v^i + b_i uv^i)$$
with $a_i,b_i \in E^{*,*}(\caX)$. Similar statements are valid for elements
in $E^{*,*}(\caX \times W_{S,n}^j)$. The latter group can be written
as the $j$-fold completed tensor product over $E^{*,*}(\caX)$ of copies of
$E^{*,*}(\caX \times W_{S,n})$.

Note that if $n$ is odd we have
$$E^{*,*}(\caX \times W_{S,n}^j) \cong E^{*,*}(\caX)\poauf v_1,\ldots,v_j \pozu (u_1,\ldots,u_j),$$
but if $n$ is even there can be more complicated relations for the $u_i^2$.

The object $W_{S,n} \in \caH(S)$ is naturally a commutative group object
(it represents motivic cohomology over certain $S$, in particular $S=\Spec(\integers)$,
and pulls back). Moreover it has exponent $n$.
This gives $E^{*,*}(\caX \times W_{S,n})$ the structure of a cocommutative Hopf algebra
object in a category whose tensor structure is the completed tensor product.

The comultiplication $$E^{*,*}(\caX \times W_{S,n}) \to E^{*,*}(\caX \times W_{S,n}^2)$$
is uniquely determined be the images of $u$ and $v$ which can be written as power
series in $u_1,u_2,v_1,v_2$. These power series obey laws which are similar to the familiar
formal group laws. We won't spell out these properties, suffices it to say
that they are grouped into unitality, associativity, commutativity, exponent $n$
and independence of the image of $v$ of $u_1,u_2$.

For $E=\MZm$, $m|n$, we have the additive law: $u \mapsto u_1+u_2$, $v \mapsto v_1+v_2$.
This follows from weight reasons for $S=\Spec(\integers)$ and thus is true in general.

There is the notion of a strict isomorphism of such laws (power series), again given by
two power series (in $u$ and $v$) in the target complete ring which start with $u$ respectively $v$. Moreover
the second power series is independent of $u$.
Caution is required in the case $n$ is even since then the complete rings in
question might not have standard form.

Two mod-$n$ orientations on a motivic ring spectrum give rise to such a
strict isomorphism.

\begin{proposition}
\label{fhj643}
Let $E$ be a motivic ring spectrum such that $E_{0,0}$ is an $\bF_l$-algebra
equipped with two additive mod-$l$ orientations.
Then the corresponding strict isomorphism has the form
$$u \mapsto u+a_0v + a_1 v^l + \cdots + a_i v^{l^i} + \cdots,$$
$$v \mapsto v+ b_1 v^l + \cdots + b_i v^{l^i} + \cdots.$$
\end{proposition}

\begin{proof}
The proof is similar to the case of the additive formal group law over an $\bF_l$-algebra.
\end{proof}

\begin{lemma}
The suspension spectrum $\Sigma_+^\infty W_{S,n,k}$ is finite cellular, in particular dualizable.
The suspension spectrum $\Sigma_+^\infty W_{S,n}$ is cellular.
\end{lemma}

\begin{proof}
This is a standard argument.
\end{proof}

In the following let $T_k:=\Sigma_+^\infty W_{S,n,k}$.

\begin{lemma}
\label{gre54zew}
Let $E$ be a mod-$n$ oriented motivic ring spectrum. For any $0 \le i \le 1$, $0 \le j \le k$
let $\Sigma^{-2j-i,-j-i} E \to E \wedge T_k^\vee$ be the $E$-module map corresponding to the
element $u^iv^j$ in the homotopy of the target spectrum. Then the induced map
$$\bigoplus_{i,j} \Sigma^{-2j-i,-j-i} E \longrightarrow E \wedge T_k^\vee$$
is an isomorphism.
\end{lemma}

\begin{proof}
Homing out of $\Sigma^{p,q} \Sigma_+^\infty X$, $X \in \Sm_S$, shows the claim.
\end{proof}

\begin{lemma}
\label{fr6u7ue}
Let $E$ be a motivic ring spectrum and $U$ be a dualizable spectrum such that
$E \wedge U$ is a finite sum of shifts of $E$ as a $E$-module. Then $E \wedge U^\vee$
is the sum over the corresponding negative shifts of $E$.
\end{lemma}

\begin{proof}
For $F$ a motivic spectrum we have
$$\Hom(F,E \wedge U^\vee)=\Hom(F \wedge U,E)=\Hom_E(E \wedge U \wedge F,E)$$
$$=\Hom_E(\bigoplus_\alpha \Sigma^{p_\alpha,q_\alpha} E \wedge F,E)=
\Hom(F, \bigoplus_\alpha \Sigma^{-p_\alpha,-q_\alpha} E).$$
\end{proof}

\begin{lemma}
Let $E$ be a mod-$n$ oriented motivic ring spectrum.
We have $E$-module isomorphisms
$$\bigoplus_{i,0 \le j \le k} \Sigma^{2j+i,j+i} E \cong E \wedge T_k$$
and
$$\bigoplus_{i,0 \le j} \Sigma^{2j+i,j+i} E \cong E \wedge T_k,$$
where the corresponding generators are the duals of the $v^j,uv^j$.
\end{lemma}

\begin{proof}
Use Lemmas \ref{gre54zew} and \ref{fr6u7ue} (the latter is applied with $U=T_k^\vee$).
\end{proof}

Let $E,F$ be mod-$n$ oriented motivic ring spectra.

\begin{lemma}
The natural map
$$E_{**} F \otimes_{F_{**}} F_{**} T_k^\vee \to (E \wedge F \wedge T_k^\vee)_{**}$$
is an isomorphism.
\end{lemma}

\begin{proof}
This follows from Lemma \ref{gre54zew}.
\end{proof}

From the above Lemma we derive a coaction map
$$E^{**}(W_{S,n,k}) \cong E_{-*,-*} T_k^\vee \to (E \wedge F \wedge T_k^\vee)_{-*,-*}
\cong E_{-*,-*} F \otimes_{F_{-*,-*}} F^{**}(W_{S,n,k}),$$
where for the second map we use the unit of $F$.
These are compatible for different values of $k$, yielding in the limit a coaction map
$$E^{**}(W_{S,n}) \to E_{-*,-*} F \hat{\otimes}_{F_{-*,-*}} F^{**}(W_{S,n}).$$

We write the image of $u$ as $$\sum_{j \ge 0} (\alpha_j \otimes v^j + \beta_j \otimes uv^j),$$
similarly we write the image of $v$ as $$\sum_{j \ge 0} \gamma_i \otimes v^j.$$
(The latter sum is independent of $u$ since the relation comes already from the projective space.
Note also that the $u$'s and $v$'s on both sides are lying in different groups.)

\begin{proposition}
\label{g5r5zu34}
The strict isomorphism relating the two mod-$n$ orientations on $E \wedge F$ has the form
$$u_E = \sum_{j \ge 0} (\alpha_j v_F^j + \beta_j u_Fv_F^j),$$
$$v_E = \sum_{j \ge 0} \gamma_i v_F^j.$$
Here the $u_E, v_E$ are those generators coming from the orientation on $E$, similarly for $u_F,v_F$.
\end{proposition}

\begin{proof}
We leave the verification to the reader.
\end{proof}

\begin{remark}
The coeffcients $\alpha_i,\beta_i,\gamma_i$ can also be described
by images of canonical homology generators with respect to the maps on $F$-homology
of the orientation maps $$\Sigma^{-2,-1} \Sigma_+^\infty \P_S^\infty \to E$$
and $$\Sigma^{-1,-1} \Sigma_+^\infty W_{S,n} \to E.$$
\end{remark}

We specialize now to the case $E=F=\MFl$.

\begin{corollary}
The coaction map
$$H^{**}(W_{S,l},\bF_l) \to (\MFl \wedge \MFl)_{-*,-*}
\hat{\otimes}_{(\MFl)_{-*,-*}} H^{**}(W_{S,l},\bF_l)$$
is given by
$$u \mapsto u + \sum_{i \ge 0} \tau_i \otimes v^{l^i},$$
$$v \mapsto v + \sum_{i \ge 1} \xi_i \otimes v^{l^i}$$
with $\tau_i \in (\MFl \wedge \MFl)_{2l^i-1,l^i-1}$ and
$\xi_i \in (\MFl \wedge \MFl)_{2(l^i-1),l^i-1}$.
\end{corollary}

\begin{proof}
This follows from Propositions \ref{g5r5zu34} and \ref{fhj643}.
\end{proof}

Set $\caA_{**}:= (\MFl \wedge \MFl)_{**}$.
We denote by $B$ the set of sequences $(\epsilon_0,r_1,\epsilon_1,r_2,\ldots)$ with
$\epsilon_i \in \{0,1\}$ and $r_i \ge 0$ with only finitely many non-zero terms.
For any $I \in B$ let $$\omega(I):=\tau_0^{\epsilon_0} \xi_1^{r_1} \tau_1^{\epsilon_1} \xi_2^{r_2} \cdots
\in \caA_{p(I),q(I)}.$$

\begin{theorem}
\label{h545egf}
Suppose $l$ is invertible on $S$. Then the map
$$\bigoplus_{I \in B} \Sigma^{p(I),q(I)} \MFl \to \MFl \wedge \MFl,$$
where the map on the summand indexed by $I$ is the $\MFl$-module map
(where we use the right module structure on the target) corresponding to the element
$\omega(I)$, is an isomorphism.
\end{theorem}

\begin{proof}
It is sufficient to show the statement for $S=\Spec(\integers[\frac{1}{l}])$.
This follows from \cite[Theorem 1.1]{hoyois-kelly-oestvaer} using Theorem \ref{grerhrh3}
and Lemma \ref{vdgther}.
\end{proof}

\begin{remark}
In the situation of the theorem the pair $(H^{-*,-*}(S,\bF_l),\caA_{**})$ has the structure
of a Hopf algebroid. The operations of $\MFl$, i.e. $\Hom^{**}(\MFl,\MFl)$, are the dual of
$\caA_{**}$.
\end{remark}

\appendix

\section{(Semi) model structures}
\label{ht34efw}

\begin{proposition}
\label{gr34zww}
Let $\caC$ be a symmetric monoidal cofibrantly generated model category and $I$ an (essentially) small
category with $2$-fold coproducts. Then the projective model structure on $\caC^I$ is symmetric monoidal.
If $I$ has an initial object and the tensor unit in $\caC$ is cofibrant then the tensor unit
in $\caC^I$ is also cofibrant.
\end{proposition}

\begin{proof}
The assertions follow from the formula $$(\Hom(i,\_) \cdot f) \Box (\Hom(j,\_) \cdot g)
\cong \Hom(i \sqcup j,\_) \cdot (f \Box g)$$ for maps $f$ and $g$ in $\caC$ and objectwise
considerations.
\end{proof}

\begin{proposition}
Let $\caC$ be a left proper combinatorial model category and $\caS$ be an (essentially) small site.
Then the projective model structure on $\caC^{\caS^\op}$ can be localized to a local projective model
structure where the local objects are presheaves satisfying descent for all hypercovers of $\caS$.
\end{proposition}

\begin{proof}
We localize at the set of maps
$$\hocolim_{n \in \bigtriangleup^\op} (U_n \times QA) \to QA,$$
where $U_\bullet \to X$ runs through a set of dense hypercovers (see \cite{dugger-hollander-isaksen}) of $\caS$
and $A$ through the set of domains and codomains of a set of generating cofibrations of $\caC$
($QA$ denotes a cofibrant replacement of $A$).
\end{proof}

\begin{proposition}
Let $R$ be a commutative ring and $\caC=\Cpx_{(\ge 0)}(R)$ be the category of (non-negative) chain complexes
of $R$-modules equipped with its standard projective model structure. Let $\caS$ be an (essentially) small site
with $2$-fold products and enough points.
Then the local projective model structure on $\caC^{\caS^\op}$ is symmetric monoidal.
\end{proposition}

\begin{proof}
The projective model structure is symmetric monoidal by Proposition \ref{gr34zww}.
It remains to see that the pushout product of a generating cofibration with a trivial cofibration
is a weak equivalence.
Checking this on stalks does the job (here use the injective model
structure on $\caC$).
\end{proof}

\begin{remark}
This result is also contained in \cite{fausk}.
\end{remark}

\begin{theorem}
Let $R$ be a commutative ring and $\caS$ an (essentially) small site
with $2$-fold products and enough points. Then $\Cpx_{(\ge 0)}(\Sh(\caS,R))$ carries a
local projective symmetric monoidal cofibrantly generated model structure transferred from the local model structure
on presheaves. The weak equivalences are the quasi isomorphisms.
\end{theorem}

\begin{proof}
One applies the transfer principle (see e.g. \cite[\S 2.5]{moerdijk-berger}): One has to check that
transfinite compositions of pushouts by images of generating trivial cofibrations are weak equivalences.
This follows since the sheafification functor
preserves all weak equivalences. The same applies to prove that the model structure is
symmetric monoidal.
\end{proof}

Let $R$ and $\caS$ be as in the Theorem above. Then the canonical generating cofibrations
of $\Cpx_{(\ge 0)}(\Sh(\caS,R))$ have cofibrant domain. Thus for a cofibrant $T \in \Cpx_{(\ge 0)}(\Sh(\caS,R))$
by \cite[Theorem 8.11]{hovey.spectra} there is a stable symmetric monoidal model structure on the category $\Sp_T^\Sigma$
of symmetric $T$-spectra in $\Cpx_{(\ge 0)}(\Sh(\caS,R))$.

It follows from \cite[Theorem 4.7]{spitzweck.thesis} that for a $\Sigma$-cofibrant operad $\caO$ in $\Sp_T^\Sigma$
the category of $\caO$-algebras inherits a semi model structure. In particular for the image of the
linear isometries operad in $\Sp_T^\Sigma$ we obtain a semi model category $E_\infty(\Sp_T^\Sigma)$
of $E_\infty$-spectra.

\section{Pullback of cycles}
\label{ht635zh}

For a regular separated Noetherian scheme $X$ of finite Krull dimension we let
$X^{(p)}$ be the set of codimension $p$ points on $X$ and $Z^p(X)$ the free
abelian group on $X^{(p)}$. Let $C \in X^{(p)}$, $D \in X^{(q)}$. We say
that $C$ and $D$ intersect properly if the scheme theoretic intersection $Z$
of the closures of $C$ and $D$ in $X$ has codimension everywhere $\ge p+q$.
If $C$ and $D$ intersect properly then
for a point $W$ of $Z$ of codimension $p+q$ in $X$ we set

$$m(W;C,D):= \sum_{i \ge 0} (-1)^i \mathrm{length}_{\caO_{X,W}}
(\Tor_i^{\caO_{X,W}}(\caO_{\overline{C},W}, \caO_{\overline{D},W})),$$
known as Serre's intersection multiplicity.

We extend the notion of proper intersection and the intersection
multiplicity at an arbitrary $W \in X^{(p+q)}$ in the canonical way to elements of $Z^p(X)$ and $Z^q(X)$.

For $C \in Z^p(X)$ and $D \in Z^q(X)$ which intersect properly we let
$$C \cdot D:=\sum_{W \in X^{(p+q)}} m(W;C,D) \cdot W.$$

For a coherent sheaf $\caF$ on $X$ whose support has everywhere codimension
$\ge p$ we let $Z_p(\caF) \in Z^p(X)$ be given by
$$Z_p(\caF):=\sum_{W \in X^{(p)}} \mathrm{length}_{\caO_{X,W}}(\caF_W) \cdot W.$$

\begin{proposition}
\label{grfe4rtz4}
Let $\caF$ and $\caG$ be coherent sheaves on $X$. Suppose that the supports of $\caF$,
$\caG$ and $\caF \otimes_{\caO_X} \caG$ have everywhere at least codimension $p$, $q$
and $p+q$ respectively. Then
$$Z_p(\caF) \cdot Z_q(\caG) = \sum_{i \ge 0} (-1)^i Z_{p+q}(\Tor_i^{\caO_X}(\caF,\caG)).$$
\end{proposition}

Of course this Proposition is a special case of a statement valid for perfect complexes on $X$.

\begin{proof}
Since the question is local on $X$ we can assume $X$ is local and the support of
$\caF \otimes_{\caO_X} \caG$ is the closed point of $X$. Then the proof proceeds
as the proof of \cite[V.C. Proposition 1]{serre.local-algebra}, using \cite[Theorem 1]{roberts.vanishing}
or \cite{gillet-soule.null} and a filtration argument (using e.g. \cite[Proposition 3.7]{eisenbud.comm-alg}).
\end{proof}

\begin{proposition}
\label{ht34t4e}
Let $C \in Z^p(X)$, $D \in Z^q(X)$ and $E \in Z^r(X)$ such that $C \cdot D$, $(C \cdot D) \cdot E$
and $D \cdot E$ are well defined. Then we have
$$(C \cdot D) \cdot E = C \cdot (D \cdot E)$$
in $Z^{p+q+r}(X)$.
\end{proposition}

\begin{proof}
The proof proceeds as the proof of \cite[V.C.3.b) Associativity]{serre.local-algebra}, using a spectral
sequence argument and Proposition \ref{grfe4rtz4}.
\end{proof}

For a flat map $X \to Y$ between regular separated Noetherian schemes of finite Krull dimension
there is a flat pullback $f^* \colon Z^p(Y) \to Z^p(X)$.

Let now $S$ be a regular separated Noetherian scheme of finite Krull dimension.
Let $f \colon X  \to Y$ be a morphism in $\Sm_S$ and $C \in Z^p(Y)$. We say that $f$
and $C$ are in good position if for every $W \in X^{(p)}$ with a non-zero coefficient in $C$
the scheme theoretic inverse image $f^{-1}(\overline{W})$ has everywhere codimension $\ge p$.
If this is the case we define $f^*(C) \in Z^p(X)$
by
$$f^*(C):=\Gamma_f \cdot \pr_Y^*(C).$$
(We view this intersection, which takes place on $X \times_SY$, in a canonical way as an element
of $Z^p(X)$. Note also that the graph is not in general of a well defined codimension, but for the
definition we can e.g. assume $X$ and $Y$ to be connected.)

\begin{theorem}
\label{h5t423t}
Let $X \overset{f}{\longrightarrow} Y \overset{g}{\longrightarrow} Z$
be maps in $\Sm_S$. Let $C \in Z^p(Z)$ and assume $g$ and $C$ are in good position and
$f$ and $g^*(C)$ are in good position. Then $g \circ f$ and $C$ are in good position
and $$(g \circ f)^*(C)= f^*(g^*(C))$$
in $Z^p(X)$.
\end{theorem}

\begin{proof}
Let $U := \Gamma_f \times_S Z \subset X \times_S Y \times_S Z$ and
$V := X \times_S \Gamma_g \subset X \times_S Y \times_S Z$. Let
$\pr_Z \colon X \times_S Y \times_S Z \to Z$ be the projection.
The assertion follows from the associativity
$$(U \cdot V) \cdot \pr_Z^*(C) = U \cdot (V \cdot \pr_Z^*(C))$$
which holds by Proposition \ref{ht34t4e}.
\end{proof}


\bibliographystyle{plain}
\bibliography{em}

\vspace{0.1in}

\begin{center}
Fakult{\"a}t f{\"u}r Mathematik, Universit{\"a}t Osnabr\"uck, Germany.\\
e-mail: markus.spitzweck@uni-osnabrueck.de
\end{center}

\end{document}